\documentclass{amsart}
\usepackage[dvipsnames]{xcolor}
\pdfoutput=1 

\usepackage{amssymb, amsmath}
\usepackage{tikz}
\usepackage{tikz-cd}
\usetikzlibrary{arrows}
\usetikzlibrary{fit}
\usepackage{bbm} 

\usepackage{tensor}

\usepackage{dynkin-diagrams}

\usepackage{float}

\usepackage{mathrsfs} 

\RequirePackage{color}
\definecolor{myred}{rgb}{0.75,0,0}
\definecolor{mygreen}{rgb}{0,0.5,0}
\definecolor{myblue}{rgb}{0,0,0.65}

\RequirePackage{ifpdf}
\ifpdf
  \IfFileExists{pdfsync.sty}{\RequirePackage{pdfsync}}{}
  \RequirePackage[pdftex,
   colorlinks = true,
   urlcolor = myblue, 
   citecolor = mygreen, 
   linkcolor = myred, 
   bookmarksopen=true]{hyperref}
\else
  \RequirePackage[hypertex]{hyperref}
\fi

\RequirePackage{ae, aecompl, aeguill} 

\DeclareFontFamily{U}{mathx}{}
\DeclareFontShape{U}{mathx}{m}{n}{<-> mathx10}{}
\DeclareSymbolFont{mathx}{U}{mathx}{m}{n}
\DeclareMathAccent{\widehat}{0}{mathx}{"70}
\DeclareMathAccent{\widecheck}{0}{mathx}{"71}

\newcommand{\into}{\hookrightarrow}

\newcommand{\co}{\colon}
\newcommand{\ot}{\otimes}

\newcommand{\namedto}[1]{\stackrel{#1}{\longrightarrow}}

\newtheorem{thm}{Theorem}[section]
\newtheorem{lem}[thm]{Lemma}
\newtheorem{prop}[thm]{Proposition}
\newtheorem{cor}[thm]{Corollary}

\newtheorem{thmx}{Theorem}

\theoremstyle{definition}
\newtheorem{defn}[thm]{Definition}
\newtheorem{notation}[thm]{Notation}
\newtheorem{ex}[thm]{Example}
\newtheorem{convention}[thm]{Convention}

\theoremstyle{remark}
\newtheorem{rem}[thm]{Remark}

\def\ZZ{\mathbb{Z}}

\def\CC{\mathbb{C}}
\def\BB{\mathbb{B}}

\def\RR{\mathbb{R}}
\def\one{\mathbbm{1}}
\def\WW{\mathbb{W}}

\def\FC{\mathcal{C}}
\def\MC{\mathcal{M}}
\def\AC{\mathcal{A}}

\DeclareMathOperator{\Obj}{Obj}
\DeclareMathOperator{\Hom}{Hom}
\DeclareMathOperator{\End}{End}
\DeclareMathOperator{\Rep}{Rep}
\DeclareMathOperator{\Vect}{Vect}
\DeclareMathOperator{\Ker}{Ker}
\DeclareMathOperator{\Coker}{Coker}
\DeclareMathOperator{\id}{id}
\DeclareMathOperator{\Id}{Id}

\DeclareMathOperator{\Irr}{Irr}
\DeclareMathOperator{\ufld}{ufld}
\DeclareMathOperator{\opp}{opp}
\DeclareMathOperator{\incl}{incl}
\DeclareMathOperator{\Eval}{eval}
\DeclareMathOperator{\FPdim}{FPdim}
\DeclareMathOperator{\Path}{Path}

\DeclareMathOperator{\GL}{GL}

\DeclareMathOperator{\TLJ}{TLJ}
\DeclareMathOperator{\Ind}{Ind}
\DeclareMathOperator{\Res}{Res}
\DeclareMathOperator{\Forget}{Forget}

\newcommand{\bluebull}{{\color{blue} \bullet}}
\newcommand{\bluea}{{\color{blue} a}}
\newcommand{\redbull}{{\color{red} \bullet}}
\newcommand{\redb}{{\color{red} b}}
\newcommand{\udim}{\underline{\dim}}

\mathchardef\mhyphen="2D 
\newcommand{\lmod}{\mhyphen \mathsf{mod}}

\newcommand{\mattwo}[1]{{\left[ \begin{array}{cc} #1 \end{array} \right]}}

\newcommand{\leftdual}[1]{{}^*{#1}}
\newcommand{\rightdual}[1]{{#1}^*}

\newcommand{\Sign}{\mathbbm{S}}

\usepackage[maxbibnames=99,
			maxalphanames=4, 
			style=alphabetic]{biblatex}
\AtEveryBibitem{\clearlist{language}} 
\usepackage{xurl}

\title[]{Classification of finite type fusion quivers}


\author[]{Ben Elias}
\address{University of Oregon.}
\email{belias@uoregon.edu}

\author[]{Edmund Heng}
\address{Institut des Hautes \'Etudes Scientifiques (IHES).}
\email{heng@ihes.fr}

\begin{document}

\begin{abstract} In recent work, the second author introduced the concept of Coxeter quivers, generalizing several previous notions of a quiver representation. Finite type Coxeter quivers were classified, and their indecomposable objects were shown to be in bijection with positive roots, generalizing a classical theorem of Gabriel. In this paper we define fusion quivers, a natural generalization of Coxeter quivers. We classify the finite type fusion quivers, and prove the analogue of Gabriel's theorem.
As a special case, this proves a generalised quantum McKay correspondence for fusion categories, an analogue of Auslander--Reiten's result for finite groups in the fusion categorical setting.
\end{abstract}

\maketitle

\setcounter{tocdepth}{1}
\tableofcontents

\section{Introduction} \label{sec:intro}

\subsection{Fusion quivers}

A well-known theorem in the theory of quiver representations is Gabriel's classification of finite type quivers (i.e.\ quivers with finitely many indecomposable representations), which
falls into an $ADE$ classification \cite{Gabriel_finitetype}. In \cite{heng2024coxeter}, a generalization of quivers known as \emph{Coxeter quivers} were introduced. The main theorem in loc.\ cit.\ is that the
classification of finite type Coxeter quivers extends the $ADE$ classification to include all Coxeter--Dynkin types, even the non-crystallographic types $H$ and $I$. Furthermore, the
bijection between indecomposable representations and (extended) positive roots extends to all finite Coxeter groups in this setting.

A representation of a quiver assigns a vector space to each vertex; when interpreted in the language of \cite{heng2024coxeter}, one also attaches a one-dimensional vector space to each
edge of the quiver itself. We think of this theory as being valued in the monoidal category of vector spaces. A fusion category is a semisimple monoidal category with finitely many
isomorphism classes of simple objects, with some additional properties we do not discuss now. Coxeter quivers instead take values in particular fusion categories associated to quantum
groups $U_q(\mathfrak{sl}_2)$, and tensor products thereof. 
Ultimately this is an ad hoc construction, tailored for Coxeter theory; there is one Coxeter quiver for each Coxeter graph.

In this paper, we define what we feel is the most natural and appropriate generalization of a quiver representation, allowing edge labels in \emph{any} fusion category $\FC$, and allowing representations to take values in \emph{any finite semisimple module category} $\MC$ over $\FC$. More concretely, we introduce the notion of \emph{fusion quivers} $Q$ over $\FC$, which are quivers whose arrows  $e$ are labelled by objects $\Pi_e$ in $\FC$.
A representation $X$ of a fusion quiver $Q$ valued in $\MC$ is an assignment of objects $X_v$ in $\MC$ for each vertex $v$, and an assignment of morphism $\phi_e: \Pi_e \otimes X_{s(e)} \rightarrow X_{t(e)}$ for each arrow $e$ with source $s(e)$ and target $t(e)$. Such representations form an abelian category $\Rep_\MC(Q)$.

Many constructions for classical quiver representations have analogues for fusion quiver representations. For example, we construct reflection functors in Definitions \ref{defn:sinkreflection} and \ref{defn:sourcereflection}. Moreover, many previous generalizations of quiver representations also fall into this framework; see \S \ref{sec:previous} for a more detailed discussion.



\subsection{Classification result}

We ask anew the question: for which fusion quivers $Q$ and which $\MC$ does the category $\Rep_{\MC}(Q)$ have finitely many indecomposable objects (up to isomorphism)? 
Our definition of fusion quivers is now unlinked from Coxeter theory, so this a general question about fusion categories and combinatorics. Delightfully and perhaps expectedly, the answer comes from
Coxeter theory. While the representation-type of classical quivers depends on the underlying graph, our main theorem states that the representation-type of a fusion quiver $Q$ will
depend on its \emph{associated Coxeter graph} $\Gamma_Q$, whose construction we describe below. A consequence of our result is that whether $\Rep_{\MC}(Q)$ has finitely many indecomposable objects is
independent of the choice of $\MC$, since the graph $\Gamma_Q$ depends solely on $Q$ and not on $\MC$.

Every object $L$ in a fusion category has an associated notion of dimension called its \emph{Frobenius--Perron dimension}, which we denote by $\FPdim(L) \in \RR_{\geq 0}$ (see \S\ref{ss:FPdimsigncoh}
for a precise definition). The Frobenius--Perron dimension is uniquely determined by the following properties: $\FPdim(L)$ for any object $L$ is a non-negative real number, and $\FPdim$
respects tensor products and direct sums. In particular, when $\FC = \Vect$ is the category of vector spaces (or any fusion category with a fibre functor to $\Vect$), $\FPdim$ is given
by the usual notion of dimension. However, $\FPdim(L)$ need not be an integer in general. A key property is that when $\FPdim(L) < 2$, there exists (a unique) $2 \leq m < \infty$ such that $\FPdim(L) = 2\cos(\pi/m)$.

Using this key property, we construct $\Gamma_Q$ as follows. Starting with the quiver $Q$, we forget orientation to obtain a graph, and label each edge with a number $m_e$ instead of an object $\Pi_e$, where 
\begin{equation} \begin{cases} \FPdim(\Pi_e) = 2\cos(\pi/m_e) & \text{ if }\FPdim(\Pi_e) < 2 \\ m_e = \infty & \text{ if } \FPdim(\Pi_e) \ge 2. \end{cases} \end{equation}
Our main theorem is a classification of finite type fusion quivers, which generalizes Gabriel's theorem.
\begin{thmx}[= Theorem \ref{thm:CoxeterDynkinifffinite}]
Let $Q$ be a fusion quiver over a fusion category $\FC$, and let $\MC$ be a semisimple module category over $\FC$ with finitely many simple objects.
Then $Q$ has only finitely many isomorphism classes of indecomposable representations valued in $\MC$ if and only if $\Gamma_Q$ is a disjoint union of Coxeter--Dynkin diagrams\footnote{We use the standard convention that edges with label $m_e=2$ are removed, and edges with label $m_e = 3$ are left unlabelled.} in Figure \ref{fig:CoxDyndiagrams}.
\end{thmx}
\noindent In the case where $\MC = \FC$, we also have a root-theoretical description of the indecomposable objects of $\Rep_\FC(Q)$; see Theorem \ref{thm:gabrielroot}.

\begin{figure}[t]
\(A_n = \dynkin [Coxeter]A{} \), \
\(B_n = C_n = \dynkin [Coxeter]B{}\), \
\(D_n = \dynkin [Coxeter]D{} \), \
\(E_6 = \dynkin [Coxeter]E6\), \ 
\newline
\(E_7 = \dynkin [Coxeter]E7\), \
\(E_8 = \dynkin [Coxeter]E8 \), \
\(F_4 = \dynkin [Coxeter]F4\), \
\(G_2 = \dynkin [Coxeter,gonality=6]G2 \), \ 
\newline
\newline
\(H_3 = \dynkin [Coxeter]H3\), \
\(H_4 = \dynkin [Coxeter]H4\), \
\(I_2(m) = \dynkin [Coxeter,gonality=m]I{}\). 
\caption{The Coxeter--Dynkin diagrams classifying the irreducible finite Coxeter groups \cite{Cox_completeenumeration}.}
\label{fig:CoxDyndiagrams}
\end{figure}

We emphasize here that for a fixed fusion quiver $Q$, its representations valued in different finite semisimple module categories $\MC$ may very well behave differently.
The theorem above shows that being finite type or not is independent of $\MC$; however the precise number of indecomposable representations does depend on $\MC$, see Example \ref{ex:differentM}.

Each of the Coxeter graphs above can be realized as the graph $\Gamma_Q$ associated to one of the Coxeter quivers in \cite{heng2024coxeter}. However, it can also be
realized by many different fusion quivers, with different underlying fusion categories. 
See Example \ref{ex:RepS2} (type $A_2$) and Example \ref{ex:Uqsl3at5} (type $I_2(5)$) for some unusual examples. In particular, there are many fusion quivers
not valued in vector spaces whose underlying Coxeter graph has type $A$.

\subsection{Proof strategy and unfolding}
To each fusion quiver $Q$ and a choice of finite semisimple module category $\MC$, we associate an ``unfolded'' classical quiver $\check{Q}_{\MC}$ such that $\Rep_\MC(Q)$ is equivalent (as an abelian category) to the category $\Rep(\check{Q}_\MC)$ of classical representations of $\check{Q}_\MC$.
As such, understanding the number of indecomposables in $\Rep_\MC(Q)$ amounts to understanding the number of indecomposables in $\Rep(\check{Q}_\MC)$, which can be done using Gabriel's theorem \cite{Gabriel_finitetype}. Nonetheless, the classical quiver $\check{Q}_\MC$ still \emph{depends on} $\MC$.

To understand $\check{Q}_\MC$ in general, we first focus on the rank two fusion quivers, which are fusion quivers $Q$ of the form
\[
Q := \quad \begin{tikzcd}
	\bullet & \bullet
	\arrow["\Pi", from=1-1, to=1-2]
	\end{tikzcd}.
\]
We prove results relating reflection functors on the categorical level with dihedral reflection groups acting on the Grothendieck group $K_0(\Rep_\MC(Q))$. As a consequence, we deduce that $\Rep_\MC(Q)$ has finitely many indecomposables if and only if $\FPdim(\Pi) = 2\cos(\pi/m) < 2$ for $m \geq 2$. More importantly, we show that in this case $\check{Q}_\MC$ will have underlying graph a disjoint union of $ADE$ Dynkin diagrams with the \emph{same Coxeter number} $m$. In other words, while $\check{Q}_\MC$ can vary as $\MC$ varies, the Coxeter number of this quiver depends only on $Q$.

After proving this classification result for rank two fusion quivers, we bootstrap the result to arbitrary fusion quivers using Coxeter theory.


We also note here that this result for rank two fusion quivers can also be interpreted as a generalized quantum McKay correspondence, which is a quantum version of Auslander--Reiten's generalized McKay correspondence for finite groups; see \S\ref{sec:McKaycorrespondence} for more details.

\begin{rem} 
As we were finishing this paper, we noticed the exciting recent work of Coulembier and Etingof \cite{CEnSphere} on $N$-spherical functors. The connections are significant enough to merit further discussion now.

Roots for dihedral groups can be described as linear combinations of simple roots, whose coefficients are two-colored quantum numbers (see \cite[Appendix]{EDihedral}). Our proofs in
the rank two case involve a computation of iterated reflection functors, relating the result in the Grothendieck group to two-colored quantum numbers. The two different colors of the
quantum number $[2]$ correspond to the symbols\footnote{These symbols need not agree, nor need they even commute, so we must define a
non-commutative version of two-colored quantum numbers, see \S\ref{ss:twocolored}.} $[\Pi]$ and $[\Pi^*]$ in the Grothendieck group $[\FC]$.

Meanwhile, Coulembier and Etingof construct a collection of complexes $\mathbb{E}_n(\Pi)$ built from tensor products of $\Pi$ and $\Pi^*$, based on a construction of
Dyckerhoff-Kapranov-Schechtmann \cite{DKS_Euler}. These complexes (and variants) also categorify two-colored quantum numbers.
They also have a ``classification'' result based on $\FPdim(\Pi)$: that $\mathbb{E}_n(\Pi)$ is zero in the derived category if and only if $\FPdim(\Pi)
= 2 \cos(\pi/m) < 2$ with $m$ dividing $n$. Indeed, the inductive construction of their complexes is completely analogous to the iterative application of reflection functors, and the
cohomology of $\mathbb{E}_n(\Pi)$ agrees with the object labeling one of the vertices after applying $n$ reflection functors to a particular simple object. 
As such, their results in \cite[Chapters 4 and 5]{CEnSphere} and our results in \S\ref{sec:twoone} are closely related and morally similar.
In particular, some of our auxiliary results -- such as our Positivity Lemma
\ref{lem:quantnumpostonon-neg} and Sign Coherence Theorem \ref{thm:signcoherence} -- could be proven in another way using their results (e.g.\ \cite[Remark 4.2.2(2)]{CEnSphere}). 
Our proofs are independent though, and use different techniques. Ultimately, we use similar ideas for very different purposes.
 
We also note that their classification result holds in vastly more generality than we consider (i.e.\ for all quasi-finite tensor categories, not just fusion categories).
\end{rem}

\subsection{Relation to previous work} \label{sec:previous}

We feel as though the definition of fusion quivers is a natural one, which encompasses and extends many existing generalizations of quiver representations. We briefly discuss some of these connections here, a subset of which are discussed also in the body of the paper.
\begin{itemize}
\item \emph{Classical quiver representations:} Classical quiver representations are recovered from representations of fusion quivers with $\FC = \MC = \Vect$. There are some subtleties involved in the choice of bases, see \S\ref{sec:fusionquivervect} for more details.
\item \emph{Coxeter quiver representations:} Coxeter quiver representations are recovered from representations of fusion quivers with $\FC = \MC$ set to (tensor products of) fusion categories associated to $U_q(\mathfrak{sl}_2)$, see Example \ref{ex:coxeterquiver}. In particular, our Theorem \ref{thm:CoxeterDynkinifffinite} recovers \cite[Theorem A]{heng2024coxeter} under this choice. Similarly, our Theorem \ref{thm:gabrielroot} also recovers \cite[Theorem B]{heng2024coxeter}.
\item \emph{Lusztig's canonical bases:} Representations of (rank two) fusion quivers over $\Rep(G)$ (the category $G$-representations) for $G$ a finite subgroup  of $SU(2)$ were studied by Lusztig in his construction of canonical bases for affine type quantum groups; see Example \ref{ex:Lusztig}.
\item \emph{Skew group algebras (of quivers):} Given a classical quiver with a finite group $G$ acting on its path algebra $A$ (fixing idempotents), modules over the skew group algebra $G\# A$ are the same as representations of some fusion quiver over $\Rep(G)$; see \S \ref{sec:Gequivariant} and Remark \ref{rmk:skewgroupalg}.
\item \emph{Modules over tensor algebras:} In \cite{EKW_tensoralgebras}, Etingof--Kinser--Walton study (and classify) tensor algebras in finite tensor categories. Our work can be related to theirs in the setting of fusion categories (i.e.\ semisimple finite tensor categories) as follows.
The representation category $\Rep_\FC(Q)$ is equivalent to the category of modules over an appropriately defined path algebra associated to $Q$, which is a $\FC$-tensor algebra in the sense of Etingof--Kinser--Walton (see Remark \ref{rmk:pathalgebra} for more details).
The $\FC$-tensor algebras associated to fusion quivers have base algebra $S$ given by a direct sum of the trivial algebra $\one$; the $\FC$-tensor algebras with more interesting semisimple base algebras would instead be tensor algebras associated to $\FC$-species, which we discuss next.
\item \emph{Representations of $\Bbbk$-species:} A classical generalization of quiver representations is given by the representations of $\Bbbk$-species. Albeit similar, this is slightly different from our consideration of fusion quivers; in fact, its generalization -- called a \emph{$\FC$-species} ($\FC$ is again a fusion category) -- was recently considered by Qiu and Zhang in \cite{QZ_fusionstable}.
Namely, a $\FC$-species is a quiver with each vertex $v$ labelled by a semisimple algebras $A_v$ (i.e. $\MC_v := A_v\lmod$ is semisimple), and each arrow from $v$ to $w$ labelled by an $(A_w,A_v)$-bimodule ${}_wB_v$.
A representation of a $\FC$-species is then the datum of $X_v \in \MC_v$ for each vertex $v$, and a morphism of $A_w$-modules
\[
\phi_e: {}_wB_v \otimes_{A_v} X_v \rightarrow X_w
\]
for each arrow $e: v \rightarrow w$.
We remark here that the classical notion of $\Bbbk$-species representations is recovered by taking $\FC = \Vect_\Bbbk$, with $A_v$ a division algebra (which is semisimple) over $\Bbbk$.

Even if we fix $A:= A_v$ to be the same for all $v$, so that $X_v$ are all living in the same finite semisimple module category $\MC:= \MC_v = A\lmod$, the objects ${}_wB_v$ live in the category of $(A,A)$-bimodules, which is the (monoidal opposite) dual fusion category $\FC^*_\MC$ of $\FC$ over $\MC$ -- this need not be the same as $\FC$ itself, which is what we used for fusion quivers.
In fact, representations of a $\FC$-species form a (right) module category over $\FC$, whilst representations of a fusion quiver valued in $\MC$ form a (right) module category over the dual fusion category ${\FC^*_\MC}^{\otimes^{op}}$; see Remark \ref{rmk:repismodulecat}.
\end{itemize}


\subsection*{Acknowledgments} 
The authors would like to thank the organizers of the conference ``Categorification in Representation Theory'' (Sydney, 2023) where this collaboration began.
We would like to thank Victor Ostrik for helping us find the example from Remark \ref{rmk:theydontcommute}, by pointing us to a very helpful mathoverflow post by Noah Snyder. We would also like to thank Victor Ostrik for the suggestion to let our quiver representations be valued in $\MC$ rather than just $\FC$, and for many other helpful discussions. 
We thank Pavel Etingof for pointing out the related paper \cite{EKW_tensoralgebras}.
We also thank Clement Delcamp, Geoffrey Janssens, Sondre Kvamme, Greg Kuperberg and Abel Lacabanne for many useful discussion on related topics.
We would also like to thank the anonymous referee for a very thorough and helpful report.
B.E. was supported by the US National Science Foundation grant [DMS-2201387], and his research group was supported by [DMS-2039316].

\section{Definitions and basic results} \label{sec:basics}

Throughout this paper, $\Bbbk$ denotes an algebraically closed field and all categories will be $\Bbbk$-linear.
By a \emph{finite abelian category}, we mean an abelian category that is equivalent to the category of modules over a finite dimensional $\Bbbk$-algebra; more intrinsically, every object has finite-length, the category is Hom-finite, has enough projectives and has finitely many isomorphism classes of simple objects (see \cite[Definition 1.8.5 and 1.8.6]{EGNO}).

\subsection{Fusion quivers} \label{sec:fusionquiverdef}

\begin{defn}\label{defn:tensorfusioncat}
A (strict) \emph{fusion category} $\FC$ is a finite abelian category with a (strict) monoidal structure $(- \otimes -, \one)$ satisfying the following properties:
\begin{enumerate}
\item the monoidal unit $\one$ is simple;
\item every object $C \in \FC$ has a left dual $\leftdual{C}$ and a right dual $\rightdual{C}$ with respect to the tensor product $\otimes$ (rigidity); and
\item $\FC$ is semisimple.
\end{enumerate}
A (finite) set of representatives of isomorphism classes of simple objects in $\FC$ will always be denoted by $\Irr(\FC)$.
\end{defn}

\begin{defn} \label{defn:modcat}
A left (resp.\ right) \emph{module category} over a fusion category $\FC$ is an abelian category $\MC$ together with an additive, monoidal functor 
\[
\Psi: \FC \text{ (resp.\ }  \FC^{\otimes^\text{op}}) \rightarrow \End(\MC)
\]
into the additive, monoidal category of additive endofunctors $\End(\MC)$.
We will denote the endofunctor $\Psi(C)$ simply by $C\otimes -$.
\end{defn}
Note that the endofunctors $C \otimes -$ are automatically exact since every object $C \in \FC$ has left and right duals.
Throughout this paper, we will suppress associators and unitors in our proofs (i.e.\ strictifying) by appealing to  Mac Lane's theorem and its module category analogue \cite[Theorem 2.8.5 and Remark 7.2.4]{EGNO}.

\begin{defn} A (finite) \emph{quiver} $Q$ is the data of a finite vertex set $V$, and a finite edge set $E$, together with functions $s, t : E \to V$ which pick out the source and target of each edge. We abbreviate this data as $Q = (V,E)$. 
\end{defn}

\begin{defn} Let $\FC$ be a fusion category over a base field $\Bbbk$. A \emph{fusion quiver} over $\FC$ is a quiver $Q = (V,E)$ together with an assignment $\Pi : E \to \Obj(\FC)$ of an object $\Pi_e$ of $\FC$ to each edge $e \in E$. We abbreviate this data as $Q = (\FC,V,E,\Pi)$, and instead refer to $(V,E)$ as the \emph{underlying quiver}. \end{defn}

\begin{defn} Let $Q = (\FC, V, E, \Pi)$ be a fusion quiver, and let $\MC$ be a finite, semisimple left module category over $\FC$. A \emph{(fusion) representation} of $Q$ in $\MC$ is
	\begin{itemize}
		\item an assignment $X : V \to \Obj(\MC)$ of an object $X_v$ of $\MC$ to each vertex $v \in V$, 
		\item for each $e \in E$, a morphism $\phi_e \in \Hom_{\MC}(\Pi_e \ot X_{s(e)}, X_{t(e)})$.
	\end{itemize}
	The data of a representation is often abbreviated $(X,\phi)$. A \emph{morphism of (fusion) representations} $\rho : (X,\phi) \to (Y,\psi)$ is the data of
	\begin{itemize}
		\item for each $v \in V$, a morphism $\rho_v \in \Hom_{\MC}(X_v, Y_v)$,
	\end{itemize}
	for which the following square commutes for each edge $e$:
\begin{equation} \label{eq:morphismdef}
	\begin{tikzcd}
	\Pi_e \ot X_{s(e)} \arrow[r,"\phi_e"] \arrow[d,"\id_{\Pi_e} \ot \rho_{s(e)}"'] & X_{t(e)} \arrow[d,"\rho_{t(e)}"'] \\
	\Pi_e \ot Y_{s(e)} \arrow[r,"\psi_e"] & Y_{t(e)}
	\end{tikzcd}
\end{equation}
The category of fusion representations valued in $\MC$ is denoted $\Rep_{\MC}Q$.
\end{defn}

\begin{rem} 
Many parts of this paper generalize to the case where $\FC$ is a (finite) tensor category, and/or $\MC$ a (finite) left module category which is not necessarily semisimple.
One could even consider relaxing the requirement of coherence datum on $\MC$ to allow for abelian categories $\MC$ that are simply equipped with the action of the monoid of objects of $\FC$.
We leave these potential generalizations to the interested reader. 
\end{rem}

\begin{prop} The category $\Rep_{\MC}Q$ is an abelian category under vertex-wise kernels and cokernels. \end{prop}

\begin{proof} 
The proof is the same as for ordinary quiver representations, mutatis mutandis. Note that $\MC$ is itself an abelian category, so that vertex-wise kernels, cokernels,
and direct sums exist. It is easy to verify that vertex-wise kernels, cokernels and direct sums naturally inherit the structure of a fusion representations, after which point it is
straightforward to verify their universal properties (etcetera) in $\Rep_{\MC}Q$. 
\end{proof}

\begin{rem} \label{rmk:repismodulecat}
With $\FC^*_\MC$ denoting the \emph{dual tensor category} of $\FC$ over $\MC$ (see \cite[\S 7.12]{EGNO} for the definition), $\Rep_\MC(Q)$ carries a left $\FC^*_\MC$-module action; we will not need this structure, so we shall leave the details of this action to the reader. 
What we will need in \S \ref{sec:M=C} is this structure in the particular case where $\MC = \FC$.
In this case, $\FC^*_\FC = \FC^{\otimes^{op}}$ and hence $\Rep_\FC(Q)$ is a \emph{right} module category over $\FC$ via tensoring on the right.
This right module category structure can be described more directly as follows: $Y \in \FC$ acts on the objects of $\Rep_\FC(Q)$ -- on the right -- by tensoring all objects on the vertices of $Q$ with $Y$ and all edge maps with $\id_Y$; the action on morphisms follows similarly.
However, we emphasize that there is no natural left action of $\FC$ on $\Rep_\FC(Q)$ nor $\Rep_\MC(Q)$.
\end{rem}

\begin{rem} As we define functors between categories of fusion representations, we will often be defining new fusion representations from old ones in abstract ways. Remember that
there are no constraints in the definition of a fusion representation. One need only confirm that the maps $\phi_e$ have the correct source and target for a fusion representation to
be well-defined. There is a constraint to check for defining morphisms of fusion representations, namely \eqref{eq:morphismdef}. \end{rem}

Here are two obvious lemmas.

\begin{lem}\label{lem:removezero} Let $Q$ be a fusion quiver and $e$ an edge for which $\Pi_e$ is the zero object. Let $Q'$ be the fusion quiver obtained from $Q$ by removing $e$. Then, for any module category $\MC$, the obvious functor gives an equivalence of categories $\Rep_{\MC} Q \cong \Rep_{\MC} Q'$.
\end{lem}

\begin{lem} \label{lem:directsum} Consider a fusion quiver $Q$ with two distinct edges $e_1, e_2$ having the same source $s$ and target $t$, and labelled by $\Pi_1$ and $\Pi_2$ respectively. Let $Q'$ be the fusion quiver obtained from $Q$ by replacing $e_1$ and $e_2$ with a single edge $e'$ having the same source $s$ and target $t$, labelled by $\Pi_1 \oplus \Pi_2$. Then for any module category $\MC$, the obvious functor gives an
equivalence of categories $\Rep_{\MC} Q \cong \Rep_{\MC} Q'$. \end{lem}

\begin{rem} In this lemma, $\Pi \oplus \Pi'$ is implicitly equipped with fixed projection and inclusion maps to and from its summands $\Pi$ and $\Pi'$. Some related subtleties are discussed in \S\ref{sec:fusionquivervect}. \end{rem}

\begin{defn} 
We say that a fusion quiver $Q$ has \emph{decomposable edges} if any edge is labelled by a decomposable object, or if there are two distinct edges $e, e'$ having the same source $s$ and target $t$ which are labelled by nonzero objects. 
\end{defn}

We conclude with the obvious notion of an isomorphism of fusion quivers. We shall tacitly use various obvious relations between isomorphic fusion quivers, e.g. isomorphic fusion quivers have isomorphic categories of representations.

\begin{defn} Let $Q = (\FC, V, E, \Pi)$ and $Q' = (\FC, V', E', \Pi')$ be two fusion quivers over $\FC$. An \emph{isomorphism} of fusion quivers over $\FC$ is the data of
	\begin{itemize} \item an isomorphism $f : (V,E) \to (V',E')$ of the underlying quivers,
		\item an isomorphism $\Pi_e \to \Pi'_{f(e)}$ in $\FC$ for each edge $e \in E$. \end{itemize}
\end{defn} 


\subsection{Fusion quivers valued in vector spaces} \label{sec:fusionquivervect}

In this section (and this section only) we wish to discuss the relation of fusion quiver representations to the classical notion of (ordinary) quiver representations. Reserving the shorthand $Q$ for fusion quivers, we write $(V,E)$ for an ordinary quiver, and $\Rep (V,E)$ for the usual category of quiver representations over a field $\Bbbk$. 

Let $(V,E)$ be an (ordinary) quiver. Let $\Vect = \Vect_{\Bbbk}$ and label each edge $e$ with the one-dimensional vector space $\Pi_e = \Bbbk$, so that $Q=(\Vect, V, E, \Pi)$ is a fusion
quiver over $\Vect$. Then we have an equivalence (of abelian categories)
\[
\Rep_{\Vect} Q \cong \Rep (V,E) 
\]
induced by the identification $\Bbbk \otimes W \cong W$ for all vector space $W \in \Vect$. 
However, there is a subtlety here which certainly bears mentioning. 
Writing $\Pi_e = \Bbbk$ equips this vector space with an implicit basis, but the concept of a fusion quiver does not entail a choice of basis for edge labels $\Pi_e$ (as in general $\FC$ may not even admit a forgetful functor to vector spaces).
Technically, it is not the basis of $\Bbbk$ which is as important as the choice of isomorphism $\Bbbk \otimes W \cong W$ corresponding to that choice of basis.

Let $\BB = \{b_e\}$ be the data of a basis for each vector space $\Pi_e$; this is just a choice of nonzero
vector $b_e \in \Pi_e = \Bbbk$ for each $e \in E$. This is the data we need to construct a functor $F_{\BB} \co \Rep_{\Vect} Q \to \Rep (V,E)$. If $(X,\phi)$ is a fusion representation, then $F_{\BB}(X,\phi) = (X,\phi')$, an ordinary quiver representation with the same underlying vector spaces $X_v$. We have
\begin{equation} \phi'_e(x) := \phi_e(b_e \ot x). \end{equation}
The inverse functor to $F_{\BB}$ is easy to define.

The subtlety is that $F_{\BB}$ is \emph{not} independent of the choice of bases $\BB$ up to isomorphism! For example, let $Q$ be the (2-)Kronecker quiver, and label one edge with $\Bbbk \cdot v$ and the other with $\Bbbk \cdot w$. Let $\BB = \{v,w\}$ and $\BB' = \{2v,w\}$ be two choices of bases. There is a fusion representation $(X,\phi)$ for which 
\begin{equation} F_{\BB}(X,\phi) = \begin{tikzcd} \Bbbk \arrow[r,bend left,"1"] \arrow[r,bend right,swap,"1"] & \Bbbk \end{tikzcd}, \qquad
F_{\BB'}(X,\phi) = \begin{tikzcd} \Bbbk \arrow[r,bend left,"2"] \arrow[r,bend right,swap,"1"] & \Bbbk \end{tikzcd}.  \end{equation}
These two quiver representations are not isomorphic, so the functors $F_{\BB}$ and $F_{\BB'}$ can not be isomorphic.

Instead, the functor $F_{\BB}$ is independent of the choice of bases $\BB$ up to post-composition with an autoequivalence of $\Rep (V,E)$. There is an autoequivalence $G$ of the
category $\Rep (V,E)$ which rescales the linear transformations on each edge, satisfying $G \circ F_{\BB} = F_{\BB'}$. (Here we truly mean equality of functors.)

We are now equipped to discuss the passage from fusion quivers over $\Vect$ back to ordinary quivers.

\begin{lem} \label{lem:ordinaryquivermultipleedge} Let $Q = (\Vect,V,E,\Pi)$ be a fusion quiver over $\Vect = \Vect_{\Bbbk}$. Let $m_e$ denote the dimension of the vector space $\Pi_e$
for each $e \in E$. Let $(V,E')$ be the ordinary quiver with the same vertex set, where each edge $e \in E$ is replaced with $m_e$ (numbered) edges in $E'$ with the same source and
target. A \emph{basis} of $Q$ is the data $\BB = \sqcup_{e \in E} \BB_e$, where $\BB_e$ is an (ordered) basis of $\Pi_e$. For each basis $\BB$ of $Q$ there is a functor $F_{\BB} \co
\Rep_{\Vect} Q \to \Rep (V,E')$, defined below. 
The functor $F_{\BB}$ is an equivalence of categories. For two bases $\BB$ and $\BB'$, the functors $F_{\BB}$ and $F_{\BB'}$ need not be isomorphic, but there exists an autoequivalence $G$ of
$\Rep (V,E')$ for which $G \circ F_{\BB} = F_{\BB'}$.

If $(X,\phi)$ is a fusion representation of $Q$, then $F_{\BB}(X,\phi) = (X,\phi')$ is an ordinary quiver representation with the same underlying vector spaces $X_v$. If $\{f_1, f_2, \ldots, f_{m_e}\}$ are the edges in $E'$ lying under an edge $e$ in $E$, and $\BB_e = \{b_1, \ldots, b_{m_e}\}$ is the corresponding ordered basis of $\Pi_e$, then for any $x \in X_{s(e)}$ we have
\begin{equation} \phi'_{f_i}(x) := \phi_e(b_i \ot x). \end{equation}
\end{lem}

\begin{proof} This is easy, and left to the reader.\end{proof}

\begin{rem} Let $G$ be an automorphism of $\Rep (V,E')$ as above, and $G(X,\phi) = (X,\psi)$. For $e \in E$, let $e'_1, \ldots, e'_{m_e}$ be the corresponding edges of $E'$. Note that $\psi_{e'_i}$ could be a linear combination of various $\phi_{e'_j}$, corresponding to a general change of basis in $\Pi_e$. \end{rem}

\begin{rem} The previous lemma also makes sense when $m_e = 0$, see Lemma \ref{lem:removezero}. We encourage the reader to compare Lemma \ref{lem:ordinaryquivermultipleedge} with Lemma
\ref{lem:directsum} to see the significance of the choice of inclusion and projection map. \end{rem}

\begin{cor} \label{cor:theotherway} Let $(V,E)$ be an ordinary quiver. Let $Q = (\Vect, V, E, \Pi)$ be the fusion quiver over $\Vect$ obtained by labeling each edge with a one-dimensional vector space $\Pi_e = \Bbbk$, equipped with a standard basis. Then $\Rep_{\Vect} Q \cong \Rep (V,E)$. \end{cor}

\begin{proof} 
This is a special case of Lemma \ref{lem:ordinaryquivermultipleedge}.
\end{proof}

Because of Corollary \ref{cor:theotherway} we opt to work entirely in the context of fusion quivers.

\begin{convention} \label{conv:ordinaryquiver} If we specify an ordinary quiver $(V,E)$, we may abusively call it $Q$ and write $\Rep_{\Vect} Q$. In this case, we refer to the fusion quiver valued in $\Vect$
which was given in Corollary \ref{cor:theotherway}. \end{convention}

\subsection{Further Conventions}

Inspired by Lemmas \ref{lem:removezero} and \ref{lem:ordinaryquivermultipleedge}, we recommend some simplifying conventions for general fusion quivers. Where they conflict with Convention \ref{conv:ordinaryquiver}, the former takes precedent.

\begin{convention} \label{conv:nomult} Unless otherwise stated, we assume that there is at most one edge with the same source and target, though it may be labeled with a
decomposable object. \end{convention}

\begin{convention} \label{conv:nozero} Unless otherwise stated, we assume that every edge in a fusion quiver $Q$ is labeled by a non-zero object. \end{convention}

Convention \ref{conv:nomult} becomes mathematically relevant in \S\ref{sec:sesquilinearformreflection}, where having only one edge makes certain definitions easier to state. 
Convention \ref{conv:nozero} is only used for the sake of simplicity.
For example, with this convention, there is no ambiguity in saying that a fusion quiver is \emph{connected} (resp.\ \emph{acyclic}) when the underlying quiver has that property.
Nonetheless, we point out that this is not a mathematical necessity, and that some statements we make about quivers which have no loops can be extended to quivers which have loops but with the zero label; similarly for cycles.

\subsection{Examples} \label{sec:fusionquiverexamples}

\begin{notation} 
Unless we state otherwise, we will set $\MC = \FC$ in this subsection, i.e.\ all representations of fusion quivers over $\FC$ are valued in $\FC$ itself. 
For simplicity we assume our fusion categories are $\Bbbk$-linear where $\Bbbk$ is a field of characteristic zero.
\end{notation}

\begin{ex}
As discussed in \S\ref{sec:fusionquivervect}, representations of fusion quivers over $\Vect$ are essentially the same as representations of ordinary quivers.
\end{ex}

\begin{ex} \label{ex:RepS2} 
Let $S_2 \cong \ZZ/2\ZZ$ be the symmetric group on 2 elements, and set $\FC_2 := \Rep(S_2)$. 
Let $\Sign$ denote the sign representation, and $\CC$ the trivial representation; these are all of the simple representations of $S_2$.
 
Consider $Q_2$ to be the fusion quiver over $\FC_2$ with two vertices, and one edge labeled by $\Sign$. 
\[
Q_2:=	
	\begin{tikzcd}
	\bullet & \bullet
	\arrow["\Sign", from=1-1, to=1-2]
\end{tikzcd}\]
Let us analyze $\Rep_{\FC_2} Q_2$ in detail. 
The simple representations involve placing either the trivial or the sign representation on one vertex, and zero on the other (so there are a total of 4 simple representations). The two non-simple indecomposable representations are
\[\begin{tikzcd}
	\CC & \Sign
	\arrow["\Sign", from=1-1, to=1-2]
\end{tikzcd}, \qquad \begin{tikzcd}
	\Sign & \CC
	\arrow["\Sign", from=1-1, to=1-2]
\end{tikzcd},\]
where in both representations the arrow is assigned some nonzero morphism in $\FC_2$. In both cases the choice of nonzero morphism is unique up to scalar, and different choices give isomorphic quiver representations.
Meanwhile, for any representation of the form
\[
\begin{tikzcd}
	\Sign & \Sign
	\arrow["\Sign", from=1-1, to=1-2]
\end{tikzcd},
\]
the morphism assigned to the arrow must be zero, since $\Sign \otimes \Sign \cong \CC$ has no non-zero map to $\Sign$ in $\FC_2$. Consequently, the representation is decomposable. The same can be said for
$\begin{tikzcd} 
	\CC & \CC
	\arrow["\Sign", from=1-1, to=1-2]
\end{tikzcd}$.
One can also check that any representation of the form\footnote{We have abused notation here. Unlike the previous quiver representations, the choice of morphism assigned to the arrow here is significant, and two choices might not yield isomorphic representations. However, the choice of morphism is irrelevant to the following discussion.}
\[
\begin{tikzcd}
	A & B
	\arrow["\Sign", from=1-1, to=1-2]
\end{tikzcd}
\]
with $A$ or $B$ decomposable leads to a decomposable representation (the process is similar to the one for the ordinary $A_2$ quiver).
As such, we see that there are effectively twice as many indecomposable fusion representations of $Q_2$ as there are indecomposable representations for the ordinary quiver of type $A_2$. 
In fact, the reader can verify that $\Rep_{\FC_2} Q_2$ is equivalent to $\Rep_{\Vect}(A_2 \sqcup A_2)$. 

Similarly, one can show that the fusion quiver $Q$ over $\FC_2$ given by
\[
\begin{tikzcd}
	\bullet & \bullet & \bullet
	\arrow["\Sign", from=1-1, to=1-2]
	\arrow["\Sign", from=1-2, to=1-3]
\end{tikzcd},
\]
has twice as many indecomposable representations as there are indecomposable representations for the ordinary quiver of type $A_3$, and $\Rep_{\FC_2} Q \cong \Rep_{\Vect}(A_3 \sqcup A_3)$ (with the linear orientation on both $A_3$).
\end{ex}

\begin{ex} 
More generally, consider the symmetric group $S_n$ on $n$ elements and let $\FC_n := \Rep(S_n)$. Let $p(n)$ be the number of irreducible objects in $\FC_n$, i.e. the number of partitions of $n$. Now let $Q$ denote the fusion quiver over $\FC_n$ with two vertices, and one edge labeled by the sign representation $\Sign$. 
Just as in the previous examples, one can show that $\Rep_{\FC_n} Q \cong \Rep_{\Vect} \sqcup_{i=1}^{p(n)} A_2$. Thus there are $3 p(n)$ indecomposable representations of $Q$.
\end{ex}

\begin{ex}\label{ex:Snstdrep}
We continue with $\FC_n := \Rep(S_n)$.
Let $V_{n-1} \in \FC_n$ denote the $(n-1)$-dimensional ``standard'' representation of $S_n$, also known as the Specht module for the partition $(n-1,1)$, or the Cartan subalgebra of $\mathfrak{sl}_n$.  
Let $Q_n$ be the fusion quiver over $\FC_n$ with two vertices, and with one edge labelled by $V_{n-1}$.
\[
\begin{tikzcd}
	\bullet & \bullet
	\arrow["V_{n-1}", from=1-1, to=1-2]
\end{tikzcd}
\]
Some examples of interesting indecomposable representations when $n=3$ are
\begin{equation} \label{ex:Snfun} \begin{tikzcd}
	\CC & V_2
	\arrow["V_2", from=1-1, to=1-2]
\end{tikzcd}, \qquad \begin{tikzcd}
	\CC \oplus V_2 & \CC \oplus V_2
	\arrow["V_2", from=1-1, to=1-2]
\end{tikzcd}.
\end{equation}
For the second example, the map associated to the edge has nonzero components $V_2 \ot \CC \to V_2$, $V_2 \ot V_2 \to V_2$, and $V_2 \ot V_2 \to \CC$.
As we will see in Example \ref{ex:Snunfold}, $Q_n$ will have infinitely many indecomposable representations for all $n \geq 3$.
\end{ex}

Up until now, we have only considered $\FC = \Rep G$ for a finite group $G$ (this includes $\FC = \Vect$). 
These fusion categories possess a forgetful functor (often called a fibre functor) to $\Vect$, and as a consequence, the Frobenius--Perron dimension $\FPdim(V)$ of any representation $V \in \FC$ is given by the dimension $\dim(V)$ of the underlying vector space of $V$. We now present some examples of fusion quivers and representations over fusion categories that are \emph{non-integral}, namely they contain objects with non-integral Frobenius--Perron dimensions. In particular, they do not possess fibre functors into $\Vect$. We continue to postpone the definition of Frobenius--Perron dimension, see \S\ref{ss:FPdimsigncoh}.

A major class of such examples are the Coxeter quivers introduced in \cite[\S 2]{heng2024coxeter}. Coxeter quivers can be viewed as fusion quivers over a certain family of fusion categories with particular allowed edge labels. Here is a prototypical example.

\begin{ex} \label{ex:coxeterquiver}
Let $\FC$ be the fusion category with two simple objects $\one$ (monoidal unit) and $\tau$, and fusion rule given by
\[
\tau \otimes \tau \cong \one \oplus \tau.
\]
This fusion category exists; see \cite{ostrik_rank2}.
It is also the fusion category denoted by $\TLJ_5^{even}$ in \cite[\S 1.2]{heng2024coxeter}. 
One has $\FPdim(\tau) = 2 \cos(\pi/5)$, also known as the golden ratio.

Consider the following fusion quiver $Q$ over $\FC$:
\[
\begin{tikzcd}
	\bullet & \bullet & \bullet & \bullet
	\arrow["\tau", from=1-1, to=1-2]
	\arrow["\one", from=1-2, to=1-3]
	\arrow["\one", from=1-3, to=1-4]
\end{tikzcd}.
\]
Then $\Rep_{\FC} Q$ is by definition the same as the category of representations for the Coxeter quiver (see \cite[Definition 2.5]{heng2024coxeter} for the definition)
\[
\begin{tikzcd}
	\bullet & \bullet & \bullet & \bullet
	\arrow["5", from=1-1, to=1-2]
	\arrow[from=1-2, to=1-3]
	\arrow[from=1-3, to=1-4]
\end{tikzcd}.
\]
By \cite[Theorem 3.5]{heng2024coxeter}, it follows that $\Rep_{\FC} Q$ has finitely many indecomposables (equal to the number of positive roots of $E_8$, which is also twice the number of positive roots of $H_4$).
\end{ex}

There are also many other ``exotic'' fusion categories, and correspondingly, exotic fusion quivers.
\begin{ex}\label{ex:exoticfusion}
Let $\FC:= \mathcal{E}^{(n)}_4$ be the fusion category given in \cite[Definition 3.20]{EM_smallFPdim}; it has 12 simple objects and is generated (monoidally) by an object denoted by (boldfaced) \textbf{5}. 
One has $\FPdim(\textbf{5}) = 2 \cos(\pi/8)$.
We consider the fusion quiver
\[
Q:= \begin{tikzcd}
	\bullet & \bullet
	\arrow["\textbf{5}", from=1-1, to=1-2]
\end{tikzcd}.
\]
(Not to be confused with the Coxeter quiver $\begin{tikzcd} \bullet \ar[r,"5"] & \bullet \end{tikzcd}$, whose `5' is not boldfaced.)
In Example \ref{ex:exoticunfold}, we will show that this fusion quiver has $2\cdot 28 + 2\cdot 20 = 96$ indecomposable representations valued in $\FC$.
\end{ex}

\subsection{Fusion quivers valued in \texorpdfstring{$\Rep G$}{Rep G} and \texorpdfstring{$G$}{G}-equivariantization} \label{sec:Gequivariant}
Throughout this subsection, $G$ will be a finite group and we will suppose that the underlying field $\Bbbk$ has characteristic not dividing $|G|$.
We shall show that representations of fusion quivers over $\Rep G$ are the same as (certain) $G$-equivariant representations of quivers over $\Vect$.

Recall that the (ordinary) $\ell$-Kronecker quiver has two vertices, a source and a sink, with $\ell$ edges from source to sink.
Following Convention \ref{conv:ordinaryquiver}, we consider its associated fusion quiver over $\Vect$ instead, which we will still call the $\ell$-Kronecker quiver and denote by $K_\ell$:
\[
K_\ell =
\begin{tikzcd}
    \bullet
    \arrow[r, "\Bbbk^\ell"]
    &
    \bullet
\end{tikzcd}.
\]
In particular, a representation $(X, \phi) \in \Rep_{\Vect} K_\ell$ consists of two vector spaces $X_1, X_2$, and a linear map $\phi: \Bbbk^\ell \otimes X_1 \rightarrow X_2$.

Let $(\Bbbk^\ell, \theta) \in \Rep G$ be an $\ell$-dimensional representation of $G$.
This defines a (categorical) group action of $G$ on $\Rep_{\Vect} K_\ell$: each $g \in G$ is assigned to a functor which sends a representation $(X,\phi)$ to $(X, \phi \circ (\theta_{g^{-1}} \otimes \id_{X_1}))$, and sends a morphism $\rho: (X,\phi) \rightarrow (Y,\phi')$ to itself (component-wise). 

Now consider the category of $G$-equivariant representations $(\Rep_{\Vect} K_\ell)^G$.
Recall that an object in $(\Rep_{\Vect} K_\ell)^G$ consists of a representation $(X,\phi)$ together with isomorphisms $u_g: (X, \phi \circ (\theta_{g^{-1}} \otimes \id_{X_1})) \xrightarrow{\cong} (X,\phi)$ of representations so that the following diagram commutes
\begin{equation} \label{eq:equvariantiso}
\begin{tikzcd}
    (X, \phi \circ (\theta_{h^{-1}} \otimes \id_{X_1}) \circ (\theta_{g^{-1}} \otimes \id_{X_1})) 
    	\arrow[r, "u_g"] \arrow[d, "="]
    & (X, \phi \circ (\theta_{h^{-1}} \otimes \id_{X_1})
    	\arrow[d, "u_h"] \\
    (X, \phi \circ (\theta_{(gh)^{-1}} \otimes \id_{X_1}))
    	\arrow[r, "u_{gh}"]
    & (X, \phi)
\end{tikzcd}
\end{equation}
for each $g, h \in G$.
The morphisms in $(\Rep_{\Vect} K_\ell)^G$ are those in $\Rep_{\Vect} K_\ell$ that commute with the associated isomorphisms $u_g$.

We claim that an object in $(\Rep_{\Vect} K_\ell)^G$ is by definition the same data as a representation of the following fusion quiver over $\Rep G$
\[
K_{(\Bbbk^\ell, \theta)} =
\begin{tikzcd}
    \bullet
    \arrow[r, "(\Bbbk^\ell{,} \theta)"]
    &
    \bullet
\end{tikzcd}.
\]
Indeed, the isomorphisms $(u_g)_i: X_i \rightarrow X_i$ define a representation of $G$ on $X_i$ (see \eqref{eq:equvariantiso}), and for each $g \in G$, the following commutative diagram 
\[
\begin{tikzcd}[column sep=large]
    \Bbbk^\ell \otimes X_1 \ar[r, "\theta_{g^{-1}} \otimes \id_{X_1}"] \ar[d, swap, "\id_{\Bbbk^\ell} \otimes (u_g)_1"] & \Bbbk^\ell \otimes X_1 \ar[r, "\phi"] \ar[d, dashed, "\theta_g \otimes (u_g)_1", swap] & X_2 \ar[d, "(u_g)_2"] \\
    \Bbbk^\ell \otimes X_1 \ar[r, "="] & \Bbbk^\ell \otimes X_1 \ar[r, "\phi"] & X_2
\end{tikzcd}
\]
can be deduced from $u_g$ being a morphism of fusion quiver representations, which implies that $\phi$ is a morphism of representations of $G$.
Similarly, given $(X,\phi, (u_g))$ and $(Y,\phi', (v_g))$ in $(\Rep_{\Vect} K_\ell)^G$, a morphism $\rho: (X,\phi) \to (Y,\phi')$ is $G$-equivariant if and only if 
\begin{equation} \label{eqn:equivariantiffmapofrep} 
(v_g)_i \circ \rho_i = \rho_i \circ (u_g)_i
\end{equation}
for all $g \in G$ and for all $i \in \{1,2\}$.
Since $(u_g)_i$ and $(v_g)_i$ are exactly the structure maps giving $X_i$ and $Y_i$ the structure of representations of $G$, \eqref{eqn:equivariantiffmapofrep} is exactly the same as saying each $\rho_i$ is a morphism of representations of $G$.
It follows that the category $(\Rep_{\Vect} K_\ell)^G$ is the same as $\Rep_{\Rep G} K_{(\Bbbk, \theta)}$.

More generally, let $(V,E)$ be an arbitrary ordinary quiver and consider its associated fusion quiver $Q = (\Vect, V, E, \Pi)$ as in Corollary \ref{cor:theotherway}. 
Following Convention \ref{conv:nomult}, we consider the fusion quiver $Q' = (\Vect, V, E', \Pi')$ obtained from $Q$ by combining edges with the same source and target into a single edge labeled by the direct sum of the original labeling objects (i.e.\ apply Lemma \ref{lem:directsum}).
The same analysis used for the $\ell$-Kronecker quiver can be applied to the general setting as well; we leave it to the reader as an exercise and only record the general statement.

\begin{prop}\label{prop:Gequivariantfusionquiver}
Let $G$ be a finite group. 
Suppose $G$ has an action on $\Rep_{\Vect} Q'$ defined by a choice of representation $(\Pi'_e, \theta_e) \in \Rep G$ for each $e \in E'$.
Let $\widetilde{Q} := (\Rep G, V, E', \widetilde{\Pi'})$ denote the fusion quiver over $\Rep G$ with $\widetilde{\Pi'}_e := (\Pi'_e, \theta_e)$ for each $e\in E'$. Then there is a natural equivalence of categories 
\[
(\Rep_{\Vect} Q)^G = \Rep_{\Rep G} \widetilde{Q}.
\]
\end{prop}
\begin{rem}\label{rmk:skewgroupalg}
The $G$-action on $\Rep_{\Vect} Q$ described in the proposition can also be understood as an action on the path algebra $A$ of $(V,E)$ as follows.
The datum of $G$-actions on $\Pi_e$ for each $e \in E'$ is the same as an action of $G$ on $A$ that fixes the idempotents point-wise and stabilizes the subspace of $A$ spanned by the arrows.
Moreover, with $G\# A$ denoting the skew group algebra of $A$ (with respect to the action of $G$), we get that
\[
G\# A\lmod \cong (\Rep (V,E))^G \cong (\Rep_{\Vect} Q)^G = \Rep_{\Rep G} \widetilde{Q};
\]
see e.g.\ \cite[\S 3.1]{demonet_skewgroup} for the first equivalence.
\end{rem}

\begin{ex}
The symmetric group $S_\ell$ on $\ell$ elements acts naturally on the $\ell$-Kronecker quiver by permuting the $\ell$ arrows.
Then the $S_\ell$-equivariant representations of the $\ell$-Kronecker quiver are the same as representations of the fusion quiver $K_{(\Bbbk^\ell, \pi)}$ over $\Rep S_\ell$, where $\pi$ is associated to the natural permutation representation that assigns to each $w \in S_\ell$ its permutation matrix.
Note that $(\Bbbk^\ell,\pi)$ is a \emph{decomposable} representation when $\ell > 1$, so the fusion quiver above has decomposable edge label.
\end{ex}

\begin{ex}\label{ex:Lusztig}
Set $\Bbbk = \CC$.
Let $G$ be a finite subgroup of $SU(2)$ and equip $\Bbbk^2 = \CC^2$ with the standard representation of $SU(2)$ on $\CC^2$.
Then $(\Rep_{\Vect} K_\ell)^G = \Rep_{\Rep G} K_{(\CC^2, \text{std})}$ was studied by Lusztig in his construction of canonical bases for affine types \cite[\S 2.2]{lusztig_affinequiver} (denoted by $\mathscr{C}^\delta$ in loc.\ cit.).
Unlike in the previous example, note that the action of $G$ on $\Rep_{\Vect} K_\ell$ does not come from an automorphism of the quiver itself.
\end{ex}

%
%
%
%

\section{Unfolding} \label{sec:unfolding}

In this chapter we argue that fusion representations of a fusion quiver $Q$, valued in $\MC$, are equivalent (as an abelian category) to ordinary representations over an ordinary quiver $Q'$. 
However, making this equivalence explicit involves choosing bases for certain morphism spaces in $\MC$ (and when these spaces have
dimension greater than $1$, $Q'$ will accordingly have multiple edges). In \S\ref{sec:fusionquivervect} we discussed how representations of fusion quivers for $\Vect$ (i.e. with $\FC = \Vect$ acting on $\MC = \Vect$) are essentially equivalent to
ordinary representations of a quiver, but we think of these fusion representations as not requiring a choice of basis. Thus the more satisfying version of our theorem will
canonically identify representations of the fusion quiver $Q = (\FC,V,E,\Pi)$ in $\MC$ with representations of a fusion quiver $\check{Q}_{\MC} = (\Vect, \check{V}, \check{E}, \check{\Pi})$ in $\Vect$; this will be done in Theorem \ref{thm:unfoldingnon-canonical}.

\subsection{Definition of unfolding} \label{sec:unfoldingdefinition}

\begin{defn} \label{defn:unfolding} Fix a fusion quiver $Q = (\FC,V,E,\Pi)$ and a finite semisimple $\FC$-module category $\MC$. Define a fusion quiver $\check{Q}_{\MC} = (\Vect, \check{V}, \check{E}, \check{\Pi})$ which we call the \emph{unfolded quiver} of $Q$ with respect to $\MC$ as follows.
\begin{itemize}
	\item Set $\check{V} = V \times \Irr(\MC)$, where $\Irr(\MC)$ denotes a set consisting of one simple object in $\MC$ from each isomorphism class. 
	\item For each edge $e \in E$ with source $s \in V$ and target $t \in V$, and each pair $L,L' \in \Irr(\MC)$, there is an edge from $(s,L)$ to $(t,L')$ in $\check{Q}$ if and only if $L'$ is a direct summand of $\Pi_e \ot L$. The set of all edges of this form 
	(as $L$ and $L'$ vary) will be denoted $\ufld(e) \subset \check{E}$.
	\item If $\check{e} \in \ufld(e)$ has source $(s,L)$ and target $(t,L')$, we label it with the vector space
	\begin{equation} \check{\Pi}_{e,L,L'} := \Hom_{\MC}(L',\Pi_e \ot L). \end{equation}
\end{itemize}
\end{defn}

It is important to note that the fusion quiver $\check{Q}_{\MC}$ depends on the choice of category $\MC$, whereas the original fusion quiver $Q$ does not. When $\MC$ is understood we may write $\check{Q}$ for brevity, but we have found that keeping the subscript sometimes helps to avoid confusion.

\begin{rem}
Note that $\check{Q}_\MC$ may be disconnected even if $Q$ was connected; see Example \ref{ex:Uqsl3at5}. The reader is also invited to compute $\check{Q}_2$ for $Q_2$ from Example \ref{ex:RepS2}.
\end{rem}

\begin{rem} Note that we could have defined $\check{Q}_{\MC}$ such that there is an edge from $(s,L)$ to $(t,L')$ for  \emph{all} $L, L' \in \Irr(\MC)$, which is labeled by $\Hom_{\MC}(L',\Pi_e \ot
L)$. This morphism space is zero precisely when $L'$ is not a direct summand of $\Pi_e \ot L$, so when one removes all edges with the zero label (see Lemma \ref{lem:removezero}), one
obtains the fusion quiver defined above. \end{rem}

\begin{rem} The difference between Definition \ref{defn:unfolding} and the definition of unfolding for Coxeter quivers given in \cite[\S 3]{heng2024coxeter} is that \begin{itemize}
	\item in the Coxeter quiver setting, $\Hom(L',\Pi_e \ot L)$ is always one-dimensional, and
	\item one has chosen a particular basis for $\Hom(L',\Pi_e \ot L)$, thus thinking of $\check{Q}$ as an ordinary quiver instead of a fusion quiver valued in $\Vect$. \end{itemize}
\end{rem}

The main result of this section is the following theorem.

\begin{thm} \label{thm:unfoldingnon-canonical}
Let $\MC$ be a finite semisimple left module category over $\FC$. 
There is an equivalence of abelian categories $\Rep_{\MC} Q \cong \Rep_{\Vect} \check{Q}_{\MC}$, up to the choice of a (non-canonical) right module structure on $\MC$ over $\Vect$ that commutes with the left module structure over $\FC$. 
\end{thm}
In fact, we will prove that the equivalence is canonical once we have fixed the right module structure on $\MC$ over $\Vect$; this will be Theorem \ref{thm:unfolding}. This right module structure is a technical distraction, which we explain in \S\ref{sec:unfoldingtheorempreamble}.

\subsection{Examples of unfolding} \label{sec:unfoldingexamples}

Before we prove the main result, we provide some examples of unfolding, and use Theorem \ref{thm:unfoldingnon-canonical} to understand some categories of fusion quiver representations. 
We lean heavily on the case where $Q$ has two vertices and one edge, because it keeps examples relatively small and easy to exposit.

For readability we will omit the labeling of edges in $\check{Q}$ by vector spaces, and instead draw the number of edges equal to the dimension of that vector space (i.e.\ we will draw the corresponding ordinary quiver as in Lemma \ref{lem:ordinaryquivermultipleedge}).

\begin{ex}\label{ex:Snunfold}
We continue from Example \ref{ex:Snstdrep}, where $V_{n-1} \in \Rep S_n$ is the standard $(n-1)$-dimensional representation of the symmetric group $S_n$ and $Q_n$ is the following fusion quiver over $\Rep S_n$.
\begin{equation} \label{eq:Qn} 
\begin{tikzcd}
	\bluebull & \redbull
	\arrow["V_{n-1}", from=1-1, to=1-2]
\end{tikzcd}
\end{equation}
Here $\check{Q}_n := (\check{Q}_n)_{\Rep S_n}$ will be a bipartite quiver with one vertex of each color for each simple representation of $S_n$.

One can verify that 
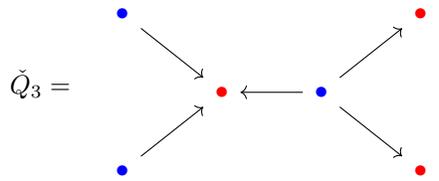
\[\check{Q}_3 = \quad \begin{tikzcd}
	\bluebull &&& \redbull \\
	& \redbull & \bluebull \\
	\bluebull &&& \redbull
	\arrow[from=1-1, to=2-2]
	\arrow[from=3-1, to=2-2]
	\arrow[from=2-3, to=2-2]
	\arrow[from=2-3, to=1-4]
	\arrow[from=2-3, to=3-4]
\end{tikzcd}.\]
The two examples of indecomposable representations in Example \ref{ex:Snstdrep} correspond respectively to the following indecomposable representations of $\check{Q}_3$:
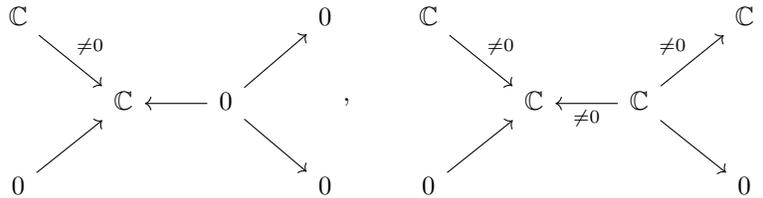
\[ \begin{tikzcd}
	\CC &&& 0 \\
	& \CC & 0 \\
	0 &&& 0
	\arrow[from=1-1, to=2-2, "\neq 0"]
	\arrow[from=3-1, to=2-2]
	\arrow[from=2-3, to=2-2]
	\arrow[from=2-3, to=1-4]
	\arrow[from=2-3, to=3-4]
\end{tikzcd}, \qquad \begin{tikzcd}
	\CC &&& \CC \\
	& \CC & \CC \\
	0 &&& 0
	\arrow[from=1-1, to=2-2, "\neq 0"]
	\arrow[from=3-1, to=2-2]
	\arrow[from=2-3, to=2-2, "\neq 0"]
	\arrow[from=2-3, to=1-4, "\neq 0"]
	\arrow[from=2-3, to=3-4]
\end{tikzcd}.\]
(Different choices of right module structures on $\MC=\Rep S_n$ over $\Vect$ will produce different non-zero maps on the arrows.)

One can also verify that
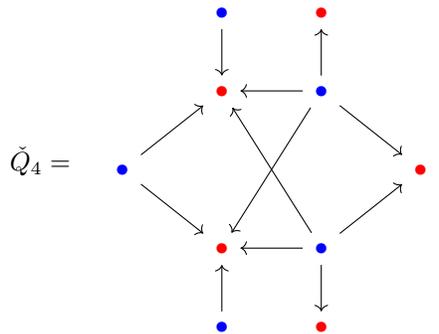
\[ \check{Q}_4 = \quad \begin{tikzcd}
	& \bluebull & \redbull \\
	& \redbull & \bluebull \\
	\bluebull &&& \redbull \\
	& \redbull & \bluebull \\
	& \bluebull & \redbull
	\arrow[from=3-1, to=2-2]
	\arrow[from=2-3, to=2-2]
	\arrow[from=3-1, to=4-2]
	\arrow[from=4-3, to=4-2]
	\arrow[from=2-3, to=3-4]
	\arrow[from=4-3, to=3-4]
	\arrow[from=4-3, to=2-2]
	\arrow[from=1-2, to=2-2]
	\arrow[from=2-3, to=4-2]
	\arrow[from=5-2, to=4-2]
	\arrow[from=4-3, to=5-3]
	\arrow[from=2-3, to=1-3]
\end{tikzcd}.\]
Multiple edges start appearing\footnote{Thanks to Sasha Kleshchev for a nice explanation why, for a partition $\lambda$, the multiplicity of the Specht module $S_{\lambda}$ inside $S_{\lambda} \ot V_{n-1}$ is one less than the number of removable boxes in $\lambda$. We chose not to reprint the explanation here.} for $n \ge 6$.

Note that $\check{Q}_n$ is finite type when $n=2$, affine type when $n=3$, and infinite and non-affine for $n \ge 4$. This corresponds to $\dim(V_{n-1})$ being $< 2$, $=2$, and $>2$ respectively. 
In particular, Theorem \ref{thm:unfoldingnon-canonical} shows that $\Rep_{\Rep S_n} Q_n$ has infinitely many indecomposables for all $n \geq 3$.
Note that each $\check{Q}_n$ is also the \emph{separated} McKay quiver associated to $S_n$ and the representation $V_{n-1}$; see \S\ref{sec:McKaycorrespondence} for more details.
\end{ex}

\begin{rem} In fact, for $n \ge 3$ one can provide an infinite family of indecomposable representations of $Q_n$ as follows. Let $K_{n-1}$ denote the fusion quiver over $\Vect$ with two vertices, and with one edge labelled by $\Res V_{n-1}$, i.e.\ $V_{n-1}$ viewed simply as an $(n-1)$-dimensional vector space. 
A representation of $K_{n-1}$ is the data of two vector spaces $X$ and $Y$ with a map $(\Res V_{n-1}) \ot X \to Y$ of vector spaces:
\[ \begin{tikzcd}
	X && Y
	\arrow["\Res V_{n-1}", from=1-1, to=1-3]
\end{tikzcd}. \]
Applying the induction functor $\Ind \co \Vect \to \Rep(S_n)$, we obtain objects $\Ind X$ and $\Ind Y$ in $\Rep(S_n)$ with a map $\Ind( (\Res V_{n-1}) \ot X) \to \Ind Y$. Using the fact that
\[ \Ind((\Res V) \ot X) \cong V \ot \Ind(X) \]
for any representation $V$ of $S_n$, we obtain a map $V \ot \Ind(X) \to \Ind(Y)$, giving the data of a representation of $Q_n$:
\[ \begin{tikzcd}
	\Ind X & \Ind Y
	\arrow["V_{n-1}", from=1-1, to=1-2]
\end{tikzcd}. \]
This process is functorial, giving a functor $\Ind \co \Rep_{\Vect} K_{n-1} \to \Rep_{\Rep(S_n)} Q_n$. Recall that $\Rep_{\Vect} K_{n-1}$ has infinite representation type whenever $n \ge 3$. The functor $\Ind$ does not send indecomposable representations to indecomposable representations, but nonetheless, one can find infinitely many indecomposable representations of $Q_n$ as summands of induced representations.

Note that if we view $\Rep_{\Rep(S_n)} Q_n = (\Rep_{\Vect} K_{n-1})^{S_n}$ via Proposition \ref{prop:Gequivariantfusionquiver} (or see \S\ref{sec:Gequivariant}), the functor $\Ind$ above is also the biadjoint of the forgetful functor $\Forget: (\Rep_{\Vect} K_{n-1})^{S_n} \rightarrow \Rep_{\Vect} K_{n-1}$.
 \end{rem}

\begin{ex}\label{ex:Uqsl3at5}
Let $\FC$ be the (braided) fusion category associated to $U_q\mathfrak{sl}_3$ at $q= e^{i\pi/5}$ \cite{AP95}.
It has 6 simple objects $\{\one, X, Y, L_{2,0}, L_{1,1,}, L_{0,2}\}$, with $X$ and $Y$ denoting the fundamental representations (mutually dual to one another).
Moreover, $X$ generates $\FC$, and the fusion rule for tensoring with $X$ (on the left) is encoded by the following diagram, where $m$ arrows from $E$ to $F$ means $F$ appears $m$ times as a summand of $X \otimes E$:
\[\begin{tikzcd}
	{L_{0,2}} && {L_{1,1}} && {L_{2,0}} \\
	& Y && X \\
	&& \one
	\arrow[from=2-2, to=3-3]
	\arrow[from=3-3, to=2-4]
	\arrow[from=2-4, to=2-2]
	\arrow[from=2-4, to=1-5]
	\arrow[from=1-5, to=1-3]
	\arrow[from=1-3, to=1-1]
	\arrow[from=1-1, to=2-2]
	\arrow[from=2-2, to=1-3]
	\arrow[from=1-3, to=2-4]
\end{tikzcd}.\]
(This is known as the McKay quiver of $\FC$ with respect to $X$; see Definition \ref{defn:McKayquiver}.)
For the general fusion rule of $\FC$, see e.g.\ \cite[\S 2]{MMMT_trihedral}.

Now consider the following fusion quiver over $\FC$.
\[
Q:=	
	\begin{tikzcd}
	s & t
	\arrow["X", from=1-1, to=1-2]
\end{tikzcd}\]
Then the unfolded quiver $\check{Q}_{\FC}$ is the following.
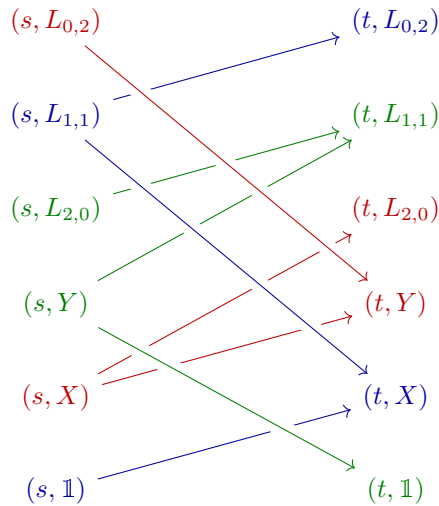
\[\begin{tikzcd}[column sep=20ex]
	{\color{myred} (s,L_{0,2})} & {\color{myblue} (t,L_{0,2})} \\
	{\color{myblue} (s,L_{1,1})} & {\color{mygreen} (t,L_{1,1})} \\
	{\color{mygreen} (s,L_{2,0})} & {\color{myred}(t,L_{2,0})} \\
	{\color{mygreen} (s,Y)} & {\color{myred}(t,Y)} \\
	{\color{myred} (s,X)} & {\color{myblue} (t,X)} \\
	{\color{myblue} (s,\one)} & {\color{mygreen} (t,\one)}
	\arrow[from=6-1, to=5-2, crossing over, myblue]
	\arrow[from=5-1, to=4-2, crossing over, myred]
	\arrow[from=5-1, to=3-2, crossing over, myred]
	\arrow[from=4-1, to=6-2, crossing over, mygreen]
	\arrow[from=4-1, to=2-2, crossing over, mygreen]
	\arrow[from=3-1, to=2-2, crossing over, mygreen]
	\arrow[from=2-1, to=1-2, crossing over, myblue]
	\arrow[from=2-1, to=5-2, crossing over, myblue]
	\arrow[from=1-1, to=4-2, crossing over, myred]
\end{tikzcd}\]
Notice that it is a disjoint union of three copies of type $A_4$ quiver, and so $\Rep_\FC Q$ has only finitely many ($3 \cdot 10 = 30$) indecomposable representations.
\end{ex}

In the following example, we shall compare representations valued in different semisimple module categories.
\begin{ex}\label{ex:differentM}
Let $\FC_\ell$ denote the fusion category associated to $\widehat{sl}_2$ at level $\ell$: it has $\ell+1$ simple objects $V_0 , V_1, ..., V_\ell$ with $V_0$ the monoidal unit and fusion rules completely determined by
\[
V_1 \otimes V_i \cong V_i \otimes V_1 \cong 
	\begin{cases}
	V_{i-1} \oplus V_{i+1}, &\text{for } 1 \leq i \leq \ell - 1; \\
	V_{\ell-1}, &\text{for } i = \ell.
	\end{cases}
\]
The finite semisimple module categories of $\FC_\ell$ were completely classified in \cite{ostrik_weakhopf}, and here we consider as examples those that are associated to ``type $D$.''

Suppose that $\ell \geq 2$ is even, so there exists the finite semisimple module category $\MC = \MC_\ell$ associated to the algebra object $V_0 \oplus V_\ell$. We refer the reader to \cite[\S 6]{ostrik_weakhopf} for details (see also \cite[\S 7]{KO_qmckay}), but mention that $\MC$ has $\ell/2 + 2$ simple objects $L_0, L_1, ..., L_{\ell/2-1}, L^+_{\ell/2}, L^-_{\ell/2}$, where the action of $V_1$ on $\MC$ is given by
\begin{align*}
V_1 \otimes L_i &\cong
	\begin{cases}
	L_1, &\text{for } i=0; \\
	L_{i-1} \oplus L_{i+1}, &\text{for } 1 \leq i \leq \ell/2 - 2; \\
	L_{\ell/2 - 2} \oplus L^+_{\ell/2} \oplus L^-_{\ell/2},  &\text{for } i=\ell/2 -1,
	\end{cases}
\\
V_1 \otimes L^+_{\ell/2} &\cong V_1 \otimes L^-_{\ell/2} \cong L_{\ell/2-1}.
\end{align*}

Take the following fusion quiver $Q_\ell$ over $\FC_\ell$.
\[
Q_\ell :=
\begin{tikzcd}
	\bullet & \bullet
	\arrow["V_1", from=1-1, to=1-2]
\end{tikzcd}
\]
We shall consider representations of $Q_\ell$ taking values in two different module categories, namely $\Rep_{\FC_\ell} Q_\ell$ and $\Rep_{\MC_\ell} Q_\ell$.

Once $\ell \geq 4$, it is clear that they will not be equivalent: $\Rep_{\FC_\ell} Q_\ell$ has $2(\ell+1)$ simple objects whereas $\Rep_{\MC_\ell} Q_\ell$ has $2(\ell/2 + 2)$ simple objects.
In fact, we can even show that the number of indecomposable objects will be different too: 
the unfolded quiver $\check{Q}_\FC$ consists of two copies of type $A_{\ell+1}$ quiver, whereas $\check{Q}_\MC$ consists of two copies of type $D_{\ell/2+2}$ quiver (however, note that the Coxeter numbers are equal).
In particular, $\Rep_{\FC_\ell} Q_\ell$ and $\Rep_{\MC_\ell} Q_\ell$ have $(\ell+1)(\ell+2)$ and $(\ell/2 + 1)(\ell+2)$ indecomposable objects respectively.
Nonetheless, they both have finitely many indecomposables.
Indeed, we will see in Theorem \ref{thm:finitetypecoxeternumber} that, in general, whether $\Rep_{\MC} Q$ has finitely many indecomposable objects is independent of the choice of the finite semisimple module category $\MC$.
\end{ex}

\begin{ex}\label{ex:exoticunfold}
Consider the fusion quiver $Q$ given in Example \ref{ex:exoticfusion}.
Using the fusion rule of $\textbf{5}$ given in \cite[Figure 2]{EM_smallFPdim}, one can show that $\check{Q}_{\FC}$ has two copies of type $A_7$ quivers and two copies of type $D_5$ quivers.
\end{ex}

\subsection{The unfolding theorem: preamble} \label{sec:unfoldingtheorempreamble}
%
Before giving the proof of the equivalence
\[
\Rep_{\MC} Q \cong \Rep_{\Vect} \check{Q}_{\MC}
\]
in Theorem \ref{thm:unfoldingnon-canonical}, let us recall some terminology from semisimple categories.  For any object $X \in \MC$ and simple object $L \in \Irr(\MC)$, $\Hom(L,X)$ is called the \emph{multiplicity space} of $L$
in $X$. There is a canonical isomorphism \begin{equation} \label{eq:isotypicdecomp} X \cong \bigoplus_{L \in \Irr(\MC)} L \boxtimes \Hom(L,X). \end{equation} The summand labeled by $L$
is called the \emph{isotypic component} of $L$ in $X$, and \eqref{eq:isotypicdecomp} is called the \emph{isotypic decomposition}.

Note that $\Hom(L,X)$ is just a vector space, and one can tensor an object of $\MC$ with a vector space to obtain an object of $\MC$. For example, $L \boxtimes (\Bbbk^{\oplus 2}) \cong (L \oplus L)$.
Indeed, any $\Bbbk$-linear category is a right module over $\Vect$, in a fashion commuting with the left action of any $\FC$; in other words, $\MC$ is now a $(\FC,\Vect)$-\emph{bimodule category} \cite[Definition 7.1.7]{EGNO}.
We write $\boxtimes$ for this action, and also
for the tensor product in the category of vector spaces. See below for further discussion.

If $A$ and $B$ are vector spaces and $L \ne L' \in \Irr(\MC)$, then Schur's lemma gives a canonical isomorphism
\begin{equation} \Hom_{\MC}(L \boxtimes A, L' \boxtimes B) = 0, \end{equation}
\begin{equation} \label{schurboxhoms} \Hom_{\MC}(L \boxtimes A, L \boxtimes B) \cong \Hom_{\Vect}(A,B). \end{equation}
After choosing a basis for $A$ and $B$, the isomorphism of \eqref{schurboxhoms} will send a matrix of numbers on the RHS to a matrix of rescaled identity maps on the LHS.

More generally, tensor-hom adjunction identifies
\begin{equation}\label{eq:tensorhom} \Hom_{\MC}(X \boxtimes A, Y \boxtimes B) \cong \Hom_{\Vect}(A, \Hom_{\MC}(X,Y)\boxtimes B).  \end{equation}
for any objects $X,Y \in \MC$ and vector spaces $A,B \in \Vect$.

We will be using these tools to construct our equivalence from $\Rep_{\Vect} \check{Q}_{\MC}$ to $\Rep_{\MC} Q$. As abstract as all of this nonsense is, we wish
to emphasize that in practice one can choose explicit bases for each morphism space $\Hom(L',\Pi_e \ot L)$, and thus get an equivalence between $\Rep_{\MC} Q$ and $\Rep_{\Vect}
\check{Q}_{\MC}$ (adapting to multiple edges as appropriate) which is much more concrete. The proofs, however, would be even more annoying - sometimes choosing bases is
inconvenient!

The rest of this section could be skipped on first reading. We have been fast and loose about one key point: how exactly does $\Vect$ act on $\MC$? We said that $L \boxtimes
(\Bbbk^{\oplus 2}) \cong (L \oplus L)$, but this raises the question: does $L \boxtimes (\Bbbk^{\oplus 2})$ have an independent existence as an object of $\MC$, or this is being
used as a definition? Also, we were trying to avoid bases, but don't we need to choose a basis of a vector space $A$ in order to interpret $L \boxtimes A$ in this way? For example,
how do we identify $L \boxtimes \Bbbk$ with $L$? 

One option is to assume from the start that $\MC$ is a right module category over $\Vect_{\Bbbk}$ (in a fashion commuting with the left $\FC$ action). This is a harmless
assumption, see the next paragraph. Then objects $X \boxtimes A$ are well-defined. The abstract tools above and the arguments in the next section will make perfect sense in this
context. In particular, we never need to pick a basis of $A$ in order to identify $X \boxtimes A$ with $X^{\oplus \dim A}$, as we are content to work with $X \boxtimes A$ directly;
we never even need to identify $X \boxtimes \Bbbk$ with $X$ in our proofs.

This option is harmless, because any additive $\Bbbk$-linear category $\MC$ can be equipped with a (right) monoidal action of $\Vect_{\Bbbk}$, though this involves some choices. The
skeletal category of $\Vect_{\Bbbk}$, whose objects are $\Bbbk^{\oplus n}$ for $n \ge 0$, is also monoidal and acts on $\MC$ via $X \boxtimes \Bbbk^{\oplus n} := X^{\oplus n}$,
using the standard basis of $\Bbbk^{\oplus n}$ to pick out these summands. Extending this action to all of $\Vect_{\Bbbk}$ requires choosing an equivalence between $\Vect_{\Bbbk}$
and its skeleton, which is a choice of basis for each finite-dimensional vector space. However, this is a choice made once and for all within $\Vect_{\Bbbk}$, which is in some
sense independent of other concerns in this paper. Again, our basis-free arguments do not care about any such choice, only that $X \boxtimes A$ is well-defined.

\begin{rem} One could also replace $\MC$ with the Deligne tensor product $\MC \boxtimes \Vect$ throughout, and only identify the equivalent categories $\MC$ and $\MC \boxtimes \Vect$ when needed. This seems like overkill. \end{rem}

\subsection{The unfolding theorem} \label{sec:unfoldingtheorem}

\begin{thm} \label{thm:unfolding}
Let $\MC$ be a finite semisimple category which is a $(\FC, \Vect)$-bimodule category.  There is a canonical equivalence of abelian categories $\Rep_{\MC} Q \cong \Rep_{\Vect} \check{Q}_{\MC}$, as constructed in the proof. \end{thm}

Theorem \ref{thm:unfoldingnon-canonical} is then an immediate consequence of the above, since any left module category $\MC$ can be (non-canonically) equipped with a commuting right $\Vect$ action; see the previous section. 

\begin{proof}[Proof of Theorem \ref{thm:unfolding}]
Let $(X,\phi)$ be a fusion representation in $\Rep_{\MC} Q$. We define a fusion representation in $\Rep_{\Vect} \check{Q}_{\MC}$ called $(FX,F\phi)$, as follows. \begin{itemize}
	\item Set $FX_{(v,L)} := \Hom_{\MC}(L,X_v)$.
	\item For $\check{e} \in \ufld(e)$ with source $(s,L)$ and target $(t,L')$, let $F\phi_{\check{e}}$ be the vector space map
\begin{equation} F\phi_{\check{e}} \co \Hom(L',\Pi_e \ot L) \ot \Hom(L,X_s) \to \Hom(L',X_t), \quad \alpha \ot \beta \mapsto \phi_e \circ (\id_{\Pi_e} \ot \beta) \circ \alpha, \end{equation}
or more visually, $F\phi_{\check{e}}$ sends $\alpha \ot \beta$ to the composition 
\begin{equation} \begin{tikzcd}
	L' & \Pi_e \ot L & \Pi_e \ot X_s & X_t
	\arrow["\alpha", from=1-1, to=1-2]
	\arrow["\id_{\Pi_e} \ot \beta", from=1-2, to=1-3]
	\arrow["\phi_e", from=1-3, to=1-4]	
\end{tikzcd}.
\end{equation}
\end{itemize}

Now suppose that $f \co (X,\phi) \to (Y,\psi)$ is a morphism in $\Rep_{\MC} Q$, so that $f_v \co X_v \to Y_v$ is a morphism in $\MC$ for all $v \in V$. Define a morphism $Ff \co (FX,F\phi) \to (FY,F\psi)$ as follows:
\begin{equation} Ff_{(v,L)} \co \Hom_{\MC}(L,X_v) \to \Hom_{\MC}(L,Y_v), \qquad \alpha \mapsto f_v \circ \alpha. \end{equation}
We leave to the reader the straightforward verification that $Ff$ is a morphism, i.e. the maps $Ff_{(v,L)}$ intertwine with the maps $F\phi$ and $F\psi$.

The fact that $Ff$ is defined, vertex-wise, using composition with $f_v$ implies that $F(\id) = \id$ and $F(f \circ g) = Ff \circ Fg$. Thus sending $(X,\phi) \mapsto (FX,F\phi)$ and $f
\mapsto Ff$ defines a functor $F \co \Rep_{\MC} Q \to \Rep_{\Vect} \check{Q}$.

The isotypic decomposition of \eqref{eq:isotypicdecomp} motivates the definition of the inverse functor.

Let $(\check{X}, \check{\phi})$ be a fusion representation in $\Rep_{\Vect} \check{Q}_{\MC}$. We define a fusion representation in $\Rep_{\MC} Q$ called $(G\check{X}, G\check{\phi})$ as follows. \begin{itemize}
\item Set
\begin{equation} G\check{X}_v = \bigoplus_{L \in \Irr(\MC)} L \boxtimes \check{X}_{(v,L)}. \end{equation}
We refer to the summand labeled by $L$ as the \emph{$L$-component} of $G\check{X}_v$.
\item For $e \in E$ with source $s$ and target $t$, recall that $\check{\Pi}_{e,L,L'} = \Hom_\MC(L',\Pi_e \ot L)$. So $\check{\phi}_{e,L,L'}$ is a vector space morphism 
\begin{equation} \check{\phi}_{e,L,L'} \co \Hom_{\MC}(L',\Pi_e \ot L) \boxtimes \check{X}_{(s,L)} \to \check{X}_{(t,L')}. \end{equation}
Tensor-hom adjunction identifies $\check{\phi}_{e,L,L'}$ with some $\MC$-morphism
\begin{equation} g_{e,L,L'} \co \Pi_e \ot L \boxtimes \check{X}_{(s,L)} \to L' \boxtimes \check{X}_{(t,L')}. \end{equation}
Now set $G\check{\phi}_e \co G\check{X}_s \to G\check{X}_t$ to be the $\MC$-morphism which, from $\Pi_e$ tensored with the $L$ component of $G\check{X}_s$, to the $L'$ component of $G\check{X}_t$, is equal to $g_{e,L,L'}$.
\end{itemize}

Now suppose $g \co (\check{X}, \check{\phi}) \to (\check{Y}, \check{\psi})$ is a morphism in $\Rep_{\Vect} \check{Q}_{\MC}$, so that $g_{(v,L)} \co \check{X}_{(v,L)} \to \check{Y}_{(v,L)}$ is a linear map for all $(v,L) \in \check{V}$. Define a morphism $Gg \co (G\check{X}, G\check{\phi}) \to (G\check{Y}, G\check{\psi})$ as follows. The $\MC$-morphism
\begin{equation} Gg_v \co \bigoplus_{L \in \Irr(\MC)} L \boxtimes \check{X}_{(v,L)} \to \bigoplus_{L \in \Irr(\MC)} L \boxtimes \check{Y}_{(v,L)} \end{equation}
will necessarily restrict to zero as a map from the $L$ component to the $L'$ component when $L \ne L'$, by Schur's lemma. So
\begin{equation} Gg_v = \sum (Gg_v)_L, \qquad (Gg_v)_L \co L \boxtimes \check{X}_{(v,L)} \to L \boxtimes \check{Y}_{(v,L)}. \end{equation}
Under the identification of \eqref{schurboxhoms}, $(Gg_v)_L$ is identified with $g_{(v,L)} \co \check{X}_{(v,L)} \to \check{Y}_{(v,L)}$.

The remainder of the proof involves tracking through a forest of tautological isomorphisms, but nothing is surprising. We leave to the reader the verification that:
\begin{itemize}
	\item $Gg$ is a morphism in $\Rep_{\MC} Q$.
	\item Sending $(\check{X}, \check{\phi})$ to $(G\check{X}, G\check{\phi})$ and $g$ to $Gg$ defines a functor $G \co \Rep_{\Vect} \check{Q}_{\MC} \to \Rep_{\MC} Q$.
	\item $F$ and $G$ are inverse functors.
\end{itemize}
\end{proof}

\begin{rem} The operation of unfolding is an idempotent, i.e. $\check{\check{Q}} = \check{Q}$.
More verbosely, 
	\[ (\check{\check{Q}}_{\MC} )_{\Vect} = \check{Q}_{\MC}. \] Indeed, when $Q$ is already a fusion quiver over $\FC = \Vect$, we claim
that $Q \cong \check{Q}_{\Vect}$ canonically, in the following sense. When $\MC = \Vect$, we let the one-dimensional space $L = \Bbbk$ be the unique element of $\Irr(\MC)$. We identify $\Hom_{\Vect}(\Bbbk,X)$
with $X$ for any vector space $X$ in the usual way, sending a linear map $f$ to the image of $1 \in \Bbbk$ under $f$. We also view $L$ as the monoidal identity, canonically identifying
$X \ot L$ with $X$ using the multiplication map. Now, if $Q = (\Vect, V, E, \Pi)$ and $\MC = \Vect$, then one can verify that $\check{Q}_{\Vect} = Q$ via these identifications. That is, $V$ and
$\check{V}$ are in canonical bijection, $\ufld(e)$ is always a single edge identified with $e$, and $\check{\Pi}_{e,L,L} = \Pi_e$ via these identifications. Moreover, under these
identifications, the functor $\Rep_{\Vect} Q \to \Rep_{\Vect} \check{Q}_{\Vect}$ defined in Theorem \ref{thm:unfolding} is the identity functor. \end{rem}

\begin{cor} \label{cor:artinian} The category $\Rep_{\MC} Q$ is an artinian abelian category with global homological dimension 1. \end{cor}
	
\begin{proof} This follows because $\Rep_{\MC} Q$ is equivalent to $\Rep_{\Vect} \check{Q}_{\MC}$, which in turn is equivalent by Lemma \ref{lem:ordinaryquivermultipleedge} to representations of some ordinary quiver. Representations of an ordinary quiver form an artinian abelian category with global homological dimension 1 \cite{Gabriel_finitetype} (see also \cite[Chapter VII.1]{ASS_elerep}). \end{proof}
	
\begin{rem} \label{rmk:pathalgebra}
Associated to a fusion quiver $Q$ over $\FC$, one can associate a (path) algebra object $\Path(Q)$ in $\FC$, whose category of modules $\Path(Q)\lmod$ is equivalent to $\Rep_{\FC}(Q)$ as module categories over $\FC$. 
Namely, the algebra object $\Path(Q)$ is a $\FC$-tensor algebra $T_S(D)$ in the sense of \cite[Definition 3.2]{EKW_tensoralgebras} with:
\begin{itemize}
\item the semisimple (exact = semisimple in fusion categories) base algebra $S$ given by the direct sum $\one^{\oplus |V|}$ of the trivial algebra $\one$; and
\item the generating bimodule $D$ given by the direct sum $\bigoplus_{e \in E} \Pi_e$,
\end{itemize}
where the $S$-bimodule structure of $D$ is defined in a similar fashion for Coxeter quivers in \cite[\S 5.2]{heng2024coxeter}.
\end{rem}

\section{Dimension vectors and reflection functors} \label{sec:reflection}

\subsection{Simple representations} \label{sec:simples}

\begin{notation} Let $Q = (\FC,V,E,\Pi)$ be a fusion quiver. We say a vertex $v\in V$ is a \emph{sink} (resp. \emph{source}) if it has only incoming (resp. outgoing) arrows. We say that a fusion representation $(X,\phi)$ is \emph{supported} at
a vertex $v$ if $X_w = 0$ for all $w \ne v$. \end{notation}

\begin{defn} Let $Q = (\FC,V,E,\Pi)$ be a fusion quiver, and $\MC$ a finite semisimple module category for $\FC$. Let $L$ be a simple object of $\MC$ and $v \in V$. Let $S(v,L) = (X,\phi)$ be the quiver representation for which
\begin{equation} X_v = L, \qquad X_w = 0 \text{ for all } w \ne v, w \in V, \qquad \phi_e = 0 \text{ for all } e \in E. \end{equation}
\end{defn}

Of course, if $Q$ has no loops at $v$, then the condition $\phi_e = 0$ is already implied by being supported at $v$.

\begin{prop} \label{prop:classifysimple} With notation as above, $S(v,L)$ is a simple object in $\Rep_{\MC} Q$. If $Q$ has no loops or cycles, then every simple object in $\Rep_{\MC} Q$ has the form $S(v,L)$ for some $v \in V$ and $L$ simple in $\MC$. \end{prop}

\begin{proof} The simplicity of $S(v,L)$ is obvious. Suppose $Q$ has no loops or cycles. If $(X,\phi)$ is any representation of $Q$, then the set of
vertices for which $X_v \ne 0$ forms a directed subgraph which must have at least one sink $t$. If $L$ is any simple submodule of $X_t$, then $S(t,L)$ is a subrepresentation of
$(X,\phi)$. Thus the only simple representations are $S(v,L)$. \end{proof}

Let $\ZZ^V$ be the free $\ZZ$-module with basis $\{\alpha_v\}$. For an abelian category $\AC$ let $[\AC]$ denote its Grothendieck group.

\begin{cor} \label{cor:grothendieckgrpidentification} If $Q$ has no loops or cycles, then $[\Rep_{\MC} Q] \cong [\MC]\ot_{\ZZ} \ZZ^V =: [\MC]^V $. \end{cor}

\begin{proof} By Corollary \ref{cor:artinian}, $\Rep_{\MC} Q$ is Artinian, so the Jordan--H\"{o}lder theorem holds. Hence $[\Rep_{\MC} Q]$ is a free abelian group with a basis  $\{[S(v,L)]\}_{v \in V, L \in \Irr(\MC)}$ in bijection with the isomorphism classes of simple objects. For any fixed $v$ we can identify the span of $\{[S(v,L)]\}_{L \in \Irr(\MC)}$ with $[\MC]$, which is spanned by $\{[L]\}$. \end{proof}

\begin{notation} \label{notn:rootspace} Henceforth, when $Q$ has no loops or cycles, we identify $[\Rep_{\MC} Q]$ and $[\MC]^V$ along the isomorphism specified in the proof above.
If $M$ is an object of $\MC$, write $[M]\alpha_v$ for $[M] \ot \alpha_v$. Thus e.g. $[L] \alpha_v$ is the symbol of $S(v,L)$. Given a fusion representation $(X,\phi)$ in $\Rep_{\MC} Q$, its \emph{dimension vector} is
	\begin{equation} \udim (X,\phi) = \sum_{v \in V} [X_v] \alpha_v \qquad \in [\MC]^{V}. \end{equation}
In the special case $\MC = \FC$, we write $[\one] \alpha_v$ and $\alpha_v$ interchangeably.
\end{notation}

\begin{rem} For $\Rep_{\MC}(Q)$, unlike ordinary quiver representations, $\alpha_v$ does not make sense as a dimension vector on its own, without being ``multiplied'' by an element of $[\MC]$. We reiterate that $[\MC]$ is in general not a ring and does not have a unit element, which prevents us from saying that $[\Rep_{\MC}(Q)]$ is free as a $[\MC]$-module. \end{rem}
	
For the rest of this chapter we assume that $Q$ has no loops or cycles.

\subsection{Reflection functors} \label{sec:reflections}

\begin{notation} Recall that for an object $\Pi$ in a fusion category, let $\leftdual{\Pi}$ denote its left dual, and $\rightdual{\Pi}$ denote its right dual. \end{notation}

\begin{defn} \label{defn:reflectquivers}
Let $Q = (\FC,V,E,\Pi)$ be a fusion quiver, and let $v \in V$ be a sink (resp. source). Let $\sigma_v Q$ be the fusion quiver over $\FC$ defined as follows. Its
underlying quiver is obtained from the underlying quiver $(V,E)$ by reversing the orientation on all arrows that abut (i.e. start or end at) $v$. If $e$ was an arrow abutting
$v$, we write $\opp(e)$ for the corresponding arrow (with reversed orientation) in $\sigma_v Q$; for all other arrows, we abusively write $e$ for the corresponding arrow (with same
orientation) in $\sigma_v Q$. The labels on edges not abutting $v$ are unchanged. For each edge $e \in E$ which abuts $v$, set $\Pi_{\opp(e)} = \leftdual{\Pi_e}$ (resp. $\Pi_{\opp(e)} = \rightdual{\Pi_e}$).  
\end{defn}

\begin{prop} \label{prop:ifnotSvL} Let $(X,\phi)$ be an indecomposable object in $\Rep_{\MC} Q$. \begin{enumerate}
	\item Let $v \in V$ be a sink. Then either $(X,\phi) \cong S(v,L)$ for some $L \in \Irr(\MC)$, or the following map $\xi$ is surjective (i.e. an epimorphism):
	\begin{equation} \label{eq:xidef} \xi \colon \bigoplus_{e \in E \text{ s.t. } t(e) = v} \Pi_e \ot X_{s(e)} \namedto{\oplus \phi_e} X_v. \end{equation}
	\item Let $v \in V$ be a source. Then either $(X,\phi) \cong S(v,L)$ for some $L \in \Irr(\MC)$, or the following map $\theta$ is injective (i.e. a monomorphism):
	\begin{equation} \label{eq:thetadef} \theta \colon  X_v \namedto{\oplus \phi'_e} \bigoplus_{e \in E \text{ s.t. } s(e) = v} \rightdual{\Pi_e} \ot X_{t(e)}. \end{equation}
	Here, $\phi'_e$ matches $\phi_e$ under the adjunction isomorphism $\Hom(\Pi_e \ot X_v, Y) \cong \Hom(X_v, \rightdual{\Pi_e} \ot Y)$.
\end{enumerate}
\end{prop}

If $(X,\phi)$ is supported at $v$, then all maps $\phi_e$ are zero. We may write $(X,\phi) = (X_v,0)$ to emphasize this. This notation is used in the proofs below.

\begin{proof} The proof in the classical case can largely be copied here, although one should be careful to use categorical arguments rather than arguments involving elements. We crucially use that $\MC$ is assumed to be semisimple. For sake
of completeness, we sketch the proof here. In both cases, the argument has the form: if $\xi$ is not surjective (resp. $\theta$ is not injective) then $(X,\phi)$ has a nonzero direct
summand which is supported at the vertex $v$. By indecomposability, $(X,\phi) = (X_v,0)$ is supported at $v$. If $(X_v,0)$ is indecomposable then $X_v$ is indecomposable
in $\MC$, and hence is isomorphic to $L$ for some $L \in \Irr(M)$.

Let $v \in V$ be a sink, and suppose $\xi$ is not surjective. Let $I$ be the image of $\xi$ inside $X_v$, and $C$ be a complementary summand. Then $(X,\phi) \cong (X',\phi') \oplus (C,0)$, where $(C,0)$ is supported at vertex $v$, and where $(X',\phi')$ is the same as $(X,\phi)$ except with $X_v$ replaced with $I$.

Let $v \in V$ be a source, and suppose $\theta$ is not injective. For an edge $e$ starting at $v$, let $K_e$ be the kernel of $\phi'_e$ inside $X_v$. By nature of the adjunction
isomorphism, $K_e$ is the maximal submodule of $X_v$ for which $\phi_e$ restricted to $\Pi_e \ot K_e$ is zero. Let $K$ be the kernel of $\theta$, so that $K = \bigcap K_e$. Thus
$\phi_e$ restricted to $\Pi_e \ot K$ is zero for all $e$. Let $D$ be a complementary summand of $X_v$. Then $(X,\phi) \cong (X',\phi') \oplus (K,0)$, where $(K,0)$ is supported at
vertex $v$, and where $(X',\phi')$ is the same as $(X,\phi)$ except with $X_v$ replaced with $D$.

We leave it to the reader to verify these direct sum decompositions of $(X,\phi)$. \end{proof}

\begin{lem} \label{lem:kerxifunctor} Let $v \in V$ be a sink.  For an object $(X,\phi)$, let the morphism $\xi = \xi_{(X,\phi)}$ be defined as in \eqref{eq:xidef}. Then $\Ker \xi$ is functorial, in the sense that a morphism $f \co (X,\phi) \to (Y,\psi)$ induces a map $\Ker \xi_{(X,\phi)} \to \Ker \xi_{(Y,\psi)}$.

Similarly, let $v \in V$ be a source. Then $\Coker \theta$ is functorial. \end{lem}
	
\begin{proof} This is straightforward. See \cite[discussion after Definition 4.6]{heng2024coxeter}. \end{proof}

\begin{notation} Let $\one$ denote the monoidal identity of $\FC$. For an object $\Pi_e$ labeling an edge $e$, we let $\Eval_e \co \leftdual{\Pi_e} \ot \Pi_e \to \one$ be the canonical \emph{evaluation} map. \end{notation}

\begin{defn} \label{defn:sinkreflection} Let $v \in V$ be a sink. Let 
	\[ R_v^+ \co \Rep_{\MC} Q \to \Rep_{\MC} \sigma_v Q, \qquad R_v^+ \co (X,\phi) \mapsto (R_v^+ X, R_v^+ \phi), \qquad R_v^+ \co f \mapsto R_v^+ f \] be the functor defined as follows.
\begin{itemize}
	\item Set $(R_v^+ X)_w = X_w$ for all $w \ne v$.
	\item Set $(R_v^+ \phi)_e = \phi_e$ for all $e$ which do not end at $v$.
	\item Set $(R_v^+ X)_v = \Ker \xi_{(X,\phi)}$. 
	\item Set $(R_v^+ \phi)_{\opp(e)} = \iota_e$ (see below) for all $e$ which end at $v$.
	\item If $f \co (X,\phi) \to (Y,\psi)$, then $(R_v^+ f)_w = f_w$ for all $w \ne v$.
	\item If $f \co (X,\phi) \to (Y,\psi)$, then $(R_v^+ f)_v$ is the induced map from $\Ker \xi_{(X,\phi)}$ to $\Ker \xi_{(Y,\psi)}$, see Lemma \ref{lem:kerxifunctor}.
\end{itemize}
Let $\incl_e$ be the map $\Ker \xi \to \Pi_e \ot X_{s(e)}$ which includes $\Ker \xi$ into the direct sum on the LHS of \eqref{eq:xidef}, and projects to the $e$ factor. Then
\begin{equation} \iota_e = (\Eval_e \ot \id_{X_{s(e)}}) \circ (\id_{\leftdual{\Pi_e}} \ot \incl_e). \end{equation}
\end{defn}

\begin{lem} The functor $R_v^+$ is well-defined and additive. \end{lem}

\begin{proof} One need only verify that $R_v^+ f$ is a well-defined morphism, and that $R_v^+$ preserves identity morphisms, morphism composition, and morphism addition. These are all vertex-wise or edge-wise conditions, which are immediate to verify for all vertices $w \ne v$ and edges not abutting $v$, and still easy to verify near $v$. We leave this to the reader. \end{proof}
	
\begin{ex} \label{ex:reflectS3}
We return to the second example from \eqref{ex:Snfun} in Example \ref{ex:Snstdrep}, where $\FC = \Rep(S_3)$, and we reflect at the unique sink $v$. We identify $V_2$ with its (left or right) dual. The kernel of the map
\[ \xi \colon V_2 \ot (\CC \oplus V_2) \to \CC \oplus V_2 \]
is isomorphic to $\Sign \oplus V_2$. Thus we have
\begin{equation} R_v^+ \left( \; \begin{tikzcd}
	\CC \oplus V_2 & \CC \oplus V_2
	\arrow["V_2", from=1-1, to=1-2]
\end{tikzcd} \; \right) \quad = \quad \left( \; \begin{tikzcd}
	\CC \oplus V_2 & \Sign \oplus V_2
	\arrow["V_2", from=1-2, to=1-1]
\end{tikzcd} \; \right). \end{equation}
The map
\[ \iota_e \colon V_2 \ot (\Sign \oplus V_2) \to \CC \oplus V_2 \]
(which is associated to the edge after applying $R_v^+$) has nonzero components $V_2 \ot \Sign \to V_2$, $V_2 \ot V_2 \to \CC$, and $V_2 \ot V_2 \to V_2$. \end{ex}

\begin{defn} \label{defn:sourcereflection} Let $v \in V$ be a source. Let 
	\[ R_v^- \co \Rep_{\MC} Q \to \Rep_{\MC} \sigma_v Q, \qquad R_v^- \co (X,\phi) \mapsto (R_v^- X, R_v^- \phi), \qquad R_v^- \co f \mapsto R_v^- f \] be the functor defined as follows.
\begin{itemize}
	\item Set $(R_v^- X)_w = X_w$ for all $w \ne v$.
	\item Set $(R_v^- \phi)_e = \phi_e$ for all $e$ which do not start at $v$.
	\item Set $(R_v^- X)_v = \Coker \theta_{(X,\phi)}$. 
	\item Set $(R_v^- \phi)_{\opp(e)} = \rho_e$ (see below) for all $e$ which end at $v$.
	\item If $f \co (X,\phi) \to (Y,\psi)$, then $(R_v^- f)_w = f_w$ for all $w \ne v$.
	\item If $f \co (X,\phi) \to (Y,\psi)$, then $(R_v^- f)_v$ is the induced map from $\Coker \theta_{(X,\phi)}$ to $\Coker \theta_{(Y,\psi)}$, see Lemma \ref{lem:kerxifunctor}.
\end{itemize}
Recalling that $\Pi_{\opp(e)} = \rightdual{\Pi_e}$, the morphism $\rho_e$ should be a map
\begin{equation} \rho_e \co \rightdual{\Pi_e} \ot X_{t(e)} \to \Coker \theta, \end{equation}
and indeed we just include $\rightdual{\Pi_e} \ot X_{t(e)}$ as one term of the direct sum on the RHS of \eqref{eq:thetadef}, followed by the quotient map to the cokernel.
\end{defn}

\begin{lem} The functor $R_v^-$ is well-defined and additive. \end{lem}

\begin{proof} Also straightforward. \end{proof}

\begin{prop} \label{prop:doublereflect} Let $v \in V$ be a sink. For any fusion representation $(X,\phi)$, the map $\theta_{R_v^+(X,\phi)}$ is injective. There is a natural transformation $R_v^- R_v^+ \to \Id_Q$, which is an isomorphism on those fusion representations $(X,\phi)$ for which $\xi_{(X,\phi)}$ is surjective. Here $\Id_Q$ represents the identity functor of $\Rep_{\MC} Q$.

Let $v \in V$ be a source. For any fusion representation $(X,\phi)$, the map $\xi_{R_v^-(X,\phi)}$ is surjective. There is a natural transformation $\Id_Q \to R_v^+ R_v^-$, which is an isomorphism on those fusion representations $(X,\phi)$ for which $\theta_{(X,\phi)}$ is injective. \end{prop}

\begin{proof}
We prove the result when $v$ is a sink. For sanity, let us write $(Y,\psi) = R_v^+ (X,\phi)$. Consider the following diagram:
\begin{equation} \label{sesforsink}
\begin{tikzcd}
0 \arrow[r] & Y_v \arrow[d,"="] \arrow[r,hookrightarrow] & \bigoplus_{e} \Pi_e \ot X_{s(e)} \arrow[d,"="] \arrow[r,"\xi_{(X,\phi)}"] & X_v \arrow[d,dashed] & \\
0 \arrow[r] & Y_v \arrow[r,"\theta_{(Y,\psi)}"] & \bigoplus_{e} \Pi_e \ot X_{s(e)} \arrow[r,twoheadrightarrow] & (R_v^- Y)_v \arrow[r] & 0
\end{tikzcd}.
\end{equation}
Let us elaborate. The first row identifies $Y_v$ as the kernel of $\xi = \xi_{(X,\phi)}$, so it is exact by definition. The first map is the natural inclusion map. However, we claim that this inclusion map agrees with $\theta = \theta_{(Y,\psi)}$, making the first square commute, and proving that $\theta$ is injective. Given this, the second row identifies $(R_v^- R_v^+ X)_v = (R_v^- Y)_v$ as the cokernel of $\theta$, so it is also exact. The dashed map $X_v \to (R_v^- R_v^+ X)_v$ is then induced by the universal property of cokernels, and is an isomorphism if and only if $\xi$ is surjective. The functoriality of this diagram in $(X,\phi)$ is a simple consequence of Lemma \ref{lem:kerxifunctor}.

It remains to prove that $\theta$ agrees with the natural inclusion. Below, let 
\[ \rho_e \co Y_v \to \Pi_e \ot X_{s(e)} \]
denote the composition of the inclusion map $Y_v \into \bigoplus \Pi_e \ot X_{s(e)}$ with the projection map to the factor indexed by $e$. Clearly, the inclusion map itself is the direct sum $\oplus \rho_e$ over all edges $e$ with target $v$. Meanwhile, by construction, $\theta$ is the direct sum $\oplus \psi_{\opp(e)}'$ of maps $\psi_{\opp(e)}' : Y_v \to \rightdual{(\leftdual{\Pi_e})} \ot X_{s(e)}$, and we identify $\rightdual{(\leftdual{\Pi_e})}$ with $\Pi_e$. We win if we show that $\psi_{\opp(e)}' = \rho_e$.

These maps $\psi_{\opp(e)}'$ were constructed from $\psi_{\opp(e)}$ using adjunction, so can be realized by precomposing $\psi_{\opp(e)}$ (tensored with an identity map) with the unit of adjunction. Unraveling the definition of $\psi_{\opp(e)}$ (c.f. $\iota_e$ from Definition \ref{defn:sinkreflection}), we see that $\psi_{\opp(e)}'$ is the composition
\begin{equation} Y_v \longrightarrow \Pi_e \ot \leftdual{\Pi_e} \ot Y_v \longrightarrow \Pi_e \ot \leftdual{\Pi_e} \ot \Pi_e \ot X_{s(e)} \longrightarrow \Pi_e \ot X_{s(e)}. \end{equation}
The first map is the unit of adjunction tensored with $\id_{Y_v}$. The second map is the identity of $\Pi_e \ot \leftdual{\Pi_e}$ tensored with $\rho_e$. The third map is $\Eval_e$ sandwiched between $\id_{\Pi_e}$ and $\id_{X_{s(e)}}$. 
Because $\Eval_e$ is the counit of adjunction, it can be cancelled with the unit of adjunction (after applying the interchange law), from which one concludes
\begin{equation} \psi_{\opp(e)}' = \rho_e \end{equation}
as desired.

The argument when $v$ is a source is significantly easier! This time, set $(Z,\beta) = R_v^- (X,\phi)$. One has the following diagram:
\begin{equation} 
\begin{tikzcd}
& X_v \arrow[r,"\theta_{(X,\phi)}"] \arrow[d, dashed] & \bigoplus_{e} \rightdual{\Pi_e} \ot X_{t(e)} \arrow[d,"="] \arrow[r,twoheadrightarrow] & Z_v \arrow[d,"="] \arrow[r] & 0 \\
0 \arrow[r] & (R_v^+ Z)_v \arrow[r,hookrightarrow] & \bigoplus_{e} \rightdual{\Pi_e} \ot X_{t(e)} \arrow[r,"\xi_{(Z,\beta)}"] & Z_v \arrow[r] & 0 \\
\end{tikzcd}.
\end{equation}
Again, both rows are exact by definition, and the interesting feature is the commutativity of the right square. However, this time $\xi_{(Z,\beta)}$ is the canonical quotient map by definition, without the need for any adjunction.
\end{proof}

\begin{rem}
Note that the natural transformations $R_v^- R_v^+ \to \Id_Q$ and $\Id_Q \to R_v^+ R_v^-$ are the counit and unit respectively of the adjunction pair $(R_v^-, R_v^+)$.
We leave this verification to the reader.
\end{rem}

Recall that only nonzero representations are called indecomposable.

\begin{cor} \label{cor:preservesindecomposable} Let $v\in V$ be a sink, and let $(X,\phi)$ be an indecomposable fusion representation. Either $(X,\phi) = S(v,L)$ for some $L$, or $R_v^+(X,\phi)$ is indecomposable.

Let $v \in V$ be a source, and let $(X,\phi)$ be an indecomposable fusion representation. Either $(X,\phi) = S(v,L)$ for some $L$, or $R_v^-(X,\phi)$ is indecomposable.\end{cor}

\begin{proof} Let $v$ be a sink. By Proposition \ref{prop:ifnotSvL}, either $(X,\phi) = S(v,L)$ or $\xi_{(X,\phi)}$ is surjective, which by Proposition \ref{prop:doublereflect} implies that $R_v^- R_v^+ (X,\phi) \cong (X,\phi)$. Suppose that $R_v^+ (X,\phi) \cong A \oplus B$ is decomposable. Then 
	\[ (X,\phi) \cong R_v^- R_v^+ (X,\phi) \cong R_v^-(A) \oplus R_v^-(B). \]
Since $(X,\phi)$ is indecomposable, one of the summands must be zero, so let us assume that $R_v^-(B) = 0$. In particular, $B \neq 0$ must be supported on the vertex $v$ inside $\sigma_v Q$, and the map $\theta_B$ is necessarily zero (its codomain is zero). But then this implies $\theta_{R_v^+ (X,\phi)}$ is not injective, being zero on the summand $B_v \neq 0$. This contradicts Proposition \ref{prop:doublereflect}.

The argument when $v$ is a source is completely parallel. \end{proof}

\subsection{Forms and reflections associated to fusion quivers} \label{sec:sesquilinearformreflection}

We define a $[\FC]$-valued form associated to $Q$; we remind the reader that we are still working with the assumption that $Q$ has no loops or cycles.

\begin{defn}\label{defn:Qsesquilinearform}
Let $Q$ be a fusion quiver over $\FC$ and let $V$ be its set of vertices. Let $\ZZ^V$ have basis $\{\alpha_v\}_{v \in V}$.
We define $\langle - , - \rangle_Q$ to be the unique $\ZZ$-bilinear form on $\ZZ^V$, valued in $[\FC]$, defined on the basis elements $\alpha_v$ as follows:
\begin{equation} \label{eqn:Qsesquilinearform}
\langle \alpha_v, \alpha_w \rangle_Q := 
	\begin{cases} 
		2 = 2 \cdot [\one], & \text{ if } v = w, \\
		- [\rightdual{\Pi_e}], & \text{ if  $v\neq w$ and $v \xrightarrow{\Pi_e} w$}, \\
		- [\Pi_e], & \text{ if $v \neq w$ and $w \xrightarrow{\Pi_e} v$}, \\
		0 & \text{ if $v \ne w$ and no edge abuts both $v$ and $w$}. 
	\end{cases}
\end{equation}
\end{defn}


\begin{rem}
We have implicitly used Convention \ref{conv:nomult} in Definition \ref{defn:Qsesquilinearform}, phrasing our definition as though there were at most one edge between two vertices. If there are multiple edges $e_1, \ldots, e_d$ from $v$ to $w$, one should instead set $\langle \alpha_v, \alpha_w \rangle_Q$ to be $- \sum [\rightdual{\Pi_{e_i}}]$, etcetera. This is equivalent to first applying Lemma \ref{lem:directsum} to merge the edges into one, and then using the definition above. Note that Convention \ref{conv:nozero} has no practical effect on the definition of this form. 
\end{rem}

When $Q$ is a fusion quiver over $\FC = \Vect$, we have $[\Vect] \cong \ZZ$ and
$\langle - , - \rangle_Q$ is the usual bilinear form on $\ZZ^V$ associated to $Q$ when
viewed as an ordinary quiver (see e.g.\ \cite[Chapter VII.4]{ASS_elerep}).
However, unlike in the case of ordinary quivers, we warn the reader that in general the form $\langle -,- \rangle_Q$ is \emph{not necessarily symmetric} and it depends on the orientation (and of course, the arrow labels) of $Q$. It is only symmetric when all the objects which label edges are self dual.

\begin{defn}
Recall that the Grothendieck group of $\Rep_{\MC}(Q)$ is identified with $[\MC]^V = [\MC] \ot \ZZ^V$, and that $[\MC]$ is a left $[\FC]$-module. We can extend $\langle - , - \rangle_Q$ to an $\ZZ$-bilinear pairing 
\[ \langle - , - \rangle_{Q,\MC} \colon \ZZ^V \ot [\MC]^V \to [\MC], \] which is uniquely determined by
\begin{equation} \langle \alpha_v , \sum_{v'} m_{v'} \alpha_{v'} \rangle_{Q,\MC} := \sum_{v'} \langle \alpha_v , \alpha_{v'} \rangle_{Q} \cdot m_{v'}. \end{equation}
Eventually we shall omit $\MC$ from the subscript when the context is clear.
\end{defn}

Equivalently, the pairing $\langle - , - \rangle_{Q,\MC}$ comes from the composition of the form $\langle - , - \rangle_Q$ and the $\FC$-action, to whit \[ \ZZ^V \ot [\MC]^V \cong
\ZZ^V \ot \ZZ^V \ot [\MC] \to [\FC] \ot [\MC] \to [\MC]. \] Note that there is no bilinear form where both the left and right inputs are $[\MC]^V$. In Remark \ref{rem:extralinearity} we will discuss
generalizations of this pairing, but first we introduce the reflections which motivate our definition.

Let $v$ be a sink or a source of $Q$.
Recall from Definition \ref{defn:reflectquivers} that we can define the reflected fusion quiver $\sigma_v Q$ with the same set of vertices $V$.
A priori the $\ZZ$-bilinear forms $\langle -, - \rangle_{Q}$ and $\langle -, - \rangle_{\sigma_v Q}$ on $\ZZ^V$ could be different, as the orientation and labels of $\sigma_v Q$ are different from $Q$.
Nonetheless, an easy comparison shows that the two forms agree, using the fact that in a fusion category we have the following isomorphisms:
\begin{itemize}
\item $\leftdual{\Pi_e} \cong \rightdual{\Pi_e}$;
\item $\leftdual{\leftdual{\Pi_e}} \cong \Pi_e \cong \tensor{\Pi}{_e^{**}}$; and
\item $\rightdual{(\leftdual{\Pi_e})} \cong \Pi_e \cong \leftdual{(\rightdual{\Pi_e})}$.
\end{itemize}
These isomorphisms exist \cite[Proposition 4.8.1]{EGNO}, even though there may not exist a natural isomorphism from the identity functor to the double dual that preserves the monoidal structures (i.e.\ $\FC$ might not be pivotal \cite[Definition 4.7.7]{EGNO}); it remains an open problem whether every fusion category has a pivotal structure \cite[Question 4.8.4]{EGNO}.

\begin{lem}
Let $Q$ be a fusion quiver over $\FC$ and let $V$ be its set of vertices.
Suppose $v$ is a sink or a source of $Q$ and $\sigma_v Q$ is the reflected fusion quiver at the vertex $v$. Then we have
\[
\langle -, - \rangle_Q = \langle -, - \rangle_{\sigma_v Q},
\]
an equality of two $[\FC]$-valued $\ZZ$-bilinear forms on $\ZZ^V$.
Similarly, one has an equality $\langle -, - \rangle_{Q,\MC} = \langle -, - \rangle_{\sigma_v Q,\MC}$ for any $\FC$-module category $\MC$.
\end{lem}

We are now ready to define the reflection group associated to $Q$.
\begin{defn}\label{defn:reflectiongroupQ}
Let $Q$ be a fusion quiver over $\FC$ with set of vertices $V$, let $\MC$ be an $\FC$-module category, and let $\langle -, - \rangle_{Q,\MC}$ be the pairing above.
To each $v\in V$, we associate a reflection $\sigma^Q_v$ on $[\MC]^V$ defined by
\begin{equation} \label{eqn:fusionreflection}
\sigma^Q_v(m) := m - \langle \alpha_v, m \rangle_{Q,\MC} \alpha_v.
\end{equation}
\end{defn}

By choosing some ordering on $V$ so that the vertices are relabelled by $\{1,2,..., |V|\}$, we can view elements in $[\MC]^V$ as column vectors and the reflections in \eqref{eqn:fusionreflection} as $|V|\times |V|$ matrices with entries in $[\FC]$.
In particular, the action of $\sigma^Q_v$ on $[\MC]^V$ is given by \textbf{left multiplication} with the following matrix
\begin{equation} \label{eqn:matrixfusionreflection}
\sigma^Q_v  = \id_{n} - 
	\begin{bmatrix}
	0 & \cdots & 0 \\
	\vdots & \ddots & \vdots \\
	0 & \cdots & 0 \\
	\langle\alpha_v, \alpha_1\rangle_Q & \cdots & \langle \alpha_v, \alpha_n \rangle_Q \\
	0 & \cdots & 0 \\
	\vdots & \ddots & \vdots \\
	0 & \cdots & 0 \\
	\end{bmatrix} \in \GL_{|V|}([\FC]),
\end{equation}
 where it is understood that the multiplication $c \cdot m_i$ for $c \in [\FC]$ and $m_i \in [\MC]$ is given by the action of $[\FC]$ on $[\MC]$.
The group of such matrices generated by all the $\sigma^Q_v$ is denoted by $\WW(Q) \subseteq \GL_{|V|}([\FC])$.


\begin{rem} \label{rem:extralinearity} 
We reiterate (see Remark \ref{rmk:repismodulecat}) that there is no left $\FC$-action on $\Rep_\MC(Q)$. The action of $\GL_{|V|}(C)$ on $[\Rep_\MC(Q)] \cong [\MC]^V$ comes from vertex-wise tensor product and linear operations that do not extend to the category.
However, the action of $\GL_{|V|}(C)$ does commute with the left $\FC^*_\MC$-action; in the particular  case where $\MC=\FC$, it commutes with the right $\FC$-action, so that the action of $\GL_{|V|}(C)$ is right $[\FC]$-linear.
\end{rem}

\begin{rem} \label{rem:sesquilinear} One could can formally extend $\langle -, - \rangle_Q$ to a form $[\FC]^V \ot [\FC]^V \to [\FC]$ in several ways. Based on the previous remark, it does make sense to keep track of the right $[\FC]$-action on the second input, though not the left $[\FC]$-action. In fact, it is natural to keep track of the left $[\FC]$-action on the first input, but in an antilinear fashion. A \emph{$[\FC]$-sesquilinear form} $B_\FC(-,-)$ on $[\FC]^V$ is a $\ZZ$-bilinear form valued in $[\FC]$
\[
B_\FC(-,-): [\FC]^V \times [\FC]^V \to [\FC]
\]
satisfying $B_\FC([X]\alpha_v , [Y]\alpha_w) = [\rightdual{X}] \cdot B_\FC( \alpha_v , \alpha_w) \cdot [Y]$ for all $v,w \in V$ and all objects $X,Y$ in $\FC$. The sesquilinear form obtained by extension from $\langle -, - \rangle_Q$ does not play a role in this paper, but we note that it coincides with the Euler pairing on $\Rep_{\FC}(Q)$ defined using internal Hom and Ext (where we view $\Rep_{\FC}(Q)$ as a right $\FC$-module category).
\end{rem}

\subsection{Reflection functors and the Grothendieck group} \label{sec:reflectionsandreflections}

For the rest of this section let $\MC$ be an indecomposable, finite semisimple left module category over $\FC$, so that $[\MC]$ is a (irreducible) left $[\FC]$-module.

\begin{notation} When $Q$ has no loops or cycles, neither does $\sigma_v Q$. For the rest of the paper, whenever $Q$ has no loops or cycles we identify $[\Rep_{\MC} Q]$ and $[\Rep_{\MC} \sigma_v Q]$, via their common identification with $[\MC]^V = [\MC]\otimes \ZZ^V$ (see Notation \ref{notn:rootspace}). 
\end{notation}



\begin{prop} \label{prop:reflectiondim} Let $v \in V$ be a sink, and $(X,\phi)$ be such that $\xi_{(X,\phi)}$ is surjective. After identifying $[\Rep_{\MC} Q]$ with $[\Rep_{\MC} \sigma_v Q]$, we have
\begin{equation} \udim(R_v^+ (X,\phi)) = \sigma^Q_v(\udim(X,\phi)). \end{equation}
Let $v \in V$ be a source, and $(X,\phi)$ be such that $\theta_{(X,\phi)}$ is injective. After identifying $[\Rep_{\MC} Q]$ with $[\Rep_{\MC} \sigma_v Q]$, we have
\begin{equation} \udim(R_v^- (X,\phi)) = \sigma^Q_v(\udim(X,\phi)). \end{equation} \end{prop}

\begin{proof} Let $v$ be a sink and set $(Y,\psi) \coloneqq R_v^+ (X,\phi)$. The dimension vectors of $(Y,\psi)$ and $(X,\phi)$ disagree only at vertex $v$, so 
\[
\udim((X,\phi)) - \udim((Y,\psi)) = ([X_v] - [Y_v])\alpha_v.
\]
We need to show that this multiple of $\alpha_v$ is $\sum_{i \in V} (\langle \alpha_v,\alpha_i \rangle_Q \cdot [X_i])\alpha_v$. 
But this follows from the short exact sequence defining the representation $(Y,\psi)$ at vertex $v$:
\[
\begin{tikzcd}
0 \arrow[r] & Y_v \arrow[r] & \bigoplus\limits_{\substack{e \in E \\ t(e) = v}} \Pi_e \ot X_{s(e)} \arrow[r,"\xi_{(X,\phi)}"] & X_v \arrow[r] & 0
\end{tikzcd}
\]
($\xi_{(X,\phi)}$ is surjective by assumption) and the definition of the form $\langle -,- \rangle_Q$ in \eqref{eqn:Qsesquilinearform}.
Similar arguments suffice when $v$ is a source. \end{proof}

We must emphasize that the reflection functors $R_v^+$ or $R_v^-$ do \emph{not} act on the Grothendieck group by the reflection $\sigma^Q_v$. The above proposition states that they act on the symbols of certain objects in a fashion agreeing with $\sigma^Q_v$. Note that $\sigma^Q_v([M] \alpha_v) = -[M] \alpha_v$, which is not the dimension vector of any representation.

\begin{rem} 
Readers familiar with derived categories may wish to consider the derived functors of $R_v^+$ and $R_v^-$, which will then act on the triangulated Grothendieck group of the derived category of $\Rep_{\MC} Q$ via $\sigma^Q_v$.
This follows from a similar argument for the classical case; see e.g.\ \cite[Theorem 3.10-3.12]{QuiverRep_Kirillov} or refer to \cite{Honoursthesis_Longbottom} for a detailed account.
We have chosen to ignore derived functors in this paper, because we wish to emphasize positivity properties later in the paper.
 \end{rem}

\subsection{Relations to unfolding} \label{sec:reflectiosunfolding}
Recall that we have the unfolded quiver $\check{Q}_{\MC}$ of $Q$ associated to $\MC$ (cf. \S\ref{defn:unfolding}), where $\check{Q}_{\MC}$ is viewed as a fusion quiver over $\Vect$.
Theorem \ref{thm:unfolding} gives us the identification
\[
\Rep_{\Vect}(\check{Q}_{\MC}) \cong \Rep_{\MC}(Q),
\]
so we have $[\Rep_{\Vect}(\check{Q}_{\MC})] \cong [\Rep_{\MC} Q]$.
However, $[\Rep_{\Vect}(\check{Q}_{\MC})]$ is more naturally viewed as a lattice $[\Vect]\otimes \ZZ^{\check{V}} \cong \ZZ^{\check{V}}$ over $\ZZ$.
This is the ``usual lattice'' that is associated to an ordinary quiver, which is equipped with the usual symmetric bilinear form $\langle -, - \rangle_{\check{Q}_{\MC}}: \ZZ^{\check{V}} \times \ZZ^{\check{V}} \rightarrow \ZZ$ defined by
\[
\langle \alpha_{(v, L)}, \alpha_{(w, L')} \rangle_{\check{Q}_{\MC}} = 
	\begin{cases}
		2, & \text{ if } v=w, L=L'; \\
		-\dim \Hom_{\MC}(L', \Pi_e \otimes L), &\text{ if } v \xrightarrow{\Pi_e} w \\
		-\dim \Hom_{\MC}(L, \Pi_e \otimes L'), &\text{ if } w \xrightarrow{\Pi_e} v \\
		0, & \text{ otherwise}.
	\end{cases}
\]
Moreover, we have a $\ZZ$-linear reflection group $\WW(\check{Q}_{\MC})$ associated to $\check{Q}_{\MC}$ acting on $\ZZ^{\check{V}}$, with generators $\sigma^{\check{Q}_{\MC}}_{(v,L)}$ acting by
\[
\sigma^{\check{Q}_{\MC}}_{(v,L)}(x) = x - \langle \alpha_{(v, L)}, x \rangle_{\check{Q}_{\MC}} \alpha_{(v,L)}.
\]

The two $\ZZ$-linear actions by the groups $\WW(Q)$ and $\WW(\check{Q}_{\MC})$ on the isomorphic abelian groups $[\MC]^V \cong \ZZ^{\check{V}}$ can be related as follows.
\begin{prop}\label{prop:WQaction=unfoldaction}
Through the natural identification $[\MC]^V \cong \ZZ^{\check{V}}$, the action of $\sigma^Q_v \in \WW(Q)$ on $[\MC]^V$ agrees with the action of $\prod_{L \in \Irr(\MC)} \sigma^{\check{Q}_{\MC}}_{(v,L)} \in \WW(\check{Q}_{\MC})$ on $\ZZ^{\check{V}}$.
\end{prop}
\begin{proof}
For a fix $v \in V$, note that the elements $\sigma^{\check{Q}_{\MC}}_{(v,L)}$ for different $L \in \Irr(\MC)$ pair-wise commute, so the product is well-defined.
Moreover, for each $\alpha_{(w,L')}$ we have
\[
\prod_{L \in \Irr(\MC)} \sigma^{\check{Q}_{\MC}}_{(v,L)}(\alpha_{(w,L')}) = \alpha_{(w,L')} - \sum_{L \in \Irr(\MC)} \langle \alpha_{(v, L)}, \alpha_{(w,L')} \rangle_{\check{Q}_{\MC}} \alpha_{(v,L)}
\]
The sum is easily computed as follows:
\begin{align*}
& \sum_{L \in \Irr(\MC)} \langle \alpha_{(v, L)}, \alpha_{(w,L')} \rangle_{\check{Q}_{\MC}}\alpha_{(v,L)} \\
& = 
	\begin{cases}
	2\alpha_{(w,L')}, &\text{ if } v=w; \\
	- \sum_{L \in \Irr{\MC}} \dim \Hom_{\MC}(L', \Pi_e \otimes L)\alpha_{(v,L)}, &\text{ if } v \xrightarrow{\Pi_e} w \\
	- \sum_{L \in \Irr{\MC}} \dim \Hom_{\MC}(L, \Pi_e \otimes L')\alpha_{(v,L)}, &\text{ if } w \xrightarrow{\Pi_e} v \\
	0, & \text{ otherwise}.
	\end{cases}
\end{align*}
On the other hand, the action of $\sigma^Q_v$ for each $[L']\alpha_w$ is given by
\[
\sigma^Q_v([L']\alpha_w) = [L']\alpha_w - \langle \alpha_v, \alpha_w \rangle_Q [L']\alpha_v,
\]
where
\[
\langle \alpha_v, \alpha_w \rangle_Q [L']\alpha_v
=
	\begin{cases}
	2[L']\alpha_w, &\text{ if } v=w; \\
	- \rightdual{[\Pi_e]} \cdot [L']\alpha_v, &\text{ if } v \xrightarrow{\Pi_e} w \\
	- [\Pi_e] \cdot [L']\alpha_v, &\text{ if } w \xrightarrow{\Pi_e} v \\
	0, &\text{ otherwise}.
	\end{cases}
\]
The fact they agree follows from the identities
\begin{align*}
\rightdual{\Pi_e} \otimes L' &\cong \bigoplus_{L \in \Irr{\MC}} L \boxtimes \Hom_{\MC}(L, \rightdual{\Pi_e} \otimes L') \cong \bigoplus_{L \in \Irr{\MC}} L \boxtimes \Hom_{\MC}(L', \Pi_e \otimes L); \\
\Pi_e \otimes L' &\cong \bigoplus_{L \in \Irr{\MC}} L \boxtimes \Hom_{\MC}(L, \Pi_e \otimes L'). \qedhere
\end{align*}
\end{proof}

This relation between reflections lifts to the categorical level as follows.

\begin{prop} \label{prop:unfoldedreflectionfunctors}
Let $F: \Rep_{\MC}Q \rightarrow \Rep_{\Vect}\check{Q}$ denote the equivalence obtained in Theorem \ref{thm:unfolding}. 
Then we have that
$
F\circ R^{\pm}_v \cong \left( \prod_{L \in \Irr(\MC)} R^{\pm}_{(v,L)} \right) \circ F.
$
\end{prop}

\begin{proof} (Sketch)
Note that the composition-product $\prod_{L \in \Irr(\MC)} R^{\pm}_{(v,L)}$ is well-defined since $R^{\pm}_{(v,L)}$ for different $L$ pair-wise commute.
The construction of the isomorphism and verification of its properties is straightforward, following directly from the semi-simplicity of $\MC$ and the kernel/cokernel definition of the reflection functors. We leave the details to the reader. 
\end{proof}

\begin{ex}\label{ex:reflectandunfoldS3}
We return to Example \ref{ex:reflectS3}. We already noted that
\[ \left( \; \begin{tikzcd}
	\CC \oplus V_2 & \CC \oplus V_2
	\arrow["V_2", from=1-1, to=1-2]
\end{tikzcd} \; \right) \quad \xmapsto{R^+_v} \quad \left( \; \begin{tikzcd}
	\CC \oplus V_2 & \Sign \oplus V_2
	\arrow["V_2", from=1-2, to=1-1]
\end{tikzcd} \; \right) \]
under the reflection functor at the unique sink. See Example \ref{ex:Snunfold} for discussion of unfolding in this example. Applying the unfolding functor $F$ to both sides, we see that
\begin{equation} \begin{tikzcd}
	\CC &&& \CC \\
	& \CC & \CC \\
	0 &&& 0
	\arrow[from=1-1, to=2-2, "\neq 0"]
	\arrow[from=3-1, to=2-2]
	\arrow[from=2-3, to=2-2, "\neq 0"]
	\arrow[from=2-3, to=1-4, "\neq 0"]
	\arrow[from=2-3, to=3-4]
\end{tikzcd} \; \mapsto \; \begin{tikzcd}
	\CC &&& 0 \\
	& \CC & \CC \\
	0 &&& \CC
	\arrow[from=2-2, to=1-1, "\neq 0"]
	\arrow[from=2-2, to=3-1]
	\arrow[from=2-2, to=2-3, "\neq 0"]
	\arrow[from=1-4, to=2-3]
	\arrow[from=3-4, to=2-3, "\neq 0"]
\end{tikzcd} \end{equation}
when one applies reflection functors to all three sinks. \end{ex}

\section{Two vertices one edge} \label{sec:twoone}

Our eventual understanding of finite type fusion quivers will come by relating them to Coxeter groups. Relations in a Coxeter group come from studying a pair of connected vertices
in the Coxeter graph, so understanding the quiver with two vertices and one edge is essential.

\subsection{Setup} \label{sec:setupranktwo}
Throughout this chapter, $\FC$ will denote a fusion category a $\MC$ will be a finite semisimple module category over $\FC$. 

By \cite[Theorem 2.6]{ENO}, any object in a fusion category is canonically isomorphic to its quadruple (left) dual. Moreover, by \cite[Proposition 2.1]{ENO}, any object $\Pi$ is (non-canonically) isomorphic to its double dual. For the rest of this chapter we fix an arbitrary object $\Pi$ in an arbitrary fusion category $\FC$, and fix an isomorphism between $\Pi$ and $\leftdual{(\leftdual{\Pi})}$.

We now closely examine the fusion quiver
\begin{equation} Q := \quad \begin{tikzcd}
	\bluebull & \redbull
	\arrow["\Pi", from=1-1, to=1-2]
\end{tikzcd},\end{equation}
over $\FC$. Our notation is that the edge $e$ goes from vertex $\bluea = \bluebull$ to vertex $\redb = \redbull$. Meanwhile, let
\begin{equation}Q' := \quad \begin{tikzcd}
	\bluebull & \redbull
	\arrow["\leftdual{\Pi}", from=1-2, to=1-1]
\end{tikzcd}.\end{equation}
We use the same vertex names for $Q'$, but now the edge $e'$ goes from $\redb$ to $\bluea$. Having fixed an isomorphism between $\Pi$ and its double dual, we get isomorphisms
\begin{equation} Q' \cong \sigma_{\bluea} Q \cong \sigma_{\redb} Q, \qquad Q \cong \sigma_{\bluea} Q' \cong \sigma_{\redb} Q'. \end{equation}
	
For sake of clarity, we hereby always use the sink reflection $\sigma_{\redb}$ and $R_{\redb}^+$ to go from $\Rep_{\MC} Q$ to $\Rep_{\MC} Q'$, and always use the sink reflection $\sigma_{\bluea}$ and $R_{\bluea}^+$ to go from $\Rep_{\MC} Q'$ to $\Rep_{\MC} Q$. One could use source reflections rather than sink reflections, and the analysis would be similar. 
With this convention in place, we also abusively write $\sigma_{\bluea}$ and $\sigma_{\redb}$ for the linear operators $\sigma_{\bluea}^{Q'}$ and $\sigma_{\redb}^Q$ from Definition \ref{defn:reflectiongroupQ} respectively.

\subsection{Quantum numbers and the Grothendieck group} \label{ss:twocolored}

The two $[\FC]$-linear involutions $\sigma_{\bluea}, \sigma_{\redb}\in \WW(Q) \subseteq \GL_2([\FC])$ are given by the following $2 \times 2$ matrices whose entries live in $[\FC]$:
\begin{equation} 
\sigma_{\bluea} = \mattwo{-1 & \left[\leftdual{\Pi}\right] \\ 0 & 1},
\qquad  \sigma_{\redb} = \mattwo{1 & 0 \\ \left[\Pi\right] & -1}.
\end{equation}
These matrices are written with respect to the ordered basis $\{\alpha_{\bluea}, \alpha_{\redb}\}$.
Our ultimate goal is to determine whether $Q$ has finite or infinite representation type. This ends up being dependent on whether the operator $\sigma_{\bluea} \circ \sigma_{\redb}$ has finite or infinite order, so we need to study this operator in depth.

To this end, let $d$ and $d'$ be formal variables (which do not commute) in the ring $\ZZ\langle d,d'\rangle$. We will eventually specialize to the ring $[\FC]$ via $d \mapsto [\Pi]$ and $d' \mapsto [\leftdual{\Pi}]$. By abuse of notation we use the same symbols to represent matrices over $\ZZ\langle d,d'\rangle$:
\begin{equation}
\sigma_{\bluea} = \mattwo{-1 & d' \\ 0 & 1},
\qquad  \sigma_{\redb} = \mattwo{1 & 0 \\ d & -1}. 
\end{equation}
In \cite[Appendix A]{EDihedral}, a computation of powers of $\sigma_{\bluea} \circ \sigma_{\redb}$ was given in terms of the so-called two-colored quantum numbers. However, this computation was done under the assumption that $d$ and $d'$ commute. We now give a variation on the construction of two-colored quantum numbers in \cite{EDihedral} which suffices for non-commuting variables.

\begin{rem} \label{rmk:theydontcommute}
For an example where $[\Pi]$ and $[\leftdual{\Pi}]$ do not commute, see the elements $A = [\Pi]$ and $B = [\leftdual{\Pi}]$ in \cite[Appendix A]{BMPS}. \end{rem}


Define \emph{two-colored quantum numbers} as certain elements of $\ZZ\langle d,d'\rangle$ indexed by $k \in \ZZ$,  using the following recursive formulae.
\begin{subequations}
\begin{equation} \label{eq:base} [0]_d = [0]_{d'} = 0, \qquad [1]_d = [1]_{d'} = 1, \end{equation}
\begin{equation} \label{eq:recurse} d \cdot [k]_{d'} = [k+1]_d + [k-1]_d, \qquad d' \cdot [k]_d = [k+1]_{d'} + [k-1]_{d'}. \end{equation}
\end{subequations}
The final equation holds for all $k \in \ZZ$. If one swaps $d$ and $d'$, one also swaps $[k]_d$ and $[k]_{d'}$.

\begin{ex} We have $[2]_d = d$, $[3]_d = d d' -1$, and $[4]_d = d d' d - 2d$.  \end{ex}
	
\begin{rem} Note that $[k]_d \ne [k]_{d'}$ for $k > 1$. In the commutative setting of \cite[\S A.1]{EDihedral}, $[k]_d = [k]_{d'}$ when $k$ is odd. \end{rem}

We apologize that the right subscript $d$ indicates that $[k]_d \in d \cdot \ZZ\langle d,d'\rangle$, i.e. $d$ appears on the left, but we prefer this mismatch to the use of left subscripts.

We record a symmetry of quantum numbers that we use repeatedly.

\begin{lem}\label{lem:upequalnegdown}
If $[m]_d = 0 = [m]_{d'}$ after specialization to some ring, then the following holds for all $k \in \ZZ$.
\begin{equation}\label{eq:upequalnegdown} 
[m+k]_d = -[m-k]_d, \qquad [m+k]_{d'} = -[m-k]_{d'}. 
\end{equation}
In particular, $[k]_d = -[-k]_d$ and $[k]_{d'} = -[-k]_{d'}$ always hold.
\end{lem}

\begin{proof}
The statement \eqref{eq:upequalnegdown} follows from a straightforward induction on $k$ using the recursive relation \eqref{eq:recurse}.
Set $m=0$ for the final statement.
\end{proof}
	
\begin{lem} \label{lem:ranktwocoxeterelementpowers} For any $k \ge 0$ we have
\begin{equation} \label{powermatrix} (\sigma_{\bluea} \circ \sigma_{\redb})^k = \mattwo{ \left[2k+1\right]_{d'} & -\left[2k\right]_{d'} \\ \left[2k\right]_d & -\left[2k-1\right]_d}. \end{equation}
\end{lem}

\begin{proof} This is a straightforward inductive computation. \end{proof}

If we think of $\sigma_{\bluea}$ and $\sigma_{\redb}$ as acting on a free $\ZZ\langle d,d'\rangle$-module with basis $\{\alpha_{\bluea}, \alpha_{\redb}\}$, then \eqref{powermatrix}
implies that $(\sigma_{\bluea} \circ \sigma_{\redb})^k \cdot \alpha_{\bluea} = [2k+1]_{d'} \alpha_{\bluea} + [2k]_d \alpha_{\redb}$.
	
\begin{lem} \label{lem:whenpoweriszero} We have $(\sigma_{\bluea} \circ \sigma_{\redb})^m = 1$ as matrices over $\ZZ\langle d,d'\rangle$ if and only if $[m]_d = [m]_{d'} = 0$. The same is true after specialization to any ring without 2-torsion, such as $[\FC]$ or $\End_{\ZZ}([\MC])$. \end{lem}

\begin{proof} By \eqref{powermatrix}, $(\sigma_{\bluea} \circ \sigma_{\redb})^m = 1$ if and only if
\begin{equation} \label{eq:whatyouneedforkism} [2m]_{d} = [2m]_{d'} = 0, \qquad [2m-1]_d = -1, \qquad [2m+1]_{d'} = 1. \end{equation}
Using \eqref{eq:recurse} with $k=2m$, one also deduces that
\begin{equation} [2m-1]_{d'} = -1, \qquad [2m+1]_d = 1. \end{equation} 
	
Suppose that $[m]_d = [m]_{d'} = 0$. 
Applying $k=m$ or $k = m \pm 1$ to \eqref{eq:upequalnegdown}, we obtain \eqref{eq:whatyouneedforkism}.

Now suppose \eqref{eq:whatyouneedforkism}. Then the $k=0$ and $k=1$ cases of the following equation hold:
\begin{equation} [2m-k]_d = -[k]_d, \qquad [2m-k]_{d'} = -[k]_{d'}. \end{equation}
Now we apply \eqref{eq:recurse} both forwards and backwards to deduce this equation for all $k \ge 0$. When $k=m$, we conclude that $2[m]_d = 2[m]_{d'} = 0$. After specialization to any ring without 2-torsion, we deduce that $[m]_d = [m]_{d'} = 0$. \end{proof}

Let us warn the reader of what is to come. In order to prove results about the representation type of $Q$, we need to prove the equality of several different numbers (possibly
infinity) related to the operator $\sigma_{\bluea} \circ \sigma_{\redb}$: its order when acting on $[\FC] \oplus [\FC]$, its order when acting on $[\MC] \oplus [\MC]$, its order
when acting on $[S(\bluea,L)]$ for any $L \in \Irr \MC$, its order when acting on $[S(\redb,L)]$, etcetera. 
These are all statements about various Grothendieck groups, but we unite them with statements involving Frobenius-Perron dimension. 
However, Frobenius-Perron dimension is an adequate measure of an element of a Grothendieck group only when that element is positive (or negative), so we need to carefully ensure positivity of quantum numbers in various settings.
This requires some careful analysis, performed over the next few sections. 

\subsection{Positivity of quantum numbers} \label{subsec:cruciallemma}

In this subsection, we focus on proving a technical result about quantum numbers in their specialization to $[\FC]$.
Let $[\FC]_{\geq 0}$ be the set of elements of $[\FC]$ realizable as the class of an object in $\FC$ (including zero). Such elements are called \emph{non-negative}. When expressing an element of $[\FC]$ as a vector with respect to the basis of simple objects, non-negative elements are vectors of integers with non-negative entries.
Let $[\FC]_{>0} := [\FC]_{\geq 0} \setminus \{0\}$ (i.e. we exclude the zero vector). 
Note that the only object of $\FC$ that is zero in the Grothendieck group is the zero object.
 
The following Positivity Lemma will be upgraded to a sign-coherence statement in Theorem \ref{thm:signcoherence}. Both are crucial technical tools.
 
\begin{lem}\label{lem:quantnumpostonon-neg}
Specialize $\ZZ \langle d,d'\rangle$ to $[\FC]$ via $d \mapsto [\Pi]$ and $d' \mapsto [\leftdual{\Pi}]$. Fix $m \ge 2$. Suppose that $[k]_d, [k]_{d'} \in [\FC]_{> 0}$ for all $0 < k < m$. Then $[m]_d, [m]_{d'} \in [\FC]_{\ge 0}$. \end{lem}

The Positivity Lemma \ref{lem:quantnumpostonon-neg} is relatively easy to prove with the technology of fusion quivers and reflection functors. However, it is an elementary statement
about fusion categories, and one might expect a proof which does not use the theory of fusion quivers. 
Unfortunately, we were unable to find a reference so we shall provide a proof using our setup. 
After giving the proof, we discuss the situation further.

The gist of our argument goes as follows. We continue to use the notation $Q$, $\bluea$, $\redb$ as in \ref{sec:setupranktwo}, and consider representations of $Q$ valued in $\FC$. If $(X,\phi) \in
\Rep_{\FC}(Q)$ and $v \in \{\bluea, \redb\}$, then $X_v$ is an object of $\FC$ so $[X_v] \in [\FC]_{\geq 0}$. This is the coefficient of $\alpha_v$ in $\udim(X)$. Reflection
functors $R_{\bluea}^+$ and $R_{\redb}^+$ send representations to representations, but they do not always act on dimension vectors by the operators $\sigma_{\bluea}$ and $\sigma_{\redb}$;
this is how the total derived functors act on dimension vectors, but the total derived functors do not send representations to representations and do not preserve non-negativity.
Thankfully, Proposition \ref{prop:reflectiondim} proves that reflection functors act on most indecomposable objects according to $\sigma_{\bluea}$ and $\sigma_{\redb}$. So, as in \eqref{powermatrix}, we can produce representations for which $[X_v]$ is a quantum number.

\begin{defn} \label{defn:Xkdef} 
Let $Q^{(\ell)}$ be the fusion quiver $Q$ when $\ell$ is odd, and $Q'$ when $\ell$ is even. Let $L$ be an object in $\MC$. 
Define $(X^{(1)}(L),\phi^{(1)}(L)) := S(\bluea,L) \in \Rep_{\MC}(Q^{(1)})$. For all $\ell \ge 2$, define $(X^{(\ell)}(L),\phi^{(\ell)}(L)) \in \Rep_{\MC}(Q^{(\ell)})$ recursively via
\begin{equation} (X^{(\ell)}(L),\phi^{(\ell)}(L)) := \begin{cases} R_{\redb}^+ (X^{(\ell-1)}(L),\phi^{(\ell-1)}(L)) & \text{ if $\ell$ is even,} \\ R_{\bluea}^+ (X^{(\ell-1)}(L),\phi^{(\ell-1)}(L)) & \text{ if $\ell$ is odd.} \end{cases} \end{equation}
We sometimes write $X^{(\ell)}(L)$ as shorthand for $(X^{(\ell)}(L),\phi^{(\ell)}(L))$. We also omit $L$ from the notation when $L = \one \in \FC$, that is, we write $X^{(\ell)}$ or $(X^{(\ell)},\phi^{(\ell)})$ as shorthand for $X^{(\ell)}(\one)$ living in $\Rep_{\FC}(Q^{(\ell)})$.
\end{defn}

Note that $X^{(1)}(L)$ has dimension vector $[1]_{d'} \cdot [L] \alpha_{\bluea} + [0]_d \cdot [L] \alpha_{\redb}$. 
In the remainder of this section we focus on the case $\MC = \FC$ and $L = \one$.

\begin{lem} With the same assumptions as Lemma \ref{lem:quantnumpostonon-neg}, $(X^{(\ell)}, \phi^{(\ell)})$ is indecomposable for all $1 \le \ell \le m$. Moreover, for $1 \le \ell \le m$ we have
\begin{equation} \label{eq:dimofXl} \udim X^{(\ell)} = \begin{cases} [\ell]_{d'} \alpha_{\bluea} + [\ell-1]_d \alpha_{\redb} & \text{ if $\ell$ is odd,} \\ 
[\ell-1]_{d'} \alpha_{\bluea} + [\ell]_d \alpha_{\redb} & \text{ if $\ell$ is even.}  \end{cases} \end{equation}
\end{lem}

\begin{proof} The result holds for $\ell = 1$. Now assume that $1 \le \ell \le m-1$ and the result holds for $\ell$. Since $[\ell]_d \ne 0$ and $[\ell]_{d'} \ne 0$, \eqref{eq:dimofXl} implies that $X^{(\ell)}$ is not supported at the sink of $Q^{(\ell)}$. Thus by Corollary \ref{cor:preservesindecomposable}, $X^{(\ell+1)}$ is indecomposable. By Proposition \ref{prop:ifnotSvL}(1), $\phi^{(\ell)}$ is surjective, and thus by Proposition \ref{prop:reflectiondim}, $X^{(\ell+1)}$ has dimension vector $\sigma_v(\udim X^{(\ell)})$, where $v \in \{\bluea,\redb\}$ depends on parity. By a straightforward calculation, this implies that $\udim X^{(\ell+1)}$ matches \eqref{eq:dimofXl} (e.g. see \eqref{powermatrix} when $\ell$ is odd). \end{proof}

We pause to note a consequence that we use in \S\ref{subsec:criteriondihedral}.

\begin{cor} \label{cor:Xkdistinct} With the same assumptions as Lemma \ref{lem:quantnumpostonon-neg}, $X^{(k)}$ and $X^{(\ell)}$ have distinct dimension vectors (and thus are non-isomorphic) for $1 \le k \ne \ell \le m$. \end{cor}
	
\begin{proof} If $b>a$ and $[a] = [b]$ and $[a+1] = [b+1]$, then by \eqref{eq:recurse}, quantum numbers must be
periodic with period $b-a$. Consequently, $[b-a] = 0$. We leave the subscripts and parity considerations to the reader. \end{proof}

\begin{proof}[Proof of Lemma \ref{lem:quantnumpostonon-neg}] By the previous lemma, $(X^{(m)},\phi^{(m)})$ is an indecomposable quiver representation. If $v \in \{ \bluea,\redb \}$ is the source of $Q_m$, then $\udim X^{(m)}$ has coefficient of $\alpha_v$ equal to either $[m]_d$ or $[m]_{d'}$, depending on parity.
Therefore $[m]_{d'} \in [\FC]_{\ge 0}$ if $m$ is odd, and $[m]_d \in [\FC]_{\ge 0}$ if $m$ is even.

The argument above proceeded by applying reflection functors to $S(\bluea,\one) \in \Rep_{\FC}(Q)$. To obtain the non-negativity of the other quantum number (e.g. $[m]_d$ when $m$ is
odd), instead apply reflection functors to $S(\redb,\one) \in \Rep_{\FC}(Q')$.\end{proof}

\subsection{Digression on positivity} \label{subsec:digression}

This section is a digression on the meaning of the Positivity Lemma \ref{lem:quantnumpostonon-neg} within the context of $\FC$. It is helpful but disrupts the flow, and can be skipped.

In a semisimple category, if $X, Y, Z$ are objects such that $[X]-[Y] = [Z]$, then $Y$ is a direct summand of $X$, and $Z$ is isomorphic to the complementary summand. By Schur's lemma, $Y$ is a summand of $X$ if and only if there is a surjective map $X \to Y$. To conclude, $[X] - [Y]$ is nonnegative if and only if there is a surjective map $X \to Y$.

If $\Pi$ categorifies $d$ and $\leftdual{\Pi}$ categorifies $d'$, then $\leftdual{\Pi} \ot \Pi$ categorifies $d' d$. We have
$[3]_{d'} = d' d - 1$; this is non-negative and hence categorified by an object $Y'_3$ of $\FC$ if and only if the monoidal identity $\one$ is a direct summand of $\leftdual{\Pi} \ot
\Pi$. In this case $Y'_3$ is the complementary summand. Note that $Y'_3 = 0$ is possible, when $\Pi$ is invertible.

For pedagogical reasons, let us rehash the argument that $[3]_{d'}$ is nonnegative when $[2]_d$ and $[2]_{d'}$ are positive.

\begin{lem} \label{lem:oneinsidepipi} The object $\one$ is a summand of $\leftdual{\Pi} \ot \Pi$ if and only if $\Pi$ is nonzero.\end{lem}

\begin{proof} The counit of adjunction is a map $\leftdual{\Pi} \ot \Pi \to \one$. If this counit map is nonzero then it is surjective, since $\one$
is a simple object. The adjunction axiom implies that the counit is zero if and only if the identity map of $\Pi$ is zero, if and only if $\Pi$ is zero. \end{proof}

Continuing when $\Pi$ is nonzero, $[4]_d = d \cdot [3]_{d'} - [2]_d$. This is non-negative and hence categorified by an object $Y_4$ if and only if $\Pi$ (which categorifies $[2]_d$) is a direct summand of
$\Pi \ot Y_3'$ (which categorifies $d \cdot [3]_{d'}$). The Positivity Lemma when $m=4$ is equivalent to the existence of a surjective map $\Pi \ot Y_3' \to \Pi$ whenever $Y_3'$
is nonzero. Through an application of adjunction, the inclusion map $Y_3' \to \leftdual{\Pi} \ot \Pi$ gives rise to a map $\phi \co \Pi \ot Y_3' \to \Pi$, which is nonzero if and only
if $Y_3'$ is nonzero. However, it is not obvious that $\phi$ should be surjective (as $\Pi$ need not be simple). This is the subtlety, and it continues as one iterates this process. 
For example, when proving that $[5]_{d'} = d' \cdot [4]_d - [3]_{d'}$ is nonnegative, one hopes for a surjective map to $Y_3'$, and $Y_3'$ need not be simple.

\begin{rem} Our proof of Lemma \ref{lem:quantnumpostonon-neg} constructs the desired surjective map as the map $\phi^{(\ell)}$ appearing within the fusion representation $(X^{(\ell)},\phi^{(\ell)})$, and surjectivity follows from Proposition \ref{prop:ifnotSvL}. \end{rem}

\begin{rem} Suppose for simplicity that $\Pi$ is self-dual. Then the Temperley-Lieb category acts on tensor powers of $\Pi$, with the value of the Temperley--Lieb circle equal to the categorical dimension of $\Pi$. Jones--Wenzl projectors are certain idempotents in the Temperley--Lieb category which may or may not exist depending on the value of the circle and the base field. See \cite{HaziJW} for precise results on the existence of Jones-Wenzl projectors. If the Jones--Wenzl projector on $k-1$ strands exists, then its image will categorify $[k] = [k]_d$. 
	

However, the Positivity Lemma \ref{lem:quantnumpostonon-neg} holds even when Jones-Wenzl projectors are not well-defined (under the assumption that $\FC$ is semisimple), as evidenced by the next example. \end{rem}

\begin{ex} Let $\FC = \Vect$ over a field $\Bbbk$ of characteristic $2$, and let $\Pi$ be a two-dimensional vector space. The categorical dimension of $\Pi$ is zero, and thus the Jones-Wenzl projector on two strands does not exist. Nonetheless, $\one$ is a direct summand of $\Pi \otimes \Pi$, and $[3]$ is positive. This example illustrates that categorical dimension is not a good indicator of which quantum numbers will vanish, and we will use a different measurement of dimension below. \end{ex}

\subsection{Frobenius-Perron dimension and sign-coherence of quantum numbers} \label{ss:FPdimsigncoh}
Here we recall a notion of dimension in the context of fusion categories and study its relation to quantum numbers (in $[\FC]$).
Let $\{ S_j \}_{j=0}^n$ denote the simples of $\FC$, with $S_0 = \one$.
For each object $X \in \FC$, the multiplication on $[\FC]$ defined by the tensor product $[X]\cdot[S_i] := [X \otimes S_i]$ satisfies
\begin{equation} \label{eqn:simpledecomposition}
[X] \cdot [S_i] = \sum_{i=0}^n {}_X r_{i,j} [S_j], \quad {}_X r_{i,j} \in \mathbb{N}_{\geq 0}.
\end{equation}
We define $\FPdim(X)$ to be the Frobenius--Perron eigenvalue of the non-negative matrix $({}_X r_{i,j})_{i,j \in \{0,1,...,n\}}$ and we call $\FPdim(X)$ the \emph{Frobenius--Perron dimension} of $X$.
The most important facts about $\FPdim$ are the following. By the first property, we can make sense of $\FPdim(x)$ for any $x \in [\FC]$ (not necessarily the class of an object).
\begin{prop} \label{prop:allfactsFP} Let $\FC$ be a fusion category and $X \in \FC$.
	\begin{enumerate} \item The map $\FPdim \colon \Obj(\FC) \to \RR$ descends to a ring homomorphism $\FPdim: [\FC] \to \RR$.
	\item $\FPdim(X) = \FPdim(\leftdual{X})$.
	\item $X = 0$ if and only if $\FPdim(X) = 0$, and otherwise $\FPdim(X) \ge 1$. In particular, if $x \in [\FC]_{\ge 0}$ then $\FPdim(x) = 0$ if and only if $x = 0$.
	\item If $0 < \FPdim(X) < 2$ then $\FPdim(X) = 2 \cos(\frac{\pi}{m})$ for some $m \ge 3$, and $X$ is simple. \end{enumerate} \end{prop}

\begin{proof} %
These are all well-known results in the theory of tensor (fusion) categories, so we shall just provide the appropriate references for them.
\begin{enumerate}
\item is \cite[Prop. 3.3.6 (1)]{EGNO}.
\item is \cite[Prop. 3.3.9]{EGNO}.
\item is \cite[Prop. 3.3.4 (2)]{EGNO}, together with the semisimplicity of $\FC$.
\item is \cite[Prop. 3.3.16]{EGNO}.
\end{enumerate}
A nonzero object has $\FPdim \ge 1$, so a decomposable object has $\FPdim \ge 2$.
\end{proof}
	
For the rest of this section we specialize $\ZZ \langle d,d'\rangle$ to $[\FC]$ via $d \mapsto [\Pi]$ and $d' \mapsto [\leftdual{\Pi}]$, and interpret two-colored quantum numbers in $[\FC]$. A consequence of the first three properties is the following.

\begin{cor} For any $k \ge 0$ we have $\FPdim([k]_d) = \FPdim([k]_{d'})$. If $[k]_d, [k]_{d'} \in [\FC]_{\ge 0}$ then $[k]_d = 0$ if and only if $[k]_{d'} = 0$. \end{cor}

\begin{proof} That $\FPdim([2]_d) = \FPdim([2]_{d'})$ follows from part (2) of Proposition \ref{prop:allfactsFP}, and the result for general $k$ follows from part (1). If $[k]_d$ is represented by an object $K$ and $[k]_{d'}$ by an object $K'$, then
\begin{align} \nonumber [k]_d &= 0 \iff K = 0 \iff \FPdim([k]_d) = 0 \iff \\ & \FPdim([k]_{d'}) = 0 \iff K' = 0 \iff [k]_{d'} = 0. \qedhere \end{align}
\end{proof}

Now we use the Positivity Lemma \ref{lem:quantnumpostonon-neg} to relate certain conditions on non-vanishing of quantum numbers.

\begin{lem} \label{lem:minimalnumbersthesame} Consider the following conditions on a positive integer $m$.
 \begin{itemize} \noindent
	\begin{minipage}{.3\textwidth}
	\item $[m]_d = [m]_{d'} = 0$.
	\end{minipage} \begin{minipage}{.35\textwidth}
	\item $[m]_d = 0$ or $[m]_{d'} = 0$.
	\end{minipage} \begin{minipage}{.3\textwidth}
	\item $\FPdim([m]_d) = 0$.
	\end{minipage} 	
\end{itemize}
The smallest positive integer satisfying any one of these properties will satisfy all of these properties. In particular, if no positive integer satisfies one of these properties, then no positive integer satisfies any of these properties. \end{lem}

\begin{rem}
Note that for $x \in [\FC]$ not necessarily a class of an object, $\FPdim(x)=0$ need not imply $x = 0$. Consider a fusion category $\FC$ with a non-self dual object $X$ and let $x = [X] - [\leftdual{X}] \neq 0$. Then $\FPdim(X) = \FPdim(\leftdual{X})$ and so $\FPdim(x)=0$.
\end{rem}
\begin{proof}[Proof of Lemma \ref{lem:minimalnumbersthesame}]
It is obvious that each property implies the property to its right.
 We need only prove that if $m$ is the smallest positive integer satisfying $\FPdim([m]_d) = 0$, then $[m]_d = [m]_{d'} = 0$.

The previous corollary shows that $\FPdim([k]_d) = \FPdim([k]_{d'})$ for all $k \geq 0$.
As such, our assumption on $m$ says that $\FPdim([k]_d)$ and $\FPdim([k]_{d'})$ are both non-zero for all $0<k<m$.
This in turn says that both $[k]_d$ and $[k]_{d'}$ can not be zero ($\FPdim$ is a ring homomorphism by Proposition \ref{prop:allfactsFP}).
Using Lemma \ref{lem:quantnumpostonon-neg} and induction, $[k]_d, [k]_{d'} \in [\FC]_{> 0}$ for all $0 < k < m$, and $[m]_d, [m]_{d'} \in [\FC]_{\ge 0}$. Now $\FPdim([m]_d) = \FPdim([m]_{d'}) = 0$ implies, by Proposition \ref{prop:allfactsFP} part (3), that $[m]_d = [m]_{d'} = 0$.\end{proof}

We now derive a relationship between $\FPdim(\Pi)$ and the order of $\sigma_{\bluea} \circ \sigma_{\redb} \in \WW(Q)$.

\begin{thm} \label{thm:whenfiniteorder} 
The operator $\sigma_{\bluea} \circ \sigma_{\redb} \in \WW(Q)$ has finite order if and only if $\FPdim(\Pi) < 2$.
Moreover, the order is $m$ if and only if $\FPdim(\Pi)= 2\cos(\pi/m)$. 
\end{thm}
\begin{proof} 
By Lemma \ref{lem:whenpoweriszero}, we see that $\sigma_{\bluea} \circ \sigma_{\redb}$ has finite order $m$ as an operator on $[\FC]\oplus [\FC]$ if and only if $m$ is the minimal positive integer such that $[m]_d = [m]_{d'} = 0 \in [\FC]$.
As such, applying Lemma \ref{lem:minimalnumbersthesame} we get that $\sigma_{\bluea} \circ \sigma_{\redb}$ has finite order $m$ if and only if $m$ is minimal among positive integers $k$ such that $\FPdim([k]_d)=0$.

We know that $\FPdim(\Pi)$ is a non-negative real number. There is always a complex number $q$ such that $\FPdim(\Pi) = q+q^{-1}$; we can even pick $q$ to be on the unit circle when $\FPdim(\Pi) \le 2$, and pick $q$ to be real when $\FPdim(\Pi) \ge 2$. 
Recall the ordinary quantum number
$[k]_q = q^{k-1} + q^{k-3} + \ldots + q^{1-k}$, which is always a real number under either of the two assumptions above. Ordinary quantum numbers satisfy a recursion analogous to \eqref{eq:recurse}.
Consequently, one proves by induction that
\begin{equation} \label{eq:ordquantnum=FPdim}
\FPdim([k]_d) = \FPdim([k]_{d'}) = [k]_q
\end{equation}
for all $k \in \ZZ$.

If $\FPdim(\Pi) = q+q^{-1} \ge 2$, then no ordinary quantum number $[k]_q$ vanishes. 
In other words, there is no $m>0$ such that $\FPdim([m]_d) = 0$, and therefore $\sigma_{\bluea} \circ \sigma_{\redb}$ has infinite order as required.

Conversely, suppose $0 < \FPdim(\Pi) < 2$.
Then by Proposition \ref{prop:allfactsFP} part (4), $\FPdim(\Pi) = 2\cos(\pi/n) = q+q^{-1}$ for some $n \ge 3$.
As such, $q$ has to be a primitive $2n$-th root of unity. In this case, $[n]_q = 0$ and $[k]_q \ne 0$ for $0 < k < n$. The minimality assumption on $m$ ensures that $m = n$, as desired.

The case $\FPdim(\Pi) = 0 = 2 \cos(\frac{\pi}{2})$ is trivial, since we must have $\Pi = 0$ and one quickly deduces that $\sigma_{\bluea} \circ \sigma_{\redb}$ has order $2$.
\end{proof}

Lastly, we prove the following sign-coherence result about two-color quantum numbers in $[\FC]$.
\begin{thm}\label{thm:signcoherence}
Specialize $\ZZ \langle d,d'\rangle$ to $[\FC]$ via $d \mapsto [\Pi]$ and $d' \mapsto [\leftdual{\Pi}]$.
Then there are two possibilities.
\begin{enumerate}
\item There exists $k \neq 0 \in \ZZ$ such that $[k]_d = 0$ or $[k]_{d'}=0$. Then there is some smallest positive integer $m = k$ satisfying this condition, such that
\begin{itemize}
\item for all $j \in \ZZ$, $[j]_d = 0$ $\iff$ $j$ is an integer multiple of $m$ $\iff$ $[j]_{d'} = 0$; and 
\item for $0 < j < m$, we have $[j+cm]_d, [j+cm]_{d'} \in (-1)^{c} \cdot [\FC]_{>0}$ for all $c \in \ZZ$.
\end{itemize}
\item Otherwise, $[k]_d \neq 0$ and $[k]_{d'} \neq 0$ for all $k \in \ZZ \setminus \{0\}$. In this case, $[k]_d, [k]_{d'} \in [\FC]_{>0}$ for all $k \geq 1$.
\end{enumerate}
In particular, in both cases we have $[k]_d, [k]_{d'} \in [\FC]_{>0} \cup \{0\} \cup -[\FC]_{>0}$ for all $k \in \ZZ$.
\end{thm}
\begin{proof}
Suppose that there is some $k\neq 0 \in \ZZ$ such that $[k]_d=0$ or $[k]_{d'}=0$.
Without lost of generality we can assume $k$ is positive since $[k]_\bullet = -[-k]_\bullet$ from Lemma \ref{lem:upequalnegdown}.
Then for $m = k$ the smallest such positive integer, Lemma \ref{lem:minimalnumbersthesame} says that 
\[
[m]_d=[m]_{d'}=0.
\]
For any $0 < j < m$, Lemma \ref{lem:quantnumpostonon-neg} and induction implies that $[j]_d, [j]_{d'} \in [\FC]_{> 0}$.
The rest of the properties follow from the symmetries of quantum numbers in Lemma \ref{lem:upequalnegdown}.

Otherwise we have $k=0$ being the only integer satisfying $[k]_d=[k]_{d'}=0$.
Then once again Lemma \ref{lem:quantnumpostonon-neg} and induction implies that $[k]_d, [k]_{d'} \in [\FC]_{> 0}$, this time for all $k \geq 1$. By Lemma \ref{lem:upequalnegdown}, $[k]_d, [k]_{d'} \in -[\FC]_{>0}$ for $k < 0$.
\end{proof}

\subsection{Action of quantum numbers on semisimple module categories \texorpdfstring{$\MC$}{M}} \label{subsec:posaction}

Throughout this section, let $\MC$ be a finite semisimple module category over $\FC$.
We define $[\MC]_{\ge 0}$ to be the set of elements realizable as the class of an object in $\MC$, and $[\MC]_{> 0}$ to exclude zero.

The previous chapter discusses the sign coherence of quantum numbers as elements within $[\FC]$. 
How does this affect their action on $[\MC]$, for a given $\FC$-module category $\MC$? 
Readers new to fusion categories might be surprised by the following fact, which is unintuitive from the perspective of ordinary module theory.

\begin{prop} \label{prop:notzero} 
Let $x \in [\FC]_{\ge 0} \cup -[\FC]_{\geq 0}$ and $u \in [\MC]_{> 0}$. If $x \cdot u = 0$ then $x = 0$.
\end{prop}
\begin{proof} 
The conclusion is invariant under replacing $x$ with $-x$, so we may assume $x \in [\FC]_{\geq 0}$.
Choose $X\in \FC$ and $U \in \MC$ such that $[X]=x$ and $[U]=u$. Because of semisimplicity, $X$ and $U$ are unique up to isomorphism, and $U$ is nonzero. Suppose that $x$, therefore $X$, is nonzero. By Lemma \ref{lem:oneinsidepipi}, $\one$ is a summand of $\leftdual{X} \ot X$. Tensoring with $U$, we see that $U$ is a summand of
$\leftdual{X} \ot X \ot U$, which must be nonzero. Therefore $X \ot U$ must be nonzero. 
\end{proof}

\begin{cor} \label{cor:zeroiffactionzero}
For all $k \in \ZZ$ and all $u \in [\MC]_{>0}$,
\[
[k]_d \cdot u = 0 \iff [k]_d = 0 \iff [k]_{d'} = 0 \iff [k]_{d'} \cdot u = 0.
\]
Thus $[k]_d$ and $[k]_{d'}$ act as the zero operator on $[\MC]$ (or indeed on $[L]$ for any $L \in \Irr(\MC)$) if and only if they are zero to begin with.
\end{cor}
\begin{proof}
The middle equivalence follows from Theorem \ref{thm:signcoherence}. 
The other two equivalences follow directly from Proposition \ref{prop:notzero} and the sign coherence property of quantum numbers, also shown in Theorem \ref{thm:signcoherence}.
The final statement is an obvious consequence.
\end{proof}

\begin{cor} \label{cor:orderagreeCM}
We have $(\sigma_{\bluea} \circ \sigma_{\redb})^m = 1$ as operator on $[\FC]\oplus[\FC]$ if and only if $(\sigma_{\bluea} \circ \sigma_{\redb})^m = 1$ as operator on $[\MC]\oplus[\MC]$.
\end{cor}
\begin{proof}
This follows from combining Lemma \ref{lem:whenpoweriszero} and Corollary \ref{cor:zeroiffactionzero}.
\end{proof}

\begin{cor} \label{cor:quantnumposact}
Specialize $\ZZ \langle d,d'\rangle$ to $[\FC]$ via $d \mapsto [\Pi]$ and $d' \mapsto [\leftdual{\Pi}]$. Fix a nonzero object $L$ in $\MC$. Fix $m \ge 2$ or $m = \infty$. Suppose that $[k]_d, [k]_{d'} \in [\FC]_{> 0}$ for all $0 < k < m$. 

Then for $0 < k < m$,
\[ [k]_d \cdot [L], [k]_{d'} \cdot [L] \in [\MC]_{> 0}. \]
If furthermore $m < \infty$, then
\[ [m]_d \cdot [L], [m]_{d'} \cdot [L] \in [\MC]_{\ge 0}.\] 
Moreover, if $L$ is simple then $X^{(\ell)}(L)$ is indecomposable for all $1 \le \ell \le m$, and we have
\begin{equation} \label{eq:dimofXlL} \udim X^{(\ell)}(L) = \begin{cases} [\ell]_{d'}\cdot[L] \alpha_{\bluea} + [\ell-1]_d\cdot[L] \alpha_{\redb} & \text{ if $\ell$ is odd,} \\ 
[\ell-1]_{d'}\cdot[L] \alpha_{\bluea} + [\ell]_d\cdot[L] \alpha_{\redb} & \text{ if $\ell$ is even.}  \end{cases} \end{equation}
Thus $X^{(k)}(L)$ and $X^{(\ell)}(L)$ have distinct dimension vectors (and are non-isomorphic) for $1 \le k \ne \ell \le m$. \end{cor}

\begin{proof} We can repeat the proofs of \S\ref{subsec:cruciallemma} verbatim with $X^{(\ell)}(L)$ replacing $X^{(\ell)}(\one)$. The essential new ingredient is that $[k]_d \cdot [L] \ne 0$ and $[k]_{d'}[L] \ne 0$ for $0 < k < m$, by Proposition \ref{prop:notzero}. Consequently, $X^{(\ell)}(L)$ is not supported at the sink of $Q^{(\ell)}$ for $1 \le \ell \le m-1$. 
\end{proof}

Finally, we conclude by adding one more equivalent condition to Corollary \ref{cor:orderagreeCM}. 
\begin{lem}\label{lem:orderequalonbasis}
Suppose that the orbit of $\sigma_{\bluea} \circ \sigma_{\redb}$ acting on $[L]\alpha_v$ is finite for some $v \in \{\bluea,\redb\}$. Then the size of this orbit is equal to the order of $\sigma_{\bluea} \circ \sigma_{\redb}$ acting on $[\FC] \oplus [\FC]$. Conversely, if the orbit of $[L] \alpha_v$ is infinite, then $\sigma_{\bluea} \circ \sigma_{\redb}$ has infinite order. \end{lem}

\begin{proof}
First we claim that the orbit is finite if and only if $\sigma_{\bluea} \circ \sigma_{\redb}$ has finite order. One direction is obvious. If $\sigma_{\bluea} \circ \sigma_{\redb}$ has infinite order, then Lemma \ref{lem:whenpoweriszero} and Theorem \ref{thm:signcoherence} imply that $[2k]_d, [2k]_{d'} \in [\FC]_{> 0}$ for all $k > 0$. From \eqref{powermatrix}, we see that
\begin{equation} 
	\begin{cases}
	(\sigma_a \circ \sigma_b)^k \cdot [L]\alpha_{\bluea} = [L] \alpha_{\bluea}  &\implies [2k]_d \cdot [L] = 0, \\
	(\sigma_a \circ \sigma_b)^k \cdot [L]\alpha_{\redb} = [L] \alpha_{\redb}  &\implies [2k]_{d'} \cdot [L] = 0.
	\end{cases}
\end{equation}
Since $[L] \in \MC_{> 0}$, this would imply $[2k]_{d} = 0$ or $[2k]_{d'} = 0$ by Corollary \ref{cor:zeroiffactionzero}, which is a contradiction. Thus the orbit is infinite.

It remains to show that the size of the orbit equals the order $m$ of $\sigma_{\bluea} \circ \sigma_{\redb}$, when both are finite. When $2k+1 \le m$ we have $(\sigma_{\bluea} \circ \sigma_{\redb})^k \cdot [L] \alpha_{\bluea} = \udim X^{(2k+1)}(L)$.  By Corollary \ref{cor:quantnumposact}, these dimensions are all distinct. When $m < 2k+1 \le 2m$, Theorem \ref{thm:signcoherence} implies that $(\sigma_{\bluea} \circ \sigma_{\redb})^k \cdot [L] \alpha_{\bluea}$ has coefficients in $[\MC]_{< 0}$. Thus the first value of $k$ for which $(\sigma_{\bluea} \circ \sigma_{\redb})^k \cdot [L] \alpha_{\bluea} = [L] \alpha_{\bluea}$ is possible is $k = m$. We leave the reader to treat the similar (but sign reversed) case of $[L] \alpha_{\redb}$. \end{proof}

\subsection{Criterion for finite type} \label{subsec:criteriondihedral}

We are now in position to classify for which $\Pi \in \FC$ the quiver $Q$ with two vertices is finite type.

\begin{thm}\label{thm:rank2classification}
Let $Q := \begin{tikzcd} \bluebull \ar[r, "\Pi"] & \redbull \end{tikzcd}$ be a rank two fusion quiver over $\FC$ and let $\MC$ be a finite semisimple module category over $\FC$.
Then the following are equivalent:
\begin{enumerate}
\item $\Rep_{\MC}(Q)$ has finitely many indecomposable objects up to isomorphism;
\item as an operator on $[\MC] \oplus [\MC] \cong [\Rep_{\MC}(Q)]$, $\sigma_{\bluea} \circ \sigma_{\redb}$ has finite order; 
\item as an operator on $[\FC] \oplus [\FC]$, $\sigma_{\bluea} \circ \sigma_{\redb} \in \WW(Q)$ has finite order; and
\item there exists $m \ge 2$ such that $\FPdim(\Pi) = 2\cos(\pi/m) <2$.
\end{enumerate}
Moreover, the order of $\sigma_{\bluea} \circ \sigma_{\redb}$ is exactly $m$.

In particular, the property that $\Rep_{\MC}(Q)$ has finitely many indecomposable objects only depends on the rank two fusion quiver $Q$ and is independent of the choice of module category $\MC$.
\end{thm}

\begin{proof}
The equivalence $(2) \iff (3)$ is Corollary \ref{cor:orderagreeCM}, and $(3) \iff (4)$ is Theorem \ref{thm:whenfiniteorder}. Conditions (3) and (4) are evidently independent of the choice of $\MC$. The order being $m$ also follows from the results cited. We need only prove that $(1)$ is equivalent to the other statements.
	
Suppose that $\sigma_{\bluea} \circ \sigma_{\redb}$ has infinite order (on either $[\MC] \oplus [\MC]$ or $[\FC] \oplus [\FC]$). By Lemma \ref{lem:whenpoweriszero}, all quantum numbers $[k]_d$ and $[k]_{d'}$ are nonzero for $k > 0$, and by Theorem \ref{thm:signcoherence}, $[k]_d$ and $[k]_{d'}$ live in $[\FC]_{> 0}$. By Corollary \ref{cor:quantnumposact}, $X^{(k)}(L)$ for $k > 0$ odd is an infinite family of non-isomorphic indecomposable objects of $\Rep_{\MC}(Q)$.

Now suppose that $\sigma_{\bluea} \circ \sigma_{\redb}$ has finite order. We will prove that the unfolded quiver $\check{Q}_{\MC}$ has finitely many indecomposable objects, which suffices since $\Rep_{\MC}(Q) \cong \Rep_{\Vect}(\check{Q}_{\MC})$ by Theorem \ref{thm:unfolding}. 
Let us be somewhat brief here, as we will elaborate on this argument in the next chapter.

By Proposition \ref{prop:WQaction=unfoldaction}, the action of $\sigma_{\bluea} \circ \sigma_{\redb}$ on $[\Rep_{\MC}(Q)]$ agrees with the action of an operator $\check{c}$ on
$[\Rep_{\Vect}(\check{Q}_{\MC})]$. Here, $\check{c}$ is a product of reflections associated to the vertices of $\check{Q}_{\MC}$, each appearing once in a particular order.
In particular, $\check{c}$ has finite order. It is well-known that when $\check{Q}_{\MC}$ is not simply-laced (i.e. has decomposable edges) then $\check{c}$ has infinite order, whereas when $\check{Q}_{\MC}$ is simply laced, then $\check{c}$ is a Coxeter element of the Coxeter group whose Coxeter graph is (the underlying unoriented graph of) $\check{Q}_{\MC}$. For a Coxeter element of a simply-laced graph to have finite order, the Coxeter graph must be a disjoint union of finite type $ADE$ Dynkin diagrams. By Gabriel's theorem \cite{Gabriel_finitetype}, this implies that $\Rep_{\Vect}(\check{Q}_{\MC})$ has finitely many indecomposable objects up to isomorphism. 
 \end{proof}

We can say more about the shape of the unfolded quiver $\check{Q}_{\MC}$.

\begin{thm}\label{thm:finitetypecoxeternumber}
Suppose the equivalent conditions in Theorem \ref{thm:rank2classification} hold.
Then the unfolded (ordinary) quiver $\check{Q}_{\MC}$ is a disjoint union of (bipartite) $ADE$ quivers whose connected components all have the same Coxeter number $m$, with $m$ satisfying $\FPdim(\Pi) = 2\cos(\pi/m)$.
Conversely, if none of the equivalent conditions in Theorem \ref{thm:rank2classification} hold, then $\check{Q}_{\MC}$ is a disjoint union of quivers whose connected components are all of infinite type.
\end{thm}

\begin{proof}
We continue with the notation of the previous proof.  The isomorphism $[\Rep_{\MC}(Q)] \to [\Rep_{\Vect}(\check{Q}_{\MC})]$ sends $[L] \cdot \alpha_v \mapsto \alpha_{(v,L)}$, for $v \in \{\bluea, \redb\}$ and $L \in \Irr \MC$. Note that $\check{c}$ preserves the span of $\alpha_{(v,L)}$ for vertices $(v,L) \in \check{Q}_{\MC}$ living in any given connected component $\Gamma$. By well-known facts about ordinary quivers, if the order of $\check{c}$ on this span is $m < \infty$, then $\Gamma$ is a finite $ADE$ quiver whose underlying graph has Coxeter number $m$. If the order of $\check{c}$ on this span is infinite, then $\Gamma$ has infinite type. So it remains to prove that the order of $\check{c}$ on the span of vertices in each individual component $\Gamma$ agrees with the order of $\sigma_{\bluea} \sigma_{\redb}$ on $[\FC] \oplus [\FC]$ (whether finite or infinite).

The order of $\sigma_{\bluea} \sigma_{\redb}$ acting on $[L] \cdot \alpha_v$ is the order of $\check{c}$ on $\alpha_{(v,L)}$, for any given $v$ and $L$. That this order agrees with the order of $\sigma_{\bluea} \sigma_{\redb}$ is precisely Lemma \ref{lem:orderequalonbasis}. 
It follows from the equivalent conditions in Theorem \ref{thm:rank2classification} that when the order $m$ of $\sigma_{\bluea} \sigma_{\redb}$ is finite, it satisfies $\FPdim(\Pi) = 2\cos(\pi/m)$. 
\end{proof}

\subsection{A generalized quantum McKay correspondence} \label{sec:McKaycorrespondence}

Inspired by \cite{AR_mckay}, we prove a generalized version of the McKay correspondence for fusion categories.
Let us first recall the definition of McKay quivers (in the context of finite semisimple module categories over fusion categories) and more importantly for us, the separated variant.
\begin{defn}\label{defn:McKayquiver}
Let $\MC$ be a finite semisimple module category over $\FC$.
The \emph{McKay quiver} of $\MC$ with respect to $\Pi \in \FC$ is the (ordinary) quiver defined as follows:
\begin{itemize}
\item its set of vertices is given by $\Irr(\MC)$; and
\item for each $L, L' \in \Irr(\MC)$, we have $m = \dim(\Hom_{\MC}(\Pi \otimes L, L'))$ arrows from $L$ to $L'$.
\end{itemize}
The \emph{separated McKay quiver} of $\MC$ with respect to $\Pi$ is the bipartite quiver obtained from the McKay quiver by doubling the number of vertices, and keeping the same number of edges, so that all directed edges now go from the left copy of its source to the right copy of its target.
\end{defn}

The following is immediate from the definitions.
\begin{lem}\label{lem:sepMcKayandunfld}
Let $\MC$ be a finite semisimple module category over $\FC$.
With $Q$ the fusion quiver 
$
\begin{tikzcd}
	\bullet & \bullet
	\arrow["\Pi", from=1-1, to=1-2]
\end{tikzcd}
$, the unfolded quiver $\check{Q}_{\MC}$ agrees with the separated McKay quiver of $\MC$ with respect to $\Pi \in \FC$.
\end{lem}
\begin{rem}
The (non-separated) McKay quiver of $\MC$ with respect to $\Pi$ is the unfolded quiver $\check{Q}_\MC$ of the fusion quiver $Q$ that has a single vertex and a single loop edge on that vertex labeled by $\Pi$.
\end{rem}

While Auslander--Reiten's generalized McKay correspondence \cite[Theorem 1]{AR_mckay} generalizes the McKay correspondence for finite subgroups of $SU(2)$ to \emph{arbitrary} finite groups, the following result generalizes the quantum McKay correspondence for quantum $SU(2)$ \cite{KO_qmckay, ostrik_weakhopf} to \emph{arbitrary} fusion categories.

\begin{cor}[Generalized quantum McKay correspondence]
Let $\FC$ be a fusion category and $\MC$ be a finite semisimple module category over $\FC$.
The separated Mckay quiver of $\MC$ with respect to $\Pi \in \FC$ is a disjoint union of $ADE$ quivers with the same Coxeter number $m$ if and only if $\FPdim(\Pi) = 2\cos(\pi/m) <2$.
\end{cor}
\begin{proof}
This follows directly from Lemma \ref{lem:sepMcKayandunfld} together with Theorem \ref{thm:rank2classification} and \ref{thm:finitetypecoxeternumber}.
\end{proof}

\section{The classification theorem} \label{sec:generalindecomp}
Throughout this section, $Q$ will be a fusion quiver over $\FC$ and $\MC$ will always be a finite semisimple left module category over $\FC$.

In this section we prove our main result as Theorem \ref{thm:CoxeterDynkinifffinite}, that each connected component of $Q$ being Coxeter--Dynkin type (which depends only on the underlying quiver and the Frobenius-Perron
dimensions of the edge labels, see Definition \ref{defn:coxdyntype}) is a necessary and sufficient condition for $\Rep_{\MC}(Q)$ to have only finitely many indecomposable objects up to isomorphism. This finite-type property is therefore independent of the choice of $\MC$.

We outline the proof after eliminating the need to consider loops or cycles.

\begin{lem} If $Q$ has a loop or a cycle (where all edge labels are nonzero objects), then so does $\check{Q}_{\MC}$ for any $\MC$. \end{lem}
	
\begin{proof} For every edge $e$ in $Q$, and every vertex $v$ in $\check{Q}_{\MC}$ lifting $s(e)$, there is an edge $\check{e}$ lifting $e$. Consequently, every path in $Q$ lifts to (at least one) path in $\check{Q}_{\MC}$. A loop or a cycle gives rise to an infinite path in $Q$, which lifts to an infinite path in $\check{Q}_{\MC}$; since $\check{Q}_{\MC}$ has finitely many vertices, this lifted path must contain a loop or a cycle. \end{proof}
	
\begin{cor} If $Q$ has a loop or a cycle (where all edge labels are nonzero objects), then $\Rep_{\MC}(Q)$ will necessarily have infinitely many (non-isomorphic) indecomposable objects. \end{cor}

\begin{proof} Recall that $\Rep_{\MC}(Q) \cong \Rep_{\Vect}(\check{Q}_{\MC})$ by Theorem \ref{thm:unfoldingnon-canonical}, and $\check{Q}_{\MC}$ is an ordinary quiver with a loop or a cycle. It is known that any quiver containing a loop or a cycle has infinitely many indecomposable representations (see e.g.\ \cite[Proposition 5.3.2]{krause2010representations}). \end{proof}

Hence, as far as the classification is concerned, we may assume that $Q$ has no loops or cycles -- and we shall assume so for the rest of this section.
In this case, we associate to each fusion quiver $Q$ a Coxeter graph $\Gamma_Q$ and its corresponding Coxeter group $\WW(\Gamma_Q)$ -- this generalizes the classical case where a Weyl group is associated to a quiver.
In particular, when applied to $\check{Q}_{\MC}$, whether $\Rep_{\Vect}(\check{Q}_{\MC})$ has finitely many indecomposable objects is completely determined by the finiteness of the group $\WW(\Gamma_Q)$ due to Gabriel's theorem \cite{Gabriel_finitetype}.
We prove our main theorem by showing that the Coxeter number of $\Gamma_Q$ (which, in the connected case, is finite if and only if $\WW(\Gamma_Q)$ is finite) agrees with the Coxeter number of $\Gamma_{\check{Q}_{\MC}}$.
This essentially follows from a Coxeter theoretic argument using our rank two result about Coxeter numbers in Theorem \ref{thm:finitetypecoxeternumber}.

\subsection{Coxeter groups associated to fusion quivers}
Recall the reflection group $\WW(Q) \subseteq \GL_{|V|}([\FC])$ from Definition \ref{defn:reflectiongroupQ}.
By applying the Frobenius--Perron map $\FPdim: [\FC] \rightarrow \RR$ to each of the entries of the matrices, we obtain a group homomorphism $\WW(Q) \rightarrow \GL_{|V|}(\RR)$.
We shall denote the image of this group homomorphism by $\WW(|Q|) \subset \GL_{|V|}(\RR)$ and the image of the matrix $\sigma^Q_v \in \GL_{|V|}([\FC])$ is denoted by $\sigma^{|Q|}_v \in \GL_{|V|}(\RR)$.

An equivalent interpretation of $\WW(|Q|)$ is as follows.
We could view $\RR$ as a left $[\FC]$-module and obtain $[\FC]^V \otimes_{[\FC]} \RR \cong \RR^V$ by applying $\FPdim$ component wise.
The bilinear form $\langle -,- \rangle_Q$ on $\ZZ^V$ valued in $[\FC]$ then naturally extends to a $\RR$-bilinear form 
\[
\langle -,- \rangle_{|Q|} : \RR^V \otimes \RR^V \rightarrow \RR,
\]
which is obtained by applying $\FPdim$ to \eqref{eqn:Qsesquilinearform}.
If we abuse notation and also use $\alpha_v$ to denote the standard basis elements of $\RR^V$, the $\RR$-bilinear form is defined by

\begin{equation}
\langle \alpha_v, \alpha_w \rangle_{|Q|} := 
	\begin{cases} 
		2, & \text{ if } v = w, \\
		- \FPdim(\rightdual{\Pi_e}), & \text{ if  $v\neq w$ and $v \xrightarrow{\Pi_e} w$}, \\
		- \FPdim(\Pi_e), & \text{ if $v \neq w$ and $w \xrightarrow{\Pi_e} v$}, \\
		0 & \text{ if $v \ne w$ and no edge abuts both $v$ and $w$}. 
	\end{cases}
\end{equation}

It is clear that the reflection group generated by the reflections
\[
m \mapsto m - \langle \alpha_v, m \rangle_{|Q|} \alpha_v
\]
is exactly the group $\WW(|Q|)$.
Note that the $\RR$-bilinear form $\langle -,- \rangle_{|Q|}$ is always symmetric, as by Proposition \ref{prop:allfactsFP}(2) we have
\[
\FPdim(\Pi_e) = \FPdim(\rightdual{\Pi_e}) \quad (= \FPdim(\leftdual{\Pi_e})).
\]
As such, the $\RR$-bilinear form $\langle -,- \rangle_{|Q|}$ and the group $\WW(|Q|)$ only depend on the underlying labeled graph of $Q$, as we now define it.

\begin{defn}\label{defn:graphofQ}
A \emph{labeled graph} will be an unoriented graph whose edges are labeled by non-negative real numbers. (By convention, if an edge is labeled by zero, one can remove it to get an equivalent labeled graph.)

Let $Q$ be a fusion quiver.  The \emph{underlying labeled graph} of $Q$, denoted $|Q|$, is obtained from $Q$ by forgetting the orientation, and labeling each edge $e$ with $\FPdim(\Pi_e)$.
\end{defn}

If $v$ is a sink or source of $Q$, then the underlying labeled graphs of $Q$ and $\sigma_v Q$ agree. Recall from Proposition \ref{prop:allfactsFP}(4) that if an edge label in $|Q|$ is less than $2$, then it is equal to $2 \cos(\pi/m)$ for some $m \ge 2$.

\begin{defn}
Let $Q$ be a fusion quiver with no loops or cycles, and let $|Q|$ be its underlying labelled graph.
Consider the graph obtained from $|Q|$ by replacing edge labels of $|Q|$ (which are real numbers) with edge labels in $\{2,3,\ldots\} \cup \{\infty\}$ as follows. Any edge label $2 \cos(\pi/m)$ is replaced by the label $m$ for $m \ge 2$, and any edge label which is at least $2$ is replaced by $\infty$. We call this the \emph{Coxeter graph of $Q$}, and denote it by $\Gamma_Q$. We use the standard conventions for Coxeter graphs -- we omit edges labeled with $2$, and omit the label on edges labeled with $3$. We also let $\WW(\Gamma_Q)$ be the corresponding Coxeter group, with simple reflections denoted $\sigma_v$ for $v \in V$. 
\end{defn}

\begin{defn}
\label{defn:coxdyntype}
Let $Q$ be a fusion quiver with no loops or cycles. 
We say that $Q$ is of \emph{(connected) Coxeter--Dynkin type} if $\Gamma_Q$ is one of the Coxeter--Dynkin diagrams in Figure \ref{fig:CoxDyndiagrams} (\S 1.2); equivalently, $\WW(\Gamma_Q)$ is a finite irreducible Coxeter group.
\end{defn}

We now have three groups associated to $Q$: $\WW(Q)$, $\WW(|Q|)$, and $\WW(\Gamma_Q)$, with generators $\sigma_v^Q$, $\sigma_v^{|Q|}$, and $\sigma_v$ respectively. Let us argue that these groups are all isomorphic.

\begin{lem} The map $\sigma_v \to \sigma_v^Q$ induces a surjective homomorphism $\WW(\Gamma_Q) \to \WW(Q)$. Composition with the natural map $\WW(Q) \to \WW(|Q|)$ gives a surjective homomorphism $\WW(\Gamma_Q) \to \WW(|Q|)$. \end{lem}

\begin{proof} The relations in $\WW(\Gamma_Q)$ come from edges $u \to v$ labeled by $2\cos(\pi/m) < 2$ in $|Q|$, so one need only argue that $\sigma_u^Q \sigma_v^Q$ (resp. $\sigma_u^Q
\sigma_v^Q$) has the expected order. The operator $\sigma_u^Q \sigma_v^Q$ preserves the span of $\{\alpha_u, \alpha_v\}$, and the action on this plane is the same as the one
analyzed in \S\ref{sec:twoone}; see Lemma \ref{lem:ranktwocoxeterelementpowers} and the subsequent paragraph. For any other vertex $x$, $\sigma_u^Q \sigma_v^Q$ preserves the span of $\{\alpha_u, \alpha_v, \alpha_x\}$, and the formula for $(\sigma_u^Q
\sigma_v^Q)^k \alpha_x$ is also controlled by two-colored quantum numbers, see \cite[equation (A.2)]{EDihedral}. In particular, it is proven there that the vanishing of $[m]_d$ and
$[m]_{d'}$ imply that $(\sigma_u^Q \sigma_v^Q)^m \alpha_x = \alpha_x$, as desired.  
Thus the homomorphism $\WW(\Gamma_Q) \to \WW(Q)$ is well-defined.
The homomorphisms above send generators to generators, so they are surjective. \end{proof}

\begin{lem} The homomorphism $\WW(\Gamma_Q) \to \WW(|Q|)$ is an isomorphism. \end{lem}
	
\begin{proof}
Suppose that all edge labels in $|Q|$ are less than or equal to $2$. Readers familiar with Coxeter theory may recognize that the action of $\WW(|Q|)$ on $\RR^V$ is
precisely the (standard) geometric representation of $\WW(\Gamma_Q)$, see e.g.\ \cite[\S 5.3]{HumphCox}.  The geometric representation is faithful, see e.g.\ \cite[\S 5.4]{HumphCox}, so $\WW(\Gamma_Q) \to \WW(|Q|)$ is injective. A generalization which suffices for all edge labels can be found in \cite[Theorem 4.2.7]{Comb_Coxeter_book}. \end{proof}

\begin{thm} \label{thm:coxeter=reflectiongroup} The groups $\WW(\Gamma_Q)$, $\WW(Q)$, and $\WW(|Q|)$ are all isomorphic. \end{thm}

\begin{proof} By the previous lemmas we have surjective homomorphisms $\WW(\Gamma_Q) \to \WW(Q) \to \WW(|Q|)$, whose composition is injective. Thus $\WW(\Gamma_Q) \to \WW(Q)$ is also injective, and they are all isomorphisms. \end{proof} 

\subsection{Coxeter groups and unfolding}
Recall that we have the unfolded quiver $\check{Q}_{\MC}$ of $Q$ associated to $\MC$ (cf. \S \ref{defn:unfolding}).
We now relate the two Coxeter groups $\WW(\Gamma_Q)$ and $\WW(\Gamma_{\check{Q}_{\MC}})$ for any finite semisimple module category $\MC$.

\begin{cor} \label{cor:faithfulactionunfld}
We have an injective homomorphism $\Psi: \WW(\Gamma_Q) \rightarrow \WW(\Gamma_{\check{Q}_{\MC}})$ defined on the generators by
\begin{equation} \label{eq:unfoldinggroupmorphism}
\sigma_v \mapsto \check{\sigma}_v := \prod_{L \in \Irr(\MC)} \sigma_{(v,L)}.
\end{equation}
In particular, the $\ZZ$-linear action of $\WW(Q)$ on $[\MC]^V$ (see Definition \ref{defn:reflectiongroupQ}) is faithful.
\end{cor}
\begin{proof}
We first prove that $\Psi$ is an injective group homomorphism of Coxeter groups.
To this end, we will show that the surjective map $\pi: V \times \Irr(\MC) \to V$ defined by $(v,L) \mapsto v$ is a Lusztig's partition (equivalently, a strong admissible map in the sense of Dyer \cite[\S 8.6]{Dyer_permroot}); this means:
\begin{enumerate}
\item $\sigma_{(v,L)}$ and $\sigma_{(v,L')}$ commute for all $(v,L), (v,L') \in \pi^{-1}(v) = \{v\} \times \Irr(\MC)$; \label{item:Lusztig1}
\item for each distinct $v,w \in V$, the irreducible components of the (standard) parabolic subgroup of $\WW(\Gamma_{\check{Q}_\MC})$ generated by $\pi^{-1}(v) \sqcup \pi^{-1}(w)$ all have the same Coxeter number $m$, where $m$ is the label of the edge connecting $v$ and $w$ in $\Gamma_Q$. \label{item:Lusztig2}
\end{enumerate}
This proves that the homomorphism $\Psi$ is injective since $\Psi$ agrees with the induced homomorphism associated to the partition $\pi$ above (see \cite[Definition 1.2]{EHCoxEmb}), where it follows from \cite[Theorem 1.3 and 2.4]{EHCoxEmb} (or \cite[\S 8]{Dyer_permroot}) that this induced homomorphism is injective.
We conclude by observing that condition \eqref{item:Lusztig1} follows the fact that no two vertices of the form $(v,L)$ and $(v,L')$ are connected by an edge in $\Gamma_{\check{Q}_\MC}$ (cf.\ Definition \ref{defn:unfolding}), whereas condition \eqref{item:Lusztig2} follows from Theorem \ref{thm:finitetypecoxeternumber} for rank two fusion quivers.

Now to see that $\ZZ$-linear action of $\WW(Q)$ on $[\MC]^V$ is faithful, recall that by Theorem \ref{thm:coxeter=reflectiongroup} the groups $\WW(Q)$ and $\WW(\check{Q}_{\MC})$ are isomorphic to the Coxeter groups $\WW(\Gamma_Q)$ and $\WW(\Gamma_{\check{Q}_{\MC}})$ respectively.
Moreover, the relation between $\WW(Q)$ and $\WW(\check{Q}_{\MC})$ in Proposition \ref{prop:WQaction=unfoldaction} translated to their corresponding Coxeter groups is described exactly by the group homomorphism $\Psi$ defined in the statement of this proposition.
It follows that the action of $\WW(Q)$ on $[\MC]^V$ is faithful if and only if $\Psi$ is an injective homomorphism.
%
%
\end{proof}

\subsection{Finite type if and only if Coxeter--Dynkin}\label{sec:finiteiffCD}
Now we are ready to prove the precise condition for $\Rep_{\MC}(Q)$ to have only finitely many indecomposable objects.
Recall that $Q$ is a fusion quiver over $\FC$, and $\MC$ is a finite semisimple left module category over $\FC$. Let $n$ denote the number of vertices of $Q$.
We choose an order on the vertices, relabeling them by the set $\{1,2,...,n\}$.

\begin{defn}
Let $\Gamma_Q$ be the Coxeter graph associated to $Q$.
The element 
\[
c:= \sigma_n \cdots \sigma_2\sigma_1 \in \WW(\Gamma_Q)
\]
is called the \emph{Coxeter element} (associated to the ordering).
\end{defn}

Now let us choose an arbitrary ordering on the finite set $\Irr(\MC)$. This gives rise to an induced ordering on the vertices of $\check{Q}_{\MC}$: for $i, j \in \{1, \ldots, n\}$ we have
\begin{itemize}
	\item $(i,L) < (j,L')$ if $i < j$, for any $L, L' \in \Irr(\MC)$,
	\item $(i,L) < (i,L')$ if $L < L'$ in $\Irr(\MC)$.
\end{itemize}
Using the same definition as above but now for $\check{Q}_{\MC}$ and its ordered set of vertices, we obtain a Coxeter element
\[
\check{c} :=  
	\prod_{L \in \Irr(\MC)} \sigma_{n, L} 
	\cdots  
	\prod_{L \in \Irr(\MC)} \sigma_{2, L} 
	\prod_{L \in \Irr(\MC)} \sigma_{1, L} 
\in \WW(\Gamma_{\check{Q}_{\MC}}).
\]
The vertices within each subset $\check{i}:= \{(i, L) \in \check{V} \mid L \in \Irr(\MC)\}$ are non-adjacent and thus the reflections $\sigma_{i,L}$ for $i$ fixed are pairwise commutative. Hence the order on $\Irr(\MC)$ was irrelevant, for the purposes of defining the Coxeter element $\check{c}$.
Crucially, observe that $\check{c}$ is the image of $c$ under the map $\Psi$ from Corollary \ref{cor:faithfulactionunfld}.

In the following, we use the convention that a connected infinite type Coxeter graph (i.e.\ its Coxeter group is irreducible and is an infinite group) has Coxeter number $\infty$.
\begin{thm}\label{thm:CoxeterDynkinifffinite}
Let $Q$ be a fusion quiver for which $Q$ is connected. For any finite semisimple module category $\MC$ over $\FC$, the quiver $\check{Q}_{\MC}$ is a disjoint union of quivers, with the Coxeter graph associated to each connected component having Coxeter number equal to the Coxeter number of $\Gamma_Q$. The category  $\Rep_{\MC}(Q)$ has only finitely many indecomposable objects (up to isomorphism) if and only if $\Gamma_Q$ is Coxeter--Dynkin type.
\end{thm}

\begin{proof}
It is known that a Coxeter group is finite if and only if a Coxeter element has finite order \cite{howlett_mmatrices}.
The homomorphism $\Psi$ from Corollary \ref{cor:faithfulactionunfld} sends $c \mapsto \check{c}$, and is injective.
Since an injective homomorphism preserves the order of group elements, $\WW(\Gamma_Q)$ is finite if and only if $\WW(\Gamma_{\check{Q}_{\MC}})$ is. By Gabriel's theorem \cite{Gabriel_finitetype}, $\WW(\Gamma_{\check{Q}_{\MC}})$ is finite if and only if $\Rep_{\Vect}(\check{Q}_{\MC})$ has only finitely many indecomposable objects (up to isomorphism). By Theorem \ref{thm:unfolding}, we have that $\Rep_{\Vect}(\check{Q}_{\MC}) \cong \Rep_{\MC}(Q)$ as abelian categories. Thus $\WW(\Gamma_Q)$ is finite if and only if $\Rep_{\MC}(Q)$ has finite type.

We have shown in the proof of Corollary \ref{cor:faithfulactionunfld} that the map $(v, L) \in \check{V} \mapsto v \in V$ is a Lusztig's partition on the simple reflections of $\WW(\Gamma_{\check{Q}_\MC})$.
It follows from \cite[Proposition 2.5]{EHCoxEmb} (or \cite[Lemma 8.9(c)]{Dyer_permroot}) that the irreducible components of the Coxeter graph $\Gamma_{\check{Q}_\MC}$ have the same Coxeter number as $\Gamma_Q$.
\end{proof}

\begin{rem}
We note that \cite{EHCoxEmb} was written without knowing that its results were already proven in an earlier paper by Dyer \cite{Dyer_permroot}. Credit for the relevant results should go to Dyer, though we also cite \cite{EHCoxEmb} as it provides complete, self-contained proofs of the results specifically needed here.
More details can be found in \cite[\S 4]{EHCoxEmb}.
\end{rem}

\section{The classification of indecomposables in the case \texorpdfstring{$\MC = \FC$}{M=C}} \label{sec:M=C}
In this section, we focus on the special case where our representation is built in the fusion category itself, i.e.\ we study $\Rep_{\FC}(Q)$.
In this case, we will prove a stronger statement than the finite type classification given in Theorem \ref{thm:CoxeterDynkinifffinite}.
Namely, we will prove in Proposition \ref{prop:extrootbijection} that if $Q$ is of Coxeter--Dynkin type, the classes of indecomposable representations are in bijection with the \emph{positive extended roots} of the root system associated $Q$ (see Definition \ref{defn:fusionrootsystem}).
Moreover, every indecomposable can be obtained from the indecomposables corresponding to the positive roots by an action of some appropriate simple object of $\FC$.

Throughout this section, we assume that $Q$ has no loops or cycles.


\subsection{Root systems associated to fusion quivers}
\begin{defn}\label{defn:fusionrootsystem}
To each fusion quiver $Q$ over $\FC$, we associate a (fusion) \emph{root system} over the fusion ring $[\FC]$ as follows.
\begin{itemize}
\item The elements $\alpha_v \in [\FC]^V$ for each $v \in V$ are called the \emph{simple roots}.
\item The \emph{set of roots} is given by $\Phi:= \WW(Q)\cdot \{ \alpha_v \mid v \in V \}$.
\item The \emph{set of extended roots} is given by $\Phi^{ext} := \{ r\cdot [L] \mid r \in \Phi, L \in \Irr(\FC) \}$.
\item A non-zero element $m$ in $[\FC]^V$ is said to be \emph{positive} if it can be expressed as 
\[
m = \sum_{v \in V} \alpha_v \cdot [A_v]
\]
for $A_v \in \FC$ (not necessarily simple).
\item We denote the set of \emph{positive roots} by $\Phi_+$ and the set of \emph{positive extended roots} by $\Phi^{ext}_+$. Note that by definition we have $\Phi_+ \subseteq \Phi^{ext}_+$.
\end{itemize}
\end{defn}
\begin{rem}
We warn the reader that the set of roots $\Phi$ is \emph{not} the same as the set of elements $m \in [\FC]^V$ satisfying $\langle m,m \rangle_Q = 2$.
Nonetheless, $\Phi$ is contained in such set; every root $r$ satisfies $\langle r,r \rangle_Q = 2$.
\end{rem}
\begin{rem}
By applying $\FPdim$ to $[\FC]^V$, we obtain the root system of $\Gamma_Q$ (over $\RR$) in the sense of Deodhar \cite{deodhar_root}, where roots are mapped to roots, and the notion of positivity coincides. Note however, different extended (positive) roots may be mapped to the same root -- this happens when two simples have the same $\FPdim$.
\end{rem}

If we view $\check{Q}:= \check{Q}_{\FC}$ as a fusion quiver over $\Vect$, there is no difference between extended roots and roots -- there is only one simple in $[\Vect]$.
Moreover, the notions ``roots'' and ``positive roots'' agree with their similarly-named-notions in the usual root system associated to $\check{Q}$ viewed as an ordinary quiver.
In other words, the fusion root system of $\check{Q}$ over $[\Vect] \cong \ZZ$ really is just the usual root system.
We can relate the root system of $Q$ and the root system of $\check{Q}$ as follows.
\begin{prop}\label{prop:extrootbijection} 
Let $\check{\Phi}$ be the set of roots associated to the root system of $\check{Q}$.
Then the set of extended positive roots $\Phi^{ext}$ of $Q$ maps injectively to $\check{\Phi}$, where the map is given by forgetting the right module structure of $[\FC]^V$ over $[\FC]$ so that $[\FC]^V$ is naturally identified with $[\Vect]^{\check{V}}\cong \ZZ^{\check{V}}$ as $\ZZ$-modules.
Explicitly, the injection is given by
\begin{equation} \label{butisitsurjective}
w(\alpha_v) \cdot [L] = w(\alpha_v \cdot L) \mapsto \check{w}(\alpha_{v, L}).
\end{equation}
When $Q$ has Coxeter--Dynkin type, \eqref{butisitsurjective} is a bijection from $\Phi^{ext}$ to $\check{\Phi}$.
\end{prop}
\begin{proof}
It follows from Proposition \ref{prop:WQaction=unfoldaction} and Corollary \ref{cor:faithfulactionunfld} that the map above is well-defined and injective. 
Moreover, it is clear that it defines a bijection between the extended simple roots of $Q$ and the simple roots of $\check{Q}$.  Now assume $Q$ has Coxeter--Dynkin type.
It remains to show that all the roots in $\check{\Phi}$ are obtainable from the simple roots of $\check{Q}$ when the group action is restricted to the subgroup $\WW(\Gamma_Q) \subseteq \WW(\Gamma_{\check{Q}})$ (viewed as subgroup through \eqref{eq:unfoldinggroupmorphism}). A well-known fact about finite Coxeter groups, due to Steinberg \cite{Stein_finitereflect}, is that for any bipartite Coxeter element $c$, every root is in the $c$-orbit of some simple root. Any Coxeter--Dynkin quiver is a bipartite graph and admits a bipartite Coxeter element $c$; a bipartite Coxeter element $c$ in $\WW(\Gamma_Q)$ will be sent to a bipartite Coxeter element $\check{c}$ in $\WW(\Gamma_{\check{Q}})$. All roots in $\check{\Phi}$ can thus be obtained from a simple root using a power of $\check{c}$, and therefore using an element of $\WW(\Gamma_Q)$.
\end{proof}

\begin{ex} Although Grothendieck groups for fusion quivers with loops or cycles are more complex, we wish to point out a simple example where \eqref{butisitsurjective} fails to be surjective. Let $Q$ be a fusion quiver over $\FC = \Rep S_2$ with underlying graph $\tilde{A}_2$ and with one edge labelled by the sign representation and the remaining edges with the trivial representation. The unfolded quiver $\check{Q}$ will have underlying graph $\tilde{A}_5$. Identifying the vertices of $\check{Q}$ with $\ZZ/6\ZZ$, the six roots $\{\alpha_{i-1} + \alpha_i + \alpha_{i+1}\}_{i \in \ZZ/6\ZZ}$ form an orbit under $\WW(\Gamma_Q)$, and this orbit does not contain any simple roots of $\Gamma_{\check{Q}}$. \end{ex}

\begin{cor}\label{cor:fusionrootdichotomy}
We have $\Phi = \Phi_+ \coprod - \Phi_+$ and similarly $\Phi^{ext} = \Phi^{ext}_+ \coprod - \Phi^{ext}_+$, where $\coprod$ denotes the disjoint union.
\end{cor}
\begin{proof}
This follows directly from the observation that through the identification of $\Phi^{ext}$ as a subset of $\check{\Phi}$ in Proposition \ref{prop:extrootbijection}, the notion of positivity for both sets coincide.
\end{proof}

\subsection{Bijection between indecomposables and positive roots}

\begin{defn}
Let $Q$ be a (fusion) quiver. An \emph{admissible sink ordering} is a labeling of the vertices by $\{1, 2, \ldots, n\}$ such that
\begin{itemize} \item $1$ is a sink in $Q$,
	\item for each $i \in \{2,\ldots,n\}$, $i$ is a sink in the quiver 	$\sigma_{i-1}\cdots \sigma_2\sigma_1 Q$.
	\end{itemize}
Note that it is impossible for $Q$ to contain any loops. Moreover, $\sigma_n \cdots \sigma_1 Q = Q$, as each edge has its orientation reversed exactly twice.
\end{defn}
	
It is well-known that any quiver that has no loops or cycles admits an admissible sink ordering (see e.g.\ \cite[Lemma 3.1.1]{krause2010representations}); it follows that any fusion quiver of Coxeter--Dynkin type -- whose underlying quiver has no loops or cycles by definition -- admits an admissible sink ordering.

\begin{defn}
Let $|Q|$ be a Coxeter--Dynkin diagram and suppose $Q$ is equipped with an admissible sink ordering.
The element $c:= \sigma_n \cdots \sigma_2\sigma_1 \in \WW(Q)$ is called the \emph{Coxeter element} associated to the admissible sink ordering.
Similarly, the composition of reflection functors $C^+:= R^+_n \cdots R^+_1 : \Rep_{\FC}(Q) \rightarrow \Rep_{\FC}(Q)$ is called the \emph{Coxeter functor} associated to the ordering.
\end{defn}

The following slight generalization of a well-known lemma in Coxeter theory will be useful to us.

\begin{lem} \label{lem:coxeterpowernon-positive}
Suppose $Q$ is of Coxeter--Dynkin type equipped with an admissible sink ordering.
If $m := (m_i)_{i \in V} \in [\FC]^V$ is a positive element (cf. \ref{defn:fusionrootsystem}), then there exists $k \in \mathbb{N}$ such that $c^k(m)$ is no longer a positive element.
\end{lem}

\begin{proof}
This follows from the identification of $\Psi^{ext}$ and $\check{\Psi}$ given in \ref{prop:extrootbijection}, together with the fact that the notion of positivity for both sets agree. We implicitly used that the lemma holds for ordinary $ADE$ quivers, see e.g.\ \cite[Lemma 6.7.2]{EGHLSVY_introrep} and \cite[Lemma 4.4.4]{krause2010representations}.
\end{proof}

We are now ready to prove the main theorem of this section; recall the definition of dimension vector $\udim$ from Notation \ref{notn:rootspace}.
\begin{thm}\label{thm:gabrielroot}
Suppose $Q$ is a fusion quiver over $\FC$ which is of Coxeter--Dynkin type.
The dimension vector map $\udim$ induces a bijection between the set of indecomposable representations of $Q$ (up to isomorphism) and the set of extended positive roots $\Phi_+^{ext}$ of $[\FC]^V$.
Moreover, every indecomposable representation $M \in \Rep_{\FC}(Q)$ is of the form
\[
M \cong M' \otimes L
\]
for some $M' \in \Rep_{\FC}(Q)$ that corresponds to a positive root (i.e. $\udim(V') \in \Phi_+$) and some simple object $L \in \FC$.
\end{thm}
\begin{proof}
Start by fixing some admissible sink ordering on the set of vertices $V$ so that $V$ is now given by $\{1,2,..., n\}$.

We first show that $\udim$ maps indecomposables to positive extended roots.
Let $M$ be an indecomposable object in $\Rep_{\FC}(Q)$.
Let $\tau:= \sigma_{v+1} \cdots \sigma_1 c^k \in \WW(Q)$ be the shortest expression such that $\tau(\udim(M))$ is not positive, whose existence is guaranteed by Lemma \ref{lem:coxeterpowernon-positive}.
The fact that dimension vectors can not be negative together with Proposition \ref{prop:reflectiondim} imply that
\begin{equation}\label{eqn:fromsimples}
R^+_v \cdots R^+_1 (C^+)^k (M) \cong S(v+1,L)
\end{equation}
for some $L \in \Irr(\FC)$ (recall that $S(v+1,L)$ is the simple quiver representation supported at vertex $v+1$, with $L$ at that vertex).
Using Proposition \ref{prop:reflectiondim} again, we get 
\[
\sigma_v \cdots \sigma_1 c^k (\udim(M)) = [S(v+1,L)] = \alpha_{v+1} \cdot [L]
\]
and so $\udim(M) = c^{-k}\sigma_1 \cdots \sigma_v(\alpha_{v+1})\cdot [L]$ is an extended root.
The fact that $M$ is an object in $\Rep_{\FC}(Q)$ implies that $\udim(M)$ is moreover a positive element.

Now let us show that $\udim$ is injective.
Suppose $\udim(M) = \udim(M')$ for two indecomposable representations $M$ and $M'$.
Then $\tau$ as defined above is again the shortest expression where $\tau(\udim(M'))$ is negative, so the same arguments show that $R^+_v \cdots R^+_1 (C^+)^k (M') \cong S(v+1,L')$ for some $L' \in \Irr(\FC)$.
In fact, we do have $L \cong L'$, since 
\begin{align*}
\alpha_{v+1} \cdot [L] 
	&= \sigma_v \cdots \sigma_1 c^k \left( \udim(M) \right) \\
	&=\sigma_v \cdots \sigma_1 c^k \left( \udim(M') \right) \\
	&= \alpha_{v+1} \cdot [L'],
\end{align*}
so $L \cong L'$ follows from freeness of $[\FC]^V$ over $[\FC]$.
In particular, we have
\[
R^+_v \cdots R^+_1 (C^+)^k (M) \cong S(v+1,L) \cong R^+_v \cdots R^+_1 (C^+)^k (M'),
\]
and thus applying the ``inverse'' reflection functors yields $M \cong M'$.

For surjectivity of $\udim$, let us start with $m \in \Phi^{ext}_+$ an arbitrary positive (non-extended) root.
Once again by Lemma \ref{lem:coxeterpowernon-positive}, there exist some shortest expression $\tau:= \sigma_{v+1} \cdots \sigma_1 c^k$ such that $\tau(m)$ is no longer positive.
By Corollary \ref{cor:fusionrootdichotomy}, we know that $\tau(m)$ is actually negative, i.e. $\tau(m) \in -\Phi^{ext}_+$.
Since $\sigma_{v+1}$ only changes the coefficient of $\alpha_{v+1}$, we must have that $\sigma_v \cdots \sigma_1 c^k(m)$ is a $[\FC]$-multiple of $\alpha_{v+1}$.
Knowing that $\sigma_v \cdots \sigma_1 c^k(m)$ is a root guarantees that
\[
\sigma_v \cdots \sigma_1 c^k(m) = \alpha_{v+1}.
\]
Denote $C^- := R^-_1\cdots R^-_n$ and define 
\begin{equation} \label{eq:M'construction}
M':= (C^-)^k R^-_1 \cdots R^-_v(S(v+1,\one)).
\end{equation}
By \eqref{eqn:fromsimples} and Corollary \ref{cor:preservesindecomposable}, $M'$ is an indecomposable object with $\udim(M') = m$.

Now we show that $\udim$ surjects onto positive extended roots as well. Every positive extended root is of the form $m\cdot [L]$ for $m$ a positive root and $L$ a simple object in
$\FC$. 
With $M'$ as defined in \eqref{eq:M'construction} we see that the object $M' \otimes L$ is again indecomposable (since it is obtained from applying the same chain of reflection functors in \eqref{eq:M'construction} applied to $S(v+1,L)$ instead) with $\udim(M'\otimes L) = m\cdot [L]$, which shows surjectivity of $\udim$. Note
that this argument also proves the final statement of the theorem. \end{proof}

\printbibliography

@mastersthesis{Honoursthesis_Longbottom,
	author = {Longbottom, Isabel},
	title = {{D}erived {Q}uiver {R}epresentations: {R}eflections, {S}tability and {F}iltrations},
	type = {Honours Thesis},
	school = {Australian National University},
	year = {2021},
	note = {Available at {https://isabel-prime.github.io/}},}

@book {QuiverRep_Kirillov,
    AUTHOR = {Kirillov, Jr., Alexander},
     TITLE = {Quiver representations and quiver varieties},
    SERIES = {Graduate Studies in Mathematics},
    VOLUME = {174},
 PUBLISHER = {American Mathematical Society, Providence, RI},
      YEAR = {2016},
     PAGES = {xii+295},
      ISBN = {978-1-4704-2307-0},
   MRCLASS = {16G20 (14C05 14D21 16G60 16G70 17B67)},
  MRNUMBER = {3526103},
MRREVIEWER = {M\'aty\'as\ Domokos},
       DOI = {10.1090/gsm/174},
       URL = {https://doi.org/10.1090/gsm/174},
}

@article {Cox_completeenumeration,
    AUTHOR = {Coxeter, H. S. M.},
     TITLE = {The {C}omplete {E}numeration of {F}inite {G}roups of the
              {F}orm {$R_i^2=(R_iR_j)^{k_{ij}}=1$}},
   JOURNAL = {J. London Math. Soc.},
  FJOURNAL = {The Journal of the London Mathematical Society},
    VOLUME = {10},
      YEAR = {1935},
    NUMBER = {1},
     PAGES = {21--25},
      ISSN = {0024-6107,1469-7750},
   MRCLASS = {99-04},
  MRNUMBER = {1581693},
       DOI = {10.1112/jlms/s1-10.37.21},
       URL = {https://doi.org/10.1112/jlms/s1-10.37.21},
}

@book {ASS_elerep,
    AUTHOR = {Assem, Ibrahim and Simson, Daniel and Skowro\'nski, Andrzej},
     TITLE = {Elements of the representation theory of associative algebras.
              {V}ol. 1},
    SERIES = {London Mathematical Society Student Texts},
    VOLUME = {65},
      NOTE = {Techniques of representation theory},
 PUBLISHER = {Cambridge University Press, Cambridge},
      YEAR = {2006},
     PAGES = {x+458},
      ISBN = {978-0-521-58423-4; 978-0-521-58631-3; 0-521-58631-3},
   MRCLASS = {16G10 (16-02)},
  MRNUMBER = {2197389},
MRREVIEWER = {Peter\ W.\ Donovan},
       DOI = {10.1017/CBO9780511614309},
       URL = {https://doi.org/10.1017/CBO9780511614309},
}

@article {MMMT_trihedral,
    AUTHOR = {Mackaay, Marco and Mazorchuk, Volodymyr and Miemietz, Vanessa
              and Tubbenhauer, Daniel},
     TITLE = {Trihedral {S}oergel bimodules},
   JOURNAL = {Fund. Math.},
  FJOURNAL = {Fundamenta Mathematicae},
    VOLUME = {248},
      YEAR = {2020},
    NUMBER = {3},
     PAGES = {219--300},
      ISSN = {0016-2736,1730-6329},
   MRCLASS = {20C08 (17B10 18M15 18N10 20F55)},
  MRNUMBER = {4046957},
MRREVIEWER = {Iv\'an\ Ezequiel\ Angiono},
       DOI = {10.4064/fm566-3-2019},
       URL = {https://doi.org/10.4064/fm566-3-2019},
}

@article {AP95,
    AUTHOR = {Andersen, Henning Haahr and Paradowski, Jan},
     TITLE = {Fusion categories arising from semisimple {L}ie algebras},
   JOURNAL = {Comm. Math. Phys.},
  FJOURNAL = {Communications in Mathematical Physics},
    VOLUME = {169},
      YEAR = {1995},
    NUMBER = {3},
     PAGES = {563--588},
      ISSN = {0010-3616,1432-0916},
   MRCLASS = {17B37 (17B10 17B67 20G05 20G15)},
  MRNUMBER = {1328736},
MRREVIEWER = {Timo\ Neuvonen},
       URL = {http://projecteuclid.org/euclid.cmp/1104272854},
}

@article {Gabriel_finitetype,
    AUTHOR = {Gabriel, Peter},
     TITLE = {Unzerlegbare {D}arstellungen. {I}},
   JOURNAL = {Manuscripta Math.},
  FJOURNAL = {Manuscripta Mathematica},
    VOLUME = {6},
      YEAR = {1972},
     PAGES = {71--103; correction, ibid. 6 (1972), 309},
      ISSN = {0025-2611,1432-1785},
   MRCLASS = {16A64},
  MRNUMBER = {332887},
MRREVIEWER = {K.\ W.\ Roggenkamp},
       DOI = {10.1007/BF01298413},
       URL = {https://doi.org/10.1007/BF01298413},
}

@Article{Dyer_permroot,
 Author = {Dyer, Matthew J.},
 Title = {Embeddings of root systems. {II}: {Permutation} root systems.},
 FJournal = {Journal of Algebra},
 Journal = {J. Algebra},
 ISSN = {0021-8693},
 Volume = {321},
 Number = {3},
 Pages = {953--981},
 Year = {2009},
 DOI = {10.1016/j.jalgebra.2008.11.001},
 zbMATH = {5549345},
 Zbl = {1193.20047}
}

@Article{EKW_tensoralgebras,
 Author = {Etingof, Pavel and Kinser, Ryan and Walton, Chelsea},
 Title = {Tensor algebras in finite tensor categories},
 FJournal = {IMRN. International Mathematics Research Notices},
 Journal = {Int. Math. Res. Not.},
 ISSN = {1073-7928},
 Volume = {2021},
 Number = {24},
 Pages = {18529--18572},
 Year = {2021},
 DOI = {10.1093/imrn/rnz332},
 zbMATH = {7456797},
 Zbl = {1490.18008}
}

@article {Stein_finitereflect,
    AUTHOR = {Steinberg, Robert},
     TITLE = {Finite reflection groups},
   JOURNAL = {Trans. Amer. Math. Soc.},
  FJOURNAL = {Transactions of the American Mathematical Society},
    VOLUME = {91},
      YEAR = {1959},
     PAGES = {493--504},
      ISSN = {0002-9947},
   MRCLASS = {50.00 (10.00)},
  MRNUMBER = {106428},
MRREVIEWER = {H. S. M. Coxeter},
       DOI = {10.2307/1993261},
       URL = {https://doi.org/10.2307/1993261},
}

@misc{DKS_Euler,
      title={N-spherical functors and categorification of Euler's continuants}, 
      author={Tobias Dyckerhoff and Mikhail Kapranov and Vadim Schechtman},
      year={2023},
      eprint={2306.13350},
      archivePrefix={arXiv},
      primaryClass={math.CT}
}

@misc{CEnSphere,
      title={N-spherical functors and tensor categories}, 
      author={Kevin Coulembier and Pavel Etingof},
      year={2023},
      eprint={2312.03972},
      archivePrefix={arXiv},
      primaryClass={math.CT}
}

@article {EM_smallFPdim,
    AUTHOR = {Edie-Michell, Cain},
     TITLE = {Classifying fusion categories {$\otimes$}-generated by an
              object of small {F}robenius-{P}erron dimension},
   JOURNAL = {Selecta Math. (N.S.)},
  FJOURNAL = {Selecta Mathematica. New Series},
    VOLUME = {26},
      YEAR = {2020},
    NUMBER = {2},
     PAGES = {Paper No. 24, 47},
      ISSN = {1022-1824},
   MRCLASS = {18M20},
  MRNUMBER = {4076701},
MRREVIEWER = {Costel Gabriel Bontea},
       DOI = {10.1007/s00029-020-0550-3},
       URL = {https://doi.org/10.1007/s00029-020-0550-3},
}

@misc{QZ_fusionstable,
      title={Fusion-stable structures on triangulation categories}, 
      author={Yu Qiu and Xiaoting Zhang},
      year={2023},
      eprint={2310.02917},
      archivePrefix={arXiv},
      primaryClass={math.RT}
}

@book {HumphCox,
    AUTHOR = {Humphreys, James E.},
     TITLE = {Reflection groups and {C}oxeter groups},
    SERIES = {Cambridge Studies in Advanced Mathematics},
    VOLUME = {29},
 PUBLISHER = {Cambridge University Press, Cambridge},
      YEAR = {1990},
     PAGES = {xii+204},
      ISBN = {0-521-37510-X},
   MRCLASS = {20-02 (20F32 20F55 20G15 20H15)},
  MRNUMBER = {1066460},
MRREVIEWER = {Louis Solomon},
       DOI = {10.1017/CBO9780511623646},
       URL = {https://doi.org/10.1017/CBO9780511623646},
}

@article {EDihedral,
    AUTHOR = {Elias, Ben},
     TITLE = {The two-color {S}oergel calculus},
   JOURNAL = {Compos. Math.},
  FJOURNAL = {Compositio Mathematica},
    VOLUME = {152},
      YEAR = {2016},
    NUMBER = {2},
     PAGES = {327--398},
      ISSN = {0010-437X},
   MRCLASS = {20C08 (18D05)},
  MRNUMBER = {3462556},
MRREVIEWER = {Vanessa Miemietz},
       DOI = {10.1112/S0010437X15007587},
       URL = {https://doi.org/10.1112/S0010437X15007587},
}

@misc{krause2010representations,
      title={Representations of quivers via reflection functors}, 
      author={Henning Krause},
      year={2010},
      eprint={0804.1428},
      archivePrefix={arXiv},
      primaryClass={math.RT}
}

@book {EGHLSVY_introrep,
    AUTHOR = {Etingof, Pavel and Golberg, Oleg and Hensel, Sebastian and
              Liu, Tiankai and Schwendner, Alex and Vaintrob, Dmitry and
              Yudovina, Elena},
     TITLE = {Introduction to representation theory},
    SERIES = {Student Mathematical Library},
    VOLUME = {59},
      NOTE = {With historical interludes by Slava Gerovitch},
 PUBLISHER = {American Mathematical Society, Providence, RI},
      YEAR = {2011},
     PAGES = {viii+228},
      ISBN = {978-0-8218-5351-1},
   MRCLASS = {20G15 (16-01 16G10 17B10 20-01)},
  MRNUMBER = {2808160},
MRREVIEWER = {Burkhard K\"{u}lshammer},
       DOI = {10.1090/stml/059},
       URL = {https://doi.org/10.1090/stml/059},
}

@article {deodhar_root,
    AUTHOR = {Deodhar, Vinay V.},
     TITLE = {On the root system of a {C}oxeter group},
   JOURNAL = {Comm. Algebra},
  FJOURNAL = {Communications in Algebra},
    VOLUME = {10},
      YEAR = {1982},
    NUMBER = {6},
     PAGES = {611--630},
      ISSN = {0092-7872},
   MRCLASS = {20H15},
  MRNUMBER = {647210},
MRREVIEWER = {A. O. Morris},
       DOI = {10.1080/00927878208822738},
       URL = {https://doi.org/10.1080/00927878208822738},
}

@article {howlett_mmatrices,
    AUTHOR = {Howlett, Robert B.},
     TITLE = {Coxeter groups and {$M$}-matrices},
   JOURNAL = {Bull. London Math. Soc.},
  FJOURNAL = {The Bulletin of the London Mathematical Society},
    VOLUME = {14},
      YEAR = {1982},
    NUMBER = {2},
     PAGES = {137--141},
      ISSN = {0024-6093},
   MRCLASS = {20F05 (14B05 20H15)},
  MRNUMBER = {647197},
MRREVIEWER = {Wim H. Hesselink},
       DOI = {10.1112/blms/14.2.137},
       URL = {https://doi.org/10.1112/blms/14.2.137},
}

@article {AR_mckay,
    AUTHOR = {Auslander, Maurice and Reiten, Idun},
     TITLE = {Mc{K}ay quivers and extended {D}ynkin diagrams},
   JOURNAL = {Trans. Amer. Math. Soc.},
  FJOURNAL = {Transactions of the American Mathematical Society},
    VOLUME = {293},
      YEAR = {1986},
    NUMBER = {1},
     PAGES = {293--301},
      ISSN = {0002-9947},
   MRCLASS = {16A64 (20C15)},
  MRNUMBER = {814923},
MRREVIEWER = {Sheila Brenner},
       DOI = {10.2307/2000282},
       URL = {https://doi.org/10.2307/2000282},
}

@misc{EHCoxEmb,
      title={Coxeter embeddings are injective}, 
      author={Ben Elias and Edmund Heng},
      year={2024},
      eprint={2402.00974},
      archivePrefix={arXiv},
      primaryClass={math.GR}
}

@book {EGNO,
    AUTHOR = {Etingof, Pavel and Gelaki, Shlomo and Nikshych, Dmitri and
              Ostrik, Victor},
     TITLE = {Tensor categories},
    SERIES = {Mathematical Surveys and Monographs},
    VOLUME = {205},
 PUBLISHER = {American Mathematical Society, Providence, RI},
      YEAR = {2015},
     PAGES = {xvi+343},
      ISBN = {978-1-4704-2024-6},
   MRCLASS = {18D10 (16T05)},
  MRNUMBER = {3242743},
MRREVIEWER = {Julien Bichon},
       DOI = {10.1090/surv/205},
       URL = {https://doi.org/10.1090/surv/205},
}

@article {KO_qmckay,
    AUTHOR = {Kirillov, Jr., Alexander and Ostrik, Viktor},
     TITLE = {On a {$q$}-analogue of the {M}c{K}ay correspondence and the
              {ADE} classification of {$\mathfrak{sl}_2$} conformal field
              theories},
   JOURNAL = {Adv. Math.},
  FJOURNAL = {Advances in Mathematics},
    VOLUME = {171},
      YEAR = {2002},
    NUMBER = {2},
     PAGES = {183--227},
      ISSN = {0001-8708},
   MRCLASS = {17B37 (81R10 81R50 81T40)},
  MRNUMBER = {1936496},
MRREVIEWER = {Stanislav Z. Pakuliak},
       DOI = {10.1006/aima.2002.2072},
       URL = {https://doi.org/10.1006/aima.2002.2072},
}

@article {ostrik_weakhopf,
    AUTHOR = {Ostrik, Victor},
     TITLE = {Module categories, weak {H}opf algebras and modular
              invariants},
   JOURNAL = {Transform. Groups},
  FJOURNAL = {Transformation Groups},
    VOLUME = {8},
      YEAR = {2003},
    NUMBER = {2},
     PAGES = {177--206},
      ISSN = {1083-4362},
   MRCLASS = {18D10 (16D90 17B10 18E10)},
  MRNUMBER = {1976459},
MRREVIEWER = {Volodymyr V. Lyubashenko},
       DOI = {10.1007/s00031-003-0515-6},
       URL = {https://doi.org/10.1007/s00031-003-0515-6},
}

@article {lusztig_affinequiver,
    AUTHOR = {Lusztig, G.},
     TITLE = {Affine quivers and canonical bases},
   JOURNAL = {Inst. Hautes \'{E}tudes Sci. Publ. Math.},
  FJOURNAL = {Institut des Hautes \'{E}tudes Scientifiques. Publications
              Math\'{e}matiques},
    NUMBER = {76},
      YEAR = {1992},
     PAGES = {111--163},
      ISSN = {0073-8301},
   MRCLASS = {16G20 (17B67)},
  MRNUMBER = {1215594},
MRREVIEWER = {Michel Van den Bergh},
       URL = {http://www.numdam.org/item?id=PMIHES_1992__76__111_0},
}

@article {demonet_skewgroup,
    AUTHOR = {Demonet, Laurent},
     TITLE = {Skew group algebras of path algebras and preprojective
              algebras},
   JOURNAL = {J. Algebra},
  FJOURNAL = {Journal of Algebra},
    VOLUME = {323},
      YEAR = {2010},
    NUMBER = {4},
     PAGES = {1052--1059},
      ISSN = {0021-8693},
   MRCLASS = {16G10 (16G20)},
  MRNUMBER = {2578593},
MRREVIEWER = {Gonzalo Aranda Pino},
       DOI = {10.1016/j.jalgebra.2009.11.034},
       URL = {https://doi.org/10.1016/j.jalgebra.2009.11.034},
}

@article {ostrik_rank2,
    AUTHOR = {Ostrik, Viktor},
     TITLE = {Fusion categories of rank 2},
   JOURNAL = {Math. Res. Lett.},
  FJOURNAL = {Mathematical Research Letters},
    VOLUME = {10},
      YEAR = {2003},
    NUMBER = {2-3},
     PAGES = {177--183},
      ISSN = {1073-2780},
   MRCLASS = {18D10 (16W30)},
  MRNUMBER = {1981895},
MRREVIEWER = {Julien Bichon},
       DOI = {10.4310/MRL.2003.v10.n2.a5},
       URL = {https://doi.org/10.4310/MRL.2003.v10.n2.a5},
}

@Article{heng2024coxeter,
 Author = {Heng, Edmund},
 Title = {Coxeter quiver representations in fusion categories and {Gabriel}'s theorem},
 FJournal = {Selecta Mathematica. New Series},
 Journal = {Sel. Math., New Ser.},
 ISSN = {1022-1824},
 Volume = {30},
 Number = {4},
 Pages = {42},
 Note = {Id/No 67},
 Year = {2024},
 DOI = {10.1007/s00029-024-00947-1},
 zbMATH = {7882868}
}

@book {Comb_Coxeter_book,
    AUTHOR = {Bj\"{o}rner, Anders and Brenti, Francesco},
     TITLE = {Combinatorics of {C}oxeter groups},
    SERIES = {Graduate Texts in Mathematics},
    VOLUME = {231},
 PUBLISHER = {Springer, New York},
      YEAR = {2005},
     PAGES = {xiv+363},
      ISBN = {978-3540-442387; 3-540-44238-3},
   MRCLASS = {05-01 (05E15 20F55)},
  MRNUMBER = {2133266},
MRREVIEWER = {Jian-yi Shi},
}

@article {HaziJW,
    AUTHOR = {Hazi, Amit},
     TITLE = {Existence and rotatability of the two-colored {J}ones-{W}enzl
              projector},
   JOURNAL = {Bull. Lond. Math. Soc.},
  FJOURNAL = {Bulletin of the London Mathematical Society},
    VOLUME = {56},
      YEAR = {2024},
    NUMBER = {3},
     PAGES = {1095--1113},
      ISSN = {0024-6093,1469-2120},
   MRCLASS = {20G42 (11R18 20F55)},
  MRNUMBER = {4735608},
MRREVIEWER = {Jean\ Michel},
       DOI = {10.1112/blms.12983},
       URL = {https://doi.org/10.1112/blms.12983},
}

@article{BMPS,
	author = {Bigelow, Stephen and Peters, Emily and Morrison, Scott and Snyder, Noah},
	coden = {ACMAA8},
	doi = {10.1007/s11511-012-0081-7},
	fjournal = {Acta Mathematica},
	issn = {0001-5962},
	journal = {Acta Math.},
	mrclass = {46L37},
	mrnumber = {2979509},
	number = {1},
	pages = {29--82},
	title = {Constructing the extended {H}aagerup planar algebra},
	url = {http://dx.doi.org/10.1007/s11511-012-0081-7},
	volume = {209},
	year = {2012}}

@article{ENO,
	author = {Etingof, Pavel and Nikshych, Dmitri and Ostrik, Viktor},
	journal = {Annals of Mathematics},
	pages = {581--642},
	publisher = {JSTOR},
	title = {On fusion categories},
	year = {2005}}

\end{document}